\documentclass{amsart}
\usepackage{tikz-cd, hyperref, cleveref}
\usepackage[maxbibnames=9]{biblatex}
\usepackage{latexsym,mathrsfs} 
\usepackage{a4wide}
\usepackage[normalem]{ulem}

\usepackage{xcolor}
\usepackage[colorinlistoftodos]{todonotes}
\usepackage{bbm}
\usepackage{fullpage}

\usepackage{amssymb, mathtools, bm}
\usepackage{amsthm}
\usepackage{aliascnt}
\usepackage{hyperref}
\usepackage{cleveref}

\newtheorem{lemma}{Lemma}[section]

\newaliascnt{theorem}{lemma}
\newtheorem{theorem}[theorem]{Theorem}
\aliascntresetthe{theorem}

\newaliascnt{claim}{lemma}

\aliascntresetthe{claim}

\newaliascnt{corollary}{lemma}
\newtheorem{corollary}[corollary]{Corollary}
\aliascntresetthe{corollary}

\newaliascnt{conjecture}{lemma}

\aliascntresetthe{conjecture}

\newaliascnt{proposition}{lemma}
\newtheorem{proposition}[proposition]{Proposition}
\aliascntresetthe{proposition}

\theoremstyle{remark}

\newaliascnt{remark}{lemma}
\newtheorem{remark}[remark]{Remark}
\aliascntresetthe{remark}

\theoremstyle{definition}

\newaliascnt{definition}{lemma}
\newtheorem{definition}[definition]{Definition}
\aliascntresetthe{definition}

\newaliascnt{question}{lemma}

\aliascntresetthe{question}

\newaliascnt{questions}{lemma}

\aliascntresetthe{questions}

\newaliascnt{example}{lemma}
\newtheorem{example}[example]{Example}
\aliascntresetthe{example}

\numberwithin{equation}{section}

\newcommand{\bsigma}{\boldsymbol{\sigma}}

\newcommand{\tr}[1]{\vphantom{#1}^t #1}

\DeclareMathOperator{\rank}{rank}

\makeatletter
\@for\cs:={A,B,C,D,E,F,G,H,I,J,K,L,M,N,O,P,Q,R,S,T,U,V,W,X,Y,Z}\do{
  \expandafter\edef\csname \cs\endcsname{\noexpand\mathbb{\cs}}
  \expandafter\edef\csname c\cs\endcsname{\noexpand\mathcal{\cs}}
  \expandafter\edef\csname c\cs\endcsname{\noexpand\mathcal{\cs}}
}
\makeatother 
\newcommand{\btau}{\boldsymbol\tau}
\newcommand{\bt}{\boldsymbol\tau}
\newcommand{\brho}{\boldsymbol\rho}
\newcommand{\brhobar}{\boldsymbol{\bar{\rho}}}
\newcommand{\bs}{\boldsymbol\sigma}

\newcommand{\Add}{\mathrm{Add}}
\newcommand{\First}{\mathsf{f}}
\newcommand{\prefixes}{\mathrm{pref}}
\newcommand{\diam}{\mathrm{diam}}
\newcommand{\gt}{>}
\newcommand{\lt}{<}
\newcommand{\boldone}{\boldsymbol{1}}
\newcommand{\vertiii}[1]{{\left\vert\kern-0.25ex\left\vert\kern-0.25ex\left\vert #1 
    \right\vert\kern-0.25ex\right\vert\kern-0.25ex\right\vert}}

\title[Eigenvalues from S-adic representations]{Continuous eigenvalues of minimal subshifts via S-adic representations and coboundaries}

\author[V.~Berth\'e]{Val\'erie Berth\'e}
\address{Universit\'e de Paris, IRIF, CNRS, F-75013 Paris, France}
\email{berthe@irif.fr}
\thanks{Val\'erie Berth\'e was supported by the ERC grant DynAMiCs (101167561) of the European Research Council,  by the bilateral grant SYMDYNAR (ANR-23-CE40-0024 and FWF~I~6750) of the Agence Nationale de la Recherche and the Austrian Science Fund, and by the ANR project IZES (ANR-22-CE40-0011).
}
\author[P.~Cecchi Bernales]{Paulina Cecchi Bernales}
\address{Universidad de Chile, Departamento de Matem\'aticas, Santiago, Chile.}
\email{pcecchi@uchile.cl}
\thanks{Paulina Cecchi Bernales was Supported by ANID/Fondecyt Iniciación
11240886, Vicerrectoría de Investigación y Desarrollo (VID)-Universidad de Chile/U-Inicia Project UI-003/23 and ANID/STIC Amsud AMSUD230049.
}
\author[B.~Espinoza]{Basti\'an Espinoza}
\address{Universit\'e de Li\`ege, Département de Mathématique, Li\`ege, Belgium} 
\email{baespinoza@uliege.be}
\thanks{Bastián Espinoza is a postdoctoral researcher supported by the Fonds de la Recherche Scientifique -- FNRS}

\date{\today}

\keywords{symbolic dynamics; topological dynamics; substitutions; $S$-adic shifts; spectral theory; eigenvalues; coboundaries; extension graphs}
\subjclass[2010]{37B10, 05A05, 37A30}

\usepackage{enumitem}
\begin{document}

\begin{abstract} 
We provide characterizations of continuous eigenvalues for minimal symbolic dynamical systems. These characterizations rely on a description in terms  of   $S$-adic structures (given as    infinite compositions of morphisms)  satisfying natural mild conditions, such as recognizability and primitiveness.  
Under the additional assumptions of finite alphabet rank or decisiveness of the directive sequence, these characterizations  involve  sequences of local letter-coboundaries. 
We emphasize the role of combinatorics in the study of continuous eigenvalues through the interplay between letter-oboundaries and extension graphs, and we give several types of sufficient conditions for having trivial coboundaries only.  These results  are applied  among others  to linear involutions and  to the Thue--Morse  system in the rational  base $3/2.$
We also  illustrate the versatility of  the notion of coboundaries  in the context of bounded  symbolic  discrepancy.
In particular we recover a simple characterization of letter-balance for primitive substitutive subshifts.
Finally, we refine known descriptions of the possible values of eigenvalues in terms of the measures of bases of Kakutani--Rohklin towers provided by the $S$-adic representation.
\end{abstract}

\maketitle

\section{Introduction}

The set of continuous eigenvalues associated to a dynamical system is an important conjugacy invariant that has been studied extensively in topological dynamics for several decades; see \emph{e.g.} \cite{ellis}. 
Beyond its intrinsic interest, it provides valuable information, for instance, on the possible topological factors of the system and on their independence (in the sense of joinings) from other dynamical systems \cite{Glasner2003}. In the symbolic setting, spectral properties, and in particular continuous eigenvalues, are well understood for substitutive subshifts (see \cite{Dekking1978}, \cite{Host86} and the books \cite{Que,Pytheas}). 
More generally, there is a substantial literature devoted to the spectral theory of minimal Cantor systems \cite{eigenvalues_LR_2003,Bressaud-Durand-Maass:2005,CORTEZ_DURAND_PETITE_2016,ADE:23,DFM19} and tiling spaces \cite{Solomyak:07,Fusion:14}.

In this article, we study continuous eigenvalues of minimal symbolic systems equipped with an $S$-adic structure. Recall that an $S$-adic structure on a subshift is given by a sequence of morphisms\footnote{Morphisms are called substitutions when the images of letters belong to the same alphabet.}
$\tau_n:\mathcal A_{n+1}^*\to\mathcal A_n^*,$
called a \emph{directive sequence}, whose iterated composition   generates the language of the subshift. 
Such representations naturally encode hierarchical structures and are closely related to Kakutani--Rokhlin partitions and Bratteli--Vershik systems \cite{GPS92}. 
 More precisely, every properly ordered Bratteli--Vershik representation gives rise to an $S$-adic expansion~\cite{DDMP}. The converse, however, does not hold: not every $S$-adic expansion arises from a properly ordered Bratteli--Vershik representation.
The prototypical examples of $S$-adic structure are those for substitutive subshifts
(for which the directive sequence is constant), and for Sturmian subshifts (for which the $S$-adic structure encodes the continued fraction of the irrational number defining the system).
Over the last decade, $S$-adic systems have become a central object in symbolic dynamics of low complexity \cite{Sadic_conj,DDMP,BST:23} and in tiling theory \cite{Fusion:14,BD14}, providing a flexible  framework  for representing a wide variety of dynamical systems.

The starting point of the present work is the classical theory developed by Host \cite{Host86} for  a primitive substitution  $\tau$ (inspired by  earlier work of Dekking \cite{Dekking1978} on the {\em constant-length} case, \emph{i.e.}, when the images by the substitution  $\tau$ of  all letters have the same length). 
His characterization of continuous eigenvalues is based on the asymptotic behavior modulo~$1$ of the quantities
$\alpha |\tau^n(a)|$, 
where $\alpha$ is a candidate (additive) eigenvalue and $|\tau^n(a)|$ denotes the length of the image of the letter  $a$ under the $n$-th iterate of the substitution $\tau$. 
A  crucial insight of Host's work relies in the understanding of the role played by  the asymptotic behavior of those quantities,   together with the fact   that  continuous eigenvalues are  governed by  \emph{letter-coboundaries}. Coboundaries are classcially   considered in topological dynamics (see \Cref{rem:coboundary}). Here, the coboundaries under consideration   have   radius-$0$, they  depend only on individual letters,  and  are based on the  combinatorics of two-letter words. Note that Host's original motivation was to prove that measurable  and topological  eigenvalues coincide for substitutions.

Several extensions of Host's ideas  have since appeared, notably for substitution tilings, interval exchange transformations, and more generally Cantor dynamical systems described by Bratteli--Vershik representations. 
In particular, a series of works \cite{eigenvalues_LR_2003,Bressaud-Durand-Maass:2005,CORTEZ_DURAND_PETITE_2016,ADE:23,DFM19} provides complete characterizations of continuous eigenvalues for minimal Cantor systems in terms of properly ordered Bratteli--Vershik diagrams. 

While every minimal Cantor system (hence every minimal subshift) admits a proper representation, obtaining one that is explicit or adapted to a concrete class of systems is often a difficult task.
An important example is provided by linear involutions. 
These systems admit Bratteli--Vershik representations through their realization as interval exchange transformations \cite{Gjerde-Johansen-IET}, but this description is poorly suited to their specific dynamics:
indeed, linear involutions form a meager subset of interval exchanges, so since many of the available renormalization techniques only describe generic interval exchanges, they fail to capture the phenomena specific to linear involutions. 
On the other hand, linear involutions possess a natural hierarchical structure arising   for instance from a Rauzy induction, which is most naturally expressed through an $S$-adic expansion
(but it cannot be expressed by a properly ordered Bratteli--Vershik diagram),
for which we have much better understanding \cite{rodolfo_quadratic,Boissy-Lanneau,Avila-Resende}. 
See also  \Cref{sec:specular} where we exhibit  $S$-adic expansions for linear involutions based on return words, and the recent article \cite{Arbulu}, where our criteria are used to analyze continuous eigenvalues of linear involutions.
Similar situations occur for many other classes of symbolic and geometric dynamical systems, where the available hierarchical representation is $S$-adic rather than  provided by a properly ordered Bratteli--Vershik system. In  fact, renormalization  for interval exchanges and their generalizations, such as described in the survey \cite{Skrip:2023}, is a rich source of  $S$-adic representations. Let us quote    \cite{Bruin} in the setting of  interval translation maps,   and  \cite{AHS:26};   see  also \cite{Solomyak:25} and its appendix  for  weak mixing properties for random $S$-adic  systems inspired by interval  exchanges,  or else   see \cite{CP:23}.  

This motivates the recent development of criteria formulated directly in terms of $S$-adic representations \cite{BCBY,BMY,mercat}. 
These works aim at extending the Bratteli--Vershik approach by removing the properly ordered condition;
at the $S$-adic level, this amounts to removing the assumption that the directive sequence is  \emph{proper} (\emph{i.e.}, each morphism assigns the same initial and same final letter to the image of every letter; see \Cref{susbsec:morphisms} for precise definitions).
Unlike the proper case, non-proper directive sequences naturally give rise to non-trivial letter-coboundaries (in Host's sense \cite{Host86}), making the analysis substantially more delicate. 
The existing results cover several important situations, but they typically require additional assumptions (such as \emph{finitariness} of the directive sequence,  which means that the directive sequence consists of finitely many morphisms), and provide only sufficient or only necessary conditions for the existence of eigenvalues.

\subsubsection*{Main results}
The main objective of the present paper is to provide complete characterizations of continuous eigenvalues directly from a given $S$-adic representation under natural hypotheses, thereby extending both the Bratteli--Vershik theory, and the existing $S$-adic results from \cite{BCBY,BMY,mercat}.

Our first main result (\Cref{theo:EigCharac:Mult}) applies to every primitive recognizable directive sequence and characterizes continuous eigenvalues through quantitative convergence conditions involving  the  heights of the   Kakutani--Rohklin towers   provided the $S$-adic expansion (see Section \ref{subsec:KR}),
together with  sequences of real numbers  $(\rho_n(a))_n$, 
whose role  is to compensate  variations  in the heights of towers; they can  be compared  \emph{e.g.} with similar  quantities  involved in the  criterion \cite[Theorem 17]{DFM19}  for  measurable eigenvalues.

We then introduce a new combinatorial notion, called \emph{decisiveness} (see Definition~\ref{def:decisive}), inspired by ideas for Bratteli--Vershik systems \cite{karpel_downa}, and prove that under this additional assumption these convergence conditions acquire more structure, in terms of letter-coboundaries at each level of the hierarchical decomposition.  
 (Together with recognizability, decisiveness captures the property that the sequence
 of  Kakutani--Rohklin partitions provided by the $S$-adic structure generates the topology; see \Cref{prop:decisive:dyn_inter}.)
This considerably strengthens the theory, as the coboundary formulation naturally admits an abelian encoding, allowing the data to be transported by the incidence matrices of the directive sequence and making tools from linear algebra and the theory of Lyapunov exponents available.

It is important to note that decisiveness is necessary for such a coboundary description, as shown by~\Cref{thm:NoCobsInftyRank}. 
However, this condition is remarkably mild: we prove in~\Cref{prop:decisive:suff_conds} that every primitive recognizable directive sequence can be transformed into a decisive one by a simple sliding-block construction. 
In fact, most of the examples in this article already satisfy this condition.
Consequently, our results apply to a broad class of $S$-adic structures for symbolic systems, allowing to handle  in particular  the  \emph{infinite alphabet rank} case (\emph{i.e.,}
 the directive sequence consists of  infinitely many morphisms, defined on  infinitely many alphabets).  
 In particular, this allows us to go beyond the proper case   which yields  trivial coboundaries (see \Cref{theo:EigCharac:proper}).  
In particular,  we prove that 
for any minimal substitutive subshift with an irrational eigenvalue, there exists a substitution generating a subshift conjugate to the original one in which all letter-coboundaries associated to that irrational eigenvalue are non-trivial (see \Cref{prop:exs_nontrival_cobs}).

Letter-coboundaries, beyond their role in spectral theory, also emerge as a versatile combinatorial tool in their own right. 
Indeed we  relate letter-coboundaries to extension graphs, balancedness, symbolic discrepancy, and return words. 
In particular, we show that the space of letter-coboundaries is determined by the connected components of the extension graph of the empty word of the subshift (\Cref{thm:manifold}), we obtain in \Cref{prop:balance_char_Sadic} a characterization of bounded symbolic discrepancy functions in terms of suitable coboundaries along an $S$-adic representation,  we
deduce in~\Cref{balanced=>spaces_decomposition}
an alternative characterization to that of Adamczewski \cite{adam03,adam05}
of balanced substitutive systems, 
and  we derive in \Cref{sec:studyofcoboundaries:conditions_for_trivial_only} simple sufficient conditions ensuring that all letter-coboundaries are trivial. 
These results illustrate that the notion of  a letter-coboundary captures structural features of symbolic systems extending beyond the study of eigenvalues.

Finally, although the present article is devoted exclusively to \emph{continuous} eigenvalues, we expect many of the techniques introduced here to extend to the measurable setting. 
Existing criteria for measurable eigenvalues based on Bratteli--Vershik representations already have a form closely related to our first characterization theorem (see \emph{e.g;}~\cite[Theorem 17]{DFM19}), suggesting that a suitable notion of letter-coboundary should lead to analogous refinements in the measurable setting.
 
\subsubsection*{Overall  strategy}

We now describe the main ideas underlying our approach. 
Let $X$ be a minimal subshift generated by a directive sequence $\btau=(\tau_n)_n$. 
As a first step, we characterize continuous eigenvalues under the sole assumptions that $\btau$ is primitive and recognizable. 
The characterization given in~\Cref{theo:EigCharac:Mult} is formulated in terms of summability conditions involving the quantities
\[	h_n(a)=\left|\tau_0\tau_1\cdots\tau_{n-1}(a)\right|,	\]
which are the sums of the entries of the columns of the corresponding products of incidence matrices. 
Equivalently, $h_n(a)$ is the height of the tower indexed by the letter $a$ in the associated Kakutani--Rokhlin partition. 
The criterion consists in finding a sequence $\brho=(\rho_n \in \R^{\cA_n})_n$ of real vectors
such that we need to control the distances to the nearest integer of expressions of the form
\[	\rho_n(u_0)-\rho_n(u_k)+\alpha\bigl(h_n(u_0)+\cdots+h_n(u_{k-1})\bigr),\]
where $u_0\cdots u_k$ belongs to the language of the level-$n$ subshift and $\alpha$ is a candidate additive eigenvalue. 
Thus, the problem amounts to understanding the convergence modulo~$1$ of the quantities $\alpha h_n(a)$, uniformly over the letters $a$, after correcting them by the terms $\rho_n(a)$. 
The role of $\rho_n$ is to compensate for possible oscillations that prevent the sequence $\alpha h_n(a)$ itself from converging modulo~$1$.

The proof is based on the structure in terms of towers provided by the recognizability of the directive sequence (recognizability guarantees  that the towers form a partition; see Section \ref{sec:preliminaries} for the corresponding definitions).
A continuous eigenfunction is first approximated on the bases of the towers, 
and the values of these approximations define the vectors $\rho_n$. 
The compatibility of the eigenfunction with the shift then translates into the convergence conditions above. Conversely, given a sequence $\brho$ satisfying the appropriate summability assumptions, one constructs a sequence of locally constant functions whose successive differences are summable. 
Their uniform limit is a continuous eigenfunction. 
In this way, the characterization isolates precisely the information needed to control the dynamics on the bases of the towers without assuming that the directive sequence is proper. Overall, our approach builds on techniques developed in \cite{Host86,DFM19,Solomyak:07,eigenvalues_LR_2003}, while introducing the additional arguments needed to account for the oscillation between towers.

As a second step, we show that, under either decisiveness or finite alphabet rank, the sequence $\brho=(\rho_n)_n$ can be chosen with an additional morphic compatibility. This compatibility turns the functions $\rho_n$ into a sequence of local letter-coboundaries; see~\Cref{def: coboundary}. The resulting characterizations are given in Theorems~\ref{theo:EigCharac:FAR:Mult} and~\ref{theo:EigCharac:decisive:positive}, the latter being the coboundary counterpart of~\Cref{theo:EigCharac:positive}.

This is the point at which the algebraic role of coboundaries becomes fully visible. 
The compatibility between consecutive levels allows the sequence $(\rho_n)_n$ to be encoded as letter-coboundaries, which permits one to transport data by the incidence matrices of the directive sequence. 
Consequently, questions about continuous eigenvalues can be studied using tools from linear algebra, including invariant subspaces, duality, and arguments involving Lyapunov exponents. 
This framework extends the one developed for finitary directive sequences (\emph{i.e.,} sequences containing only finitely many different substitutions) found in \cite{mercat,BCBY}, 
thus allowing us to treat systems of infinite topological rank, provided the directive sequence is decisive. 
The latter situation is illustrated in Section~\ref{sec:TM} by the subshift generated by the Thue--Morse word in base $3/2$,
which cannot be generated by a finitary, primitive directive sequence.

For proper directive sequences (see \Cref{susbsec:morphisms} for the definition), the coboundary contribution becomes trivial.
More precisely, when every morphism of a directive sequence generating a minimal $S$-adic subshift is proper, the associated letter-coboundaries may be taken to be trivial in our criteria; see \Cref{theo:EigCharac:proper}. 
Our characterizations then reduce to the previously known criteria obtained from properly ordered Bratteli--Vershik representations, such as \cite[Theorem~2]{DFM19}. Thus, the coboundary formalism does not replace the proper theory, but extends it by identifying and controlling the additional combinatorial obstruction that appears in the non-proper setting.

A central tool in the study of these coboundaries is the extension graph, which records the possible left and right extensions of words in the language; see Section~\ref{subsec:extensiongraphs}. We establish a precise relation between the topology of the extension graph of the empty word and the space of letter-coboundaries. In particular, if this extension graph has $r$ connected components, then the real vector space of letter-coboundaries has dimension $r-1$; see Theorem~\ref{thm:manifold}. This result makes explicit the extent to which the combinatorics of the language governs the possible coboundary corrections, and therefore the structure of continuous eigenvalues.

We exploit this relation throughout \Cref{sec:studyofcoboundaries}. In \Cref{prop:dimeigenvalues}, we obtain an upper bound for the rational dimension of the group of continuous eigenvalues in terms of the behaviour of extension graphs and factor complexity. In \Cref{thm:tijdeman}, we give a proof, in the language of extension graphs, of a theorem of Tijdeman \cite{tijdeman} providing a lower bound for the factor complexity of a transitive subshift in terms of the rational dimension of the vector space generated by its letter frequencies. As a consequence, we recover a result of Andrieu and Cassaigne \cite{andrieu-notes} concerning dendric subshifts; see \Cref{def:dendric}. Our argument follows the same general strategy as theirs, but replaces the flow-matrix formalism by the combinatorics of extension graphs.


The final sections illustrate the scope and the sharpness of the theory. In \Cref{sec:specular}, we study rational eigenvalues of specular subshifts and linear involutions, where non-trivial coboundaries arise naturally from the underlying non-proper $S$-adic representation. Section~\ref{sec:TM}  handles 
the Thue--Morse word in base $3/2$, whose natural directive sequence has infinite alphabet rank. Further examples are developed in \Cref{sec:examples}. In particular, the construction of \Cref{subsec:ex:nocoboundary} shows that the additional assumptions in the coboundary characterization cannot simply be omitted: for a directive sequence that is neither decisive, nor of finite alphabet rank, the existence of a continuous eigenvalue need not imply the existence of a sequence of letter-coboundaries satisfying the convergence condition of \Cref{theo:EigCharac:FAR:Mult}. We also consider constant-length directive sequences of finite alphabet rank in \Cref{subsec:constantlength}.
Using Theorem~\ref{theo:EigCharac:Mult}, we recover the fact that the corresponding systems have only rational additive eigenvalues \cite{BMY}.

\subsubsection*{Organization}
Section~\ref{sec:preliminaries} recalls the necessary background on symbolic dynamical systems, continuous eigenvalues, and directive sequences. In Section~\ref{sec:coboundaries}, we introduce letter-coboundaries, symbolic discrepancy, extension graphs, and decisiveness. Section~\ref{sec:carac} contains the general characterizations of continuous eigenvalues for primitive recognizable directive sequences; the overall strategy is explained in \Cref{subsec:strategy}. In Section~\ref{sec:carac_finite_rank}, these characterizations are refined under the additional assumptions of decisiveness or finite alphabet rank, and reformulated in terms of  letter-coboundaries. Section~\ref{sec:studyofcoboundaries} develops the relations among coboundaries, extension graphs, factor complexity, frequencies, discrepancy, and balancedness. In particular, we prove that, for a proper, unimodular, and primitive directive sequence on $d$ letters, the abelianizations of the return words generate $\Z^d$; see \Cref{thm:proper+unimod=>RW_gen_Zd}. We also collect in \Cref{cor:EigCharac:FAR:Mult:cocoboundary} several practical sufficient and necessary conditions for a real number to be an additive eigenvalue. The applications to specular subshifts, linear involutions, and to the Thue--Morse subshift in base $3/2$ are presented in Sections~\ref{sec:specular} and~\ref{sec:TM}, respectively, while additional examples are gathered in Section~\ref{sec:examples}.

\bigskip

\paragraph{{\bf Acknowledgements}}  
We would like to thank  Clemens M\"ullner  for  his very  valuable   insight  which inspired  \Cref{thm:manifold}, as well as  Samuel Petite and Maria Clara Werneck for bringing  \cite{tijdeman} to our attention. 
We also would like to thank Paul Mercat who pointed out a mistake in Proposition \ref{prop:decisive:suff_conds} in an earlier version of the paper.

\section{Preliminaries}\label{sec:preliminaries}

\subsection{Topological dynamical systems}\label{subsec:topologicalds}

A \emph{topological dynamical system} (or simply a \emph{system}) is a pair $(X,T)$ with $X$ a compact metric space and $T\colon X\to X$ a homeomorphism. 
The system is \emph{minimal} if it has no non-empty proper closed $T$-invariant subset; equivalently, if every $T$-orbit $\{T^n x : n\in\Z\}$ is dense in $X$. 
It is \emph{transitive} if it there is a dense $T$-orbit, and \emph{aperiodic} if it has no periodic points, {\em i.e.}, $T^n x\neq x$ for all $x\in X$ and $n\ge 1$. 
Two systems $(X_1,T_1)$ and $(X_2,T_2)$ are \emph{conjugate} if there exists a homeomorphism $\varphi\colon X_1\to X_2$ such that $\varphi\circ T_1 = T_2\circ \varphi$.

A real number $\alpha$ is an {\em additive continuous eigenvalue} of $(X,T)$ if there exists a continuous function $g \colon X  \rightarrow \R / \Z$ such that $g(Tx) = g(x) + \alpha \pmod{\Z}$ for every $x \in X$. 
In this article, this is the only type of eigenvalues we consider, so we usually refer to them simply as {\em eigenvalues} of $(X,T)$.
We denote by $E(X,T)$ the  set of eigenvalues of $(X,T)$; it is known to be  an additive subgroup of $(\R,+)$ , containing $\Z$, that is  moreover invariant by conjugacy.

A Borel probability measure $\mu$ on $X$ is said to be {\em $T$-invariant} if $\mu(T^{-1} B) = \mu(B)$ for every Borel subset $B$ of $X$.
We denote by $\cM(X,T)$ the set of all $T$-invariant Borel probability measures on $X$. 
A measure $\mu\in\mathcal{M}(X,T)$ is {\em ergodic} if the only Borel subsets $B$ satisfying $T(B) = B$ are those such that $\mu(B)=0$ or $\mu(B)=1$. 
It is well known that any topological dynamical system admits an ergodic invariant measure~\cite{KrBo}. 
A topological dynamical system $(X,T)$ is {\em uniquely ergodic} if there exists a unique $T$-invariant probability measure on $X$.

Let $C(X,\Z)$ (resp.  $C(X,\R)$) stand for the set of continuous  functions  
$f \colon X  \rightarrow  \Z$
(resp. $\R$). The {\em image subgroup} of the system $(X,T)$ is  the additive group defined by
\[  I(X,T)=\bigcap_{\mu\in\mathcal{M}(X,T)}\biggl\lbrace \int f d\mu \mid f\in C(X,\mathbb{Z})\biggr\rbrace. \]
It is known that $E(X,T)$ is a subgroup of $I(X,T)$ by~\cite{itzaortiz,CORTEZ_DURAND_PETITE_2016,ghh18}.
See also Remark \ref{rem:measures} below.

\subsection{Subshifts}\label{subsec:subshifts}

A finite set $\cA$ with at least two elements is called the \emph{alphabet}.
We write $\cA^*$ for the set of all {\em finite words} over $\cA$ (including the empty word $\varepsilon$), with concatenation as the operation; this is the free monoid on $\cA$, and its identity element is $\varepsilon$.
Elements $x=(x_n)_{n\in\Z}\in\cA^\Z$ are viewed as bi-infinite words over $\cA$, and we write $x = \dots x_{-2} x_{-1} \boldsymbol{.} x_0 x_1 \dots$, where the dot marks coordinate $0$.
For integers $i \leq j$, define 
$x_{[i,j]} = x_i x_{i+1}\cdots x_{j}$,
and $x_{[i,j)} = x_i x_{i+1}\cdots x_{j-1}$, with the convention $x_{[i,i+1)} = \varepsilon$.

For a finite word $w = w_0\cdots w_{\ell-1}\in\cA^\ell$, its \emph{length} is $|w| = \ell$, and $|w|_a$, for a letter $a$, stands for the number of occurrences of  the letter $a$ in $w$. The notation $\First(w)$  stands  for the first letter of a non-empty  word $w$.
A word $u$ is a \emph{factor} of $w$ if $w=pus$ for some (possibly empty) $p,s\in\cA^*$.
If $p=\varepsilon$ (resp.\ $s=\varepsilon$), then $u$ is a \emph{prefix} (resp.\ \emph{suffix}) of $w$.
The set of prefixes of  a word $w$ is denoted as $\prefixes(w)$.
An index $j$ with $0 \le j \le |w|-|u|$ such that $w_j w_{j+1} \dots w_{j+|u|-1} = u$ is an \emph{occurrence} of $u$ in $w$.
The same terminology applies to elements of $\cA^\Z$.

The set $\cA^\Z$ endowed with the prodiscrete topology is called the \emph{full shift}, and is homeomorphic to the Cantor set.
The \emph{shift map} $S \colon \cA^{\mathbb Z}\to\cA^{\mathbb Z}$, defined by $S((x_n)_{n\in\mathbb Z}) = (x_{n+1})_{n\in\mathbb Z}$, is a homeomorphism, so $(\cA^\Z,S)$ is a topological dynamical system.
A \emph{subshift} (also called \emph{shift}) is a closed subset $X \subseteq \cA^\Z$ such that $S(X) = X$.
Thus, $(X, S|_X$), where $S|_X \colon X \to X$ is the restriction of $S$ to $X$, is a system as well.
By convention we often identify the pair $(X,S|_X)$ with $X$

For $x\in X$, let $\cL(x)$ be the set of all factors of $x$.
The \emph{language} of a subshift $X$ is $\cL(X) = \bigcup_{x\in X}\cL(x)$.
It is classical that $X$ is minimal if and only if $\cL(x) = \cL(x')$ for all $x,x'\in X$.
Let $\cL_n(X)$ be  the set of  factors of $X$ of length $n$.
The \emph{factor complexity} of $X$ is the map $p_X \colon \N \to \N$ defined by $$p_X(n) = |\cL_n(X)|.$$
The {\em cylinder} associated to $u \in \cL(X)$ is the set $$[u] = \{x \in X : x_{[0,|u|)} = u\}.$$

\subsection{Morphisms and directive sequences}
\label{susbsec:morphisms}

Let $\tau \colon \cA^{\ast} \to \cB^{\ast}$ be a monoid morphism, {\em i.e.}, $\tau(a_1 a_2 \dots a_\ell) = \tau(a_1) \tau(a_2) \dots \tau(a_\ell)$ for all $a_1,a_2,\dots,a_\ell \in \cA$ and $\ell \geq 1$.
We say that $\tau$ is {\em non-erasing} if the image of any letter is a non-empty word, and is {\em letter-onto} if for each $b \in \cB$ there is $a \in \cA$ for which $b$ occurs in $\tau(a)$. 
A morphism that is both non-erasing and letter-onto is called a {\em substitution}.
If moreover $\cA = \cB$, $\tau$ is a substitution {\em on} $\cA$, and $\tau^n$ stands for its $n$-fold self-composition $\tau \circ \tau \circ \dots \circ \tau$.

The {\em incidence matrix} of a substitution $\tau \colon \cA^* \to \cB^*$ is the $\cB \times \cA$ matrix $M_\tau$ whose entry at a position $(b,a)$ is the number of times that $b$ occurs in $\tau(a)$.
We say that $\tau$ is {\it left-proper} (resp. {\it right-proper}) if there exists a single letter $b\in\cB$ such that $\tau(a)$ starts with $b$ (resp. ends with $b$) for every $a\in\cA$; it is {\it proper} if it is both left- and right-proper. 

A substitution $\tau\colon \cA^{\ast}\to \cB^{\ast}$ can be extended by concatenation to $\cA^\N$ and $\cA^\Z$. 
For elements $x$ of $\cA^\Z$, we apply $\tau$ to the negative and positive parts of $x$ and then concatenate them at position zero; explicitly,
\[  \tau(x) = \tau(\dots x_{-2} x_{-1} \boldsymbol{.} x_0 x_1 \dots) \coloneqq 
    \dots \tau(x_{-2}) \tau(x_{-1}) \boldsymbol{.} \tau(x_0) \tau(x_1) \dots. \]

An $S$-adic representation of a subshift comes from iterating substitutions; we outline the basics here and refer to \cite[Sec.~7.4]{DP} and \cite{BSTY,BD14} for details.
A {\em directive sequence} $\bt = (\tau_n)_{n\geq 0}$ consists of substitutions $\tau_n \colon \cA_{n+1}^{\ast}\to \cA_n^{\ast}$ so that consecutive domains and codomains match.
This allows us to define the compositions
\begin{equation} \label{eq:taunl}
    \tau_{n,m} \coloneqq \tau_n\circ\tau_{n+1}\circ\cdots\circ \tau_{m-1}
    \enspace \text{for every $m \gt n \geq 0$.}
\end{equation}
We say that $\bt$ is {\it everywhere growing} if $\min_{a\in \cA_n}|\tau_{0,n}(a)|$ tends to $\infty$ as $n\to\infty$; 
$\bt$ is {\it primitive} if for every $n \geq 0$ there exists $m > n$ with $\tau_{n,m}$ positive.
Given a subsequence $(n_k)_{k\geq 0}$, the {\em contraction of $\bt$ along $(n_k)_{k\geq 0}$} is the directive sequence $(\tau_{0,n_0}, \tau_{n_0,n_1}, \tau_{n_1,n_2}, \dots)$, where $\tau_{0,n_0}$ is the identity map if $n_0 = 0$.
If for all $n \geq 0$ the incidence matrix of $\tau_n$ has strictly positive entries, $\bt$ is called {\em positive}.
Hence, a positive directive sequence is always primitive. 

We now introduce a set of conditions that are used repeatedly in different contexts throughout the paper.
\begin{definition}\label{def:far}
The directive sequence $\btau$
has {\em finite alphabet rank} if $(|\cA_{n_k}| : k \geq 0)$ is bounded for some subsequence $n_0 < n_1 < \dots$.  
It is said  {\em left proper} (resp. {\em right proper, proper}) if every morphism $\tau_n$ is left proper (resp. right proper, proper). 
Finally, it is said to be {\em unimodular} if each incidence matrix $M_{\tau_n}$ has determinant of modulus 1.
\end{definition}

For $n\geq 0$, the {\em level-$n$ language} $\cL_{{\bt}}^{(n)}$ generated by ${\bt}$ is the set of all words $w \in \cA_n^\ast$ for which there exist $m > n$ and $a \in \cA_m$ such that $w$ occurs in $\tau_{n,m}(a)$.
For each $n\geq 0$, we let $X_{{\bt}}^{(n)}$ be the set of infinite words $x\in\cA_n^\Z$ whose factors belong to $\cL_{{\bt}}^{(n)}$. 
In general $\cL(X_{\bt}^{(n)})$ is only a subset of $\cL_{\bt}^{(n)}$, but if $\btau$ is primitive, then $X_{\tau}^{(n)}$ is minimal and $\cL(X_{\tau}^{(n)}) = \cL_{\tau}^{(n)}$ for every $n \geq 0$ (see {\em e.g.} \cite[Theorem 5.2]{BD14}).
We set $X_{\btau}\coloneqq X_{\btau}^{(0)}$ and call this set the subshift {\it generated} by ${\btau}$. We also set $\cA:=\cA_0$. One has  $X_{\bt}\subseteq \cA^{\mathbb Z}.$
A subshift $X$ is said to be  {\em generated} by $\btau$ if $X = X_{\btau}$.
The fundamental relation between the different subshifts $X_{\btau}^{(n)}$ is that 
\begin{equation}
    \label{eq:Sadic:connection_levels}
    X_{\btau}^{(n)} = 
    \bigl\{S^k \tau_{n,m}(x) : x \in X_{\btau}^{(m)},\, 
                        0 \leq k < |\tau_{n,m}(x_0)| \bigr\}
\end{equation}
for all $m > n \geq 0$ \cite{BSTY}. The subshift
$X_{\bt}^{(n)}$ is called  the level-$n$  subshift associated to $\bt$. 

The classical setting of substitutive subshifts can be viewed as a special case of the $S$-adic formalism, as we now explain.
Let $\tau$ be a substitution on the alphabet $\cA$, and let $M_{\tau}$ denote its incidence matrix.
The substitution $\tau$ is said to be \emph{primitive} if there exists a positive integer $k$ such that
$M_{\tau}^k$ has strictly positive entries.
Consider the constant directive sequence $\btau = (\tau)_{n \geq 0}$.
This directive sequence is then (strongly) primitive.
The substitutive subshift $X_{\tau}$ is defined as the subshift $X_{\btau}$ generated by the constant directive sequence $\btau$.
In particular, $X_{\btau}$ is minimal and uniquely ergodic  (see \emph{e.g.} \cite{Que}).

The following lemma is useful for constructing limit points. An alterative formulation using the notion of {\em nested sequences} is proved in \cite[Lemma 4.4]{mercat}. See also \cite{ICALP} where the term \emph{congenial} is used.

\begin{lemma}
    \label{lem:existence_limit_points}
    Let $\btau=(\tau_n:\cA_{n+1}^*\to\cA_n^*)_{n\ge0}$ be a primitive directive sequence.
    Then, there exists a sequence $(a_n \in \cA_n)_{n\ge0}$ such that, for every $n\ge0$, the word $\tau_n(a_{n+1})$ starts with $a_n$.
\end{lemma}
\begin{proof}
    Fix $n\ge 0$ and $b_n\in\cA_n$.
    We inductively define $a_{n,n}=b_n$, and for $0\le i<n$, let $a_{i,n}$ be the first letter of $\tau_i(a_{i+1,n})$.
    Since the product $\prod_{i\ge0} \cA_i$ is compact in the profinite topology, there is an increasing sequence $(n_k)_{k\ge 0}$ such that, for every $i \ge 0$, the sequence $(a_{i,n_k})_{k\ge i}$ is eventually constant; denote its limiting value by $a_i\in\cA_i$.
    Then, for each $i$ and all sufficiently large $k$, the letter $a_{i,n_k}$ is the first letter of $\tau_i(a_{i+1,n_k})$.
    Passing to the limit shows that $a_i$ is the first letter of $\tau_i(a_{i+1})$.
    Thus $(a_i)_{i\ge 0}$ satisfies the required property.
\end{proof}

\subsection{Recognizability and Kakutani--Rohklin partitions}
\label{subsec:KR}

In this article we often ask the directive sequences to be recognizable, which means roughly speaking that one can desubstitute  subshifts in a unique way (for more on this subject, we refer the reader \emph{e.g.} to \cite{BSTY,DP}).
We  express here recognizability
(see Definition \ref{def:recognizable} below)
in terms of  Kakutani--Rohklin partitions. These  are partitions associated to minimal  subshifts (and  even Cantor systems) in order  to  provide representations  as adic transformations on ordered Bratteli diagrams (see \cite{GPS92}).

Let $\btau = (\tau_n \colon \cA_{n+1}^* \to \cA_n^* : n \geq 0)$ be a directive sequence.
Throughout the article, we use the {\em heights} of $\btau$, which are defined by
\[  h_n(u) = |\tau_{0,n}(u)|
    \enspace \text{for $u \in \cA_n^*$ and $n \geq 1$,}  \]
and $h_0(u) = |u|$ for $u \in \cA_0^*$.
Note that $h_n(u) = \sum_{a\in\cA_n} h_n(a) \, |u|_a$, where $|u|_a$ is the number of times that $a$ occurs in $u$.
We remark that the heights $h_n(u)$ defined in this article are not related to the {\em height} of a constant-length substitution, as defined in \cite{Que}.

We define the sets 
\begin{equation}
    \label{eq:Sadic:defi_Bn(a)}
    B_n(a) = \bigl\{ \tau_{0,n}(x) : x \in X_{\btau}^{(n)}, \, x_0 = a \bigr\}
    \enspace \text{for $n \geq 1$ and $a \in \cA_n$.}
\end{equation}
In view of \eqref{eq:Sadic:connection_levels}, $X_{\btau}$ is equal to the union of the sets in 
\begin{equation}
\label{eq:P_n}
\cP_n \coloneqq \{ S^k B_n(a) : a \in \cA_n,\, 0 \leq k < h_n(a) \}.   
\end{equation}

\begin{definition}\label{def:recognizable}
The directive sequence $\btau$ is  said to be {\em recognizable} if the sets in $\cP_n$ are  pairwise disjoint; equivalently, if $\cP_n$ is a partition of $X_{\btau}$.
When $\btau$ is recognizable, $\cP_n$ is called the {\em $n$-th Kakutani--Rohklin partition} of $\btau$.
\end{definition}

\begin{remark}\label{rem:recog_substitution}
In the case of substitutive subshifts, if a substitution $\tau$ is primitive and aperiodic, the constant directive sequence $\bt=(\tau)_{n\geq0}$ is recognizable by Moss\'e’s theorem~\cite{mosse92}; see also~\cite{BSTY}.
\end{remark}

Let us assume that $\btau$ is recognizable. Let us fix $n\geq 0$.
A \emph{tower} of $X_{\btau}$  is a  union of sets of the form  $S^k B_n(a), \, 0 \leq k < h_n(a)$
for a fixed  letter $a$. Such a tower  has   height $h_n(a)$  and  \emph{basis} $B_n(a)$. It consists of $h_n(a)$ levels of the form
$S^k B_n(a)$.
Towers are well adapted to the study of eigenvalues. 
Indeed, let $f$ be an eigenfunction associated to the  additive eigenvalue $\alpha$.
Assume that  $f$ takes constant values on the  atoms of the   partition $\cP_n$.
Once the value of $f$ is fixed on  the basis  $B_n(a)$, then $f$
takes value $f + k\alpha$ on the $k$-th   level of the  corresponding tower.



We now give a combinatorial description of the Kakutani--Rohklin partitions $\cP_n$.
Equation \eqref{eq:Sadic:connection_levels} ensures that, for every point $x^{(n)} \in X_{\btau}^{(n)}$, there exist $x^{(n+1)} \in X_{\btau}^{(n+1)}$ and $0 \leq k_n < |\tau_n(x^{(n+1)}_0)|$ such that $x^{(n)} = S^{k_n} \tau_n(x^{(n+1)})$.
So, for any fixed $x^{(0)} \in X_{\btau}$, one can inductively construct a sequence $(k_n, x^{(n+1)})_{n\geq0}$, with $x^{(n)} \in X_{\btau}^{(n)}$, $0 \leq k_n < |\tau_n(x^{(n+1)}_0)|$, and $x^{(n)} = S^{k_n} \tau_n(x^{(n+1)})$.
By decomposing $\tau_n(x^{(n+1)}_0)$ as $ \tau_n(x^{(n+1)}_0)= u_n \, a_n \, v_n$, with $u_n$ of length $k_n$, $a_n = x^{(n)}_0$ and $v_n \in \cA_n^\ast$, we obtain
\begin{multline}
    \label{eq:unroll_address}
    x_{[-\sum_{0\leq j \leq n} h_j(u_j),\sum_{0\leq j \leq n} h_j(v_j)]} \\ = 
        \tau_{0,n}(u_n) \, \tau_{0,n-1}(u_{n-1}) \dots \tau_0(u_1) \, u_0 
        \, \boldsymbol{.} a_0 \,
        v_0 \, \tau_0(v_1) \dots \, \tau_{0,n}(v_n). 
\end{multline}
This hierarchical decomposition of the point $x^{(0)} \in X_{\btau}$ is analogous to that obtained for substitutive sequences using the prefix-suffix automaton \cite{CanSie}, or  equivalently the Dumont--Thomas representation \cite{Dumont-Thomas}.
It is also central in our analysis, so, to handle it efficiently, we introduce the following notion.

\begin{definition} \label{def:address}
A {\em $\btau$-address} is a sequence $(u_n,a_n)_{n\geq0}$, with $u_n \in \cA_n^*$ and $a_n \in \cA_n \cap \cL(X_{\btau}^{(n)})$, such that $u_n a_n$ is a prefix of $\tau_n(a_{n+1})$ for all $n \geq 0$.
The {\em address space} of $\btau$, denoted $\Add(\btau)$, is the set of all $\btau$-addresses.
\end{definition}
We remark that the condition $a_n \in \cL(X_{\btau}^{(n)})$ is superfluous for primitive directive sequences, as in this case $\cA_n \subseteq \cL(X_{\btau}^{(n)})$ for every $n \geq 0$.

A $\btau$-address represents a path in the Bratteli diagram defined by the directive sequence; see \emph{e.g.} \cite{Durand:10}.
Note that, however, this Bratteli diagram might not have a continuous Vershik action since $\btau$ is not assumed to be proper. 
We refer the reader to  \cite{Durand:10,DDMP,BSTY} for additional details on this connection.

The basic properties of $\btau$-addresses are given in \Cref{prop:addresses} and \Cref{lem:address}.

\begin{proposition} \label{prop:addresses}
    Let $\btau = (\tau_n \colon \cA_{n+1}^* \to \cA_n^* : n \geq 0)$ be a directive sequence.
We use the notation
     $j_n \coloneqq \sum_{0 \leq i < n} h_i(u_i) $
      for every $\btau$-address $(u_n,a_n)_{n\geq0} \in \Add(\btau)$.
    \begin{enumerate}
        \item For every 
        $(u_n,a_n)_{n\geq0} \in \Add(\btau)$, there exists $x \in X_{\btau}$ such that,  for all $n \geq 0$, $x \in S^{j_n} B_n(a_n)$, with $0 \leq j_n < h_n(a_n)$. 
        
        \item For every $x \in X_{\btau}$, there exists 
        $(u_n,a_n)_{n\geq 0} \in \Add(\btau)$ such that $x \in S^{j_n} B_n(a_n)$, where  $j_n = \sum_{0 \leq i < n} h_i(u_i)$, for all $n \geq 0$.
    \end{enumerate}
    Furthermore, the $\btau$-address $(u_n,a_n)_{n\geq0}$ in Item (2) is unique if $\btau$ is recognizable.
\end{proposition}

We remark that in Item (1) of \Cref{prop:addresses}, the point $x$ might not be unique, that is, a $\btau$-address might not uniquely determine a point of the subshift.
This happens exactly when the partition $(\cP_n)_{n\geq 0}$ does not generate the topology of $X_{\btau}$.
We will introduce in Section \ref{subsec:decisive} sufficient conditions guaranteeing that the  sequence of 
partitions $(\cP_n)_n$ generates  the topology of $X_{\btau}$, based on the notion of decisiveness.  This  yields  in particular that,
for any continuous  function in $C(X,\Z)$,  there exists  a positive  integer $n$ such that $f$ takes constant values on the atoms of  $\cP_n$  (see  Proposition \ref{prop:decisive:dyn_inter}
for such  sufficient conditions).
Note that for primitive directive sequences, we can use the centered 1-block presentation of $\btau$ to recover decisive  directive sequences,  as discussed  below in \Cref{subsec:block_presentations} (see also \Cref{prop:decisive:suff_conds}).

\begin{lemma}
    \label{lem:address}
    Let $\btau = (\tau_n \colon \cA_{n+1}^* \to \cA_n^* : n \geq 0)$ be a recognizable and primitive directive sequence.
    Fix $x \in X_{\btau}$, with $(u_n,a_n)_{n\geq 0} \in \Add(\btau)$ being  its unique $\btau$-address. 
    Then, for every $n \geq 0$, $x_{[-\sum_{0 \leq j < n} h_j(u_j), 0]}$ equals
        $\tau_{0,n}(u_n) \, \tau_{0,n-1}(u_{n-1}) \dots \tau_0(u_1) \, u_0 \, a_0$
        and is a prefix $\tau_{0,n}(a_n)$.
\end{lemma}
\begin{proof} 
The proof works as previously for \eqref{eq:unroll_address}.
Inductively define $x^{(0)} = x$ and $(k'_n, x^{(n+1)})$ as the unique tuple such that $x^{(n)} = S^{k'_n} \tau_n(x^{(n+1)})$.
Then, $k'_n = k_n$ and $x^{(n)}_{[-k_n,0]} = u_n \, a_n$, from which the result follows by induction.
\end{proof}


The following result is folklore relating the values taken by the measure  $\mu$ and the  measures $\mu_n$. It implies in particular that the vectors  with entries   of the form  $\mu(B_n(a))$ and   $\mu_n([a]_n)$ are proportional.
We refer the reader to \cite{BKMS12} for more details.

\begin{proposition}
    \label{prop:measure_transfer}
    Let $\btau = (\tau_n \colon \cA_{n+1}^* \to \cA_n^* : n \geq 0)$ be a recognizable directive sequence, and let $n \geq 1$ be fixed.
    We denote by $[a]_n$ the cylinder associated to $a \in \cA_n$ in $X_{\btau}^{(n)}$.
    Define, for each invariant probability measure $\mu_n \in \cM(X_{\btau}^{(n)}, S)$ and Borel set $U \subseteq X_{\btau}$,
    \begin{equation}
        \label{eq:prop:measure_transfer:statement}
        \mu(U) = 
        \! \! \! \!
        \left. 
        \sum_{\substack{a \in \cA_n\\ 0 \leq k \lt h_n(a)}} 
        \! \! \! \!
        \mu_n(\{x \in [a]_n : S^k\tau_{0,n}(x) \in U\}) 
        \middle/ \sum_{a\in\cA_n} \!\! h_n(a) \, \mu_n([a]_n)
        \right. .
    \end{equation}
    Then, map $\mu_n \mapsto \mu$ is an affine bijection between $\cM(X_{\btau}^{(n)}, S)$ and $\cM(X_{\btau}, S)$.
\end{proposition}

We recall here the following immediate consequence of \Cref{prop:measure_transfer}:
if $\mu_n \in \cM(X_{\btau}^{(n)}, S)$, then, with the notation of \Cref{eq:Sadic:defi_Bn(a)}, one has 
\begin{equation}
    \label{eq:prop:measure_transfer:remark}
    \mu(B_n(a)) = 
    \left. 
    \mu_n([a]_n) 
    \middle/ \sum_{b\in\cA_n} h_n(b) \, \mu_n([b]_n) \right.
    \enspace \text{for every $a \in \cA_n$.}
\end{equation}

One  deduces that  (see also \cite{BD14})
\begin{equation}\label{eq:mubn}
\mu (B_n(a)) \in  \bigcap_{m \geq n} M_n\cdots M_m   \R_+^d, \text{ for all } n.
\end{equation}

\subsection{Centered block presentations}
\label{subsec:block_presentations}

We consider a centered version of the notion of a block substitution from \cite{Que}, focusing on the central letter $a_k$ of a word $a_0 \dots a_k \dots a_{2k}$ of length $2k+1$.

Let $\btau = (\tau_n \colon \cA_{n+1}^* \to \cA_n^* : n \geq 0)$ be a directive sequence and fix $k \geq 1$.
Define $\cA_n^{[k]} = \cL_{2k+1}(X_{\btau}^{(n)})$ for every $n \ge 0$.
We consider each word $a_0 \dots a_{2k} \in \cA_n^{[k]}$ as a single ``bracketed'' letter, written $[a_0 \dots a_{2k}]$.
This bracketed letter represents the letter $a_k$ together with its ``context'', given by the preceding letters $a_0 \dots a_{k-1}$ and the following letters $a_{k+1} \dots a_{2k}$.
Next, we define a substitution $\tau_n^{[k]} \colon \cA_{n+1}^{[k]} \to (\cA_{n}^{[k]})^*$ whose image on $[a_0 \dots a_{2k}]$ encodes $\tau_n(a_k)$ with the context inherited from $\tau_n(a_0 \dots a_{2k})$, as follows.
Given $[a_0 \dots a_{2k}] \in \cA_n^{[k]}$, write $\tau_n(a_0 \dots a_{2k}) = b_0 b_1 \dots b_{\ell-1}$, where $\ell \coloneqq \sum_{0 \leq i \leq 2k} |\tau_n(a_i)|$ is the length of $\tau_n(a_0 \dots a_{2k})$.
Set $r \coloneqq |\tau_n(a_0 \dots a_{k-1})|$, which is the position at which $\tau_n(a_k)$ occurs in $\tau_n(a_0 \dots a_{2k})$.
We then define $\tau_n^{[k]}([a_0 \dots a_{2k}])$ to be the word of bracketed letters obtained by sliding a window of length $2k+1$ across $\tau_n(a_0 \dots a_{2k})$, centered successively at each position of $\tau_n(a_k)$, namely
\[  [b_{r-k} b_{r-k+1} \dots b_{r+k}]  [b_{r-k+1}  b_{r-k+2}  \dots  b_{r+k+1}]  \dots 
    [b_{r+|\tau_n(a_k)|-1-k}  b_{r+|\tau_n(a_k)|-k}  \dots  b_{r+|\tau_n(a_k)|-1+k}].     \]
The word $\tau^{[k]}_n([a_0 \dots a_{2k}])$ is thus the centered $k$-sliding block presentation of the subword of $\tau_n(a_0 \dots a_{2k})$ of length $|\tau_n(a_k)|$ starting at position $|\tau_n(a_0 \dots a_{k-1})|$ (with positions indexed from $0$) together with its context.

This yields substitutions $\tau_n^{[k]} \colon (\cA_{n+1}^{[k]})^* \to (\cA_n^{[k]})^*$, and therefore a directive sequence $\btau^{[k]} = (\tau_n^{[k]} : n \geq 0)$, called the \emph{centered $k$-block presentation} of $\btau$.

\begin{proposition}
Let $\btau = (\tau_n \colon \cA_{n+1}^* \to \cA_n^* : n \geq 0)$ be a directive sequence and let $\btau^{[k]}$ be its centered $k$-block presentation.
Then, the map $\pi \colon \cA_0^{[k]} \to \cA_0$ given by $\pi([a_0 \dots a_{2k}]) = a_k$ defines a conjugacy $\pi \colon X_{\btau^{[k]}} \to X_{\btau}$.
\end{proposition}

\section{Coboundaries, extension graphs and decisiveness} \label{sec:coboundaries}
This section  is devoted to combinatorial properties which will play a key role for the following. Section  \ref{subsec:coboundaries} first introduces letter-coboundaries and describes their basic properties, by  focusing on  the notion  of symbolic discrepancy and on  its  spectral relevance.  We then consider  extension graphs  and  highlight their significance in terms of coboundaries
  in Section \ref{subsec:extensiongraphs}.
The notion of decisiveness is  then introduced in Section 
\ref{subsec:decisive} with its interpretation in terms of extension graphs.

\subsection{Letter-coboundaries and symbolic discrepancy}\label{subsec:coboundaries}
The notion of a \emph{letter-coboundary} was introduced by B.~Host in \cite{Host86}, who showed how  crucial they are  for computing the eigenvalues of  substitutive minimal subshifts.
This  notion will also prove to  play a central role in this article.

Let $X \subseteq \cA^\Z$ be a subshift and fix $u \in \cA$. 
A word $w \in \cA^*$ is a \emph{return word to $u$} if $wu \in \cL(X)$ and $wu$ contains exactly two occurrences of $u$, one as a prefix and one as a suffix. 

\begin{definition} \label{def: coboundary}
A \emph{letter-coboundary on $X$} is a morphism $c \colon \cA^* \to \R$ (with respect to the additive structure of $\R$) such that, for every $a \in \cA$ and every return word $u$ to $a$ in $X$, we have $c(u) = 0$. 
We call $c$ \emph{trivial} if $c(a) = 0$ for all $a \in \cA$.
\end{definition}

\begin{remark}
Host defined letter-coboundaries as   morphisms with values in $\S^1$ instead of $\R$ \cite{Host86}.
It can be shown that both approaches are equivalent  \cite{mercat}.
The present additive version is adapted for notational simplicity, especially in \Cref{sec:carac,sec:carac_finite_rank}, and for its closer connection to the usual matrix computations that are necessary for computing eigenvalues in concrete examples, such as the ones in \Cref{sec:examples}
\end{remark}

The following two lemmas are widely used throughout the article when dealing with letter-coboundaries. \Cref{lem:cobord&rho} gives a key description of them in terms of associated functions. Its version for coboundaries with values in $\S^1$ was given by Host \cite{Host86}; see also  the proof in \cite[Lemma 3.17]{BCBY}.  For the sake of clarity, we provide here a proof for letter-coboundaries with values in $\R$.


\begin{lemma}\label{lem:compositioncoboundary}
    Let $\tau\colon \cA^*\to \cB^*$ be a monoid morphism, $X\subseteq \cA^\Z$ be a subshift and set $Y=\bigcup_{k\in \Z}S^k\tau(X)$. If $c \colon \cB^* \to \R$ is a letter-coboundary on $Y$, then $c\circ \tau$ is a letter-coboundary on $X$.
\end{lemma}

\begin{proof}
    It suffices to note that if $u$ is a return word to the letter $a\in\cA$ in $X$, and 
    $b\in\cB$ is the first letter of $\tau(a)$,
    then $\tau(u)$ is a return word to $b$ in $Y$, so $c\circ\tau(u)=0$.
\end{proof}

\begin{lemma}
\label{lem:cobord&rho}
Let $X\subseteq \cA^\Z$ be a minimal subshift. 
A morphism $c \colon \cA^*\to \R$ is a letter-coboundary on $X$ if and only is there exists a map $\rho \colon \cA \to \R$ such that $c(a) = \rho(b) - \rho(a)$ for every length-$2$ word $ab \in \cL(X)$.
\end{lemma}
\begin{proof}
Suppose that  there exists $\rho$ with $c(a)=\rho(b)-\rho(a)$ for all $ab\in\cL(X)$.  
We can compute, for any $u=u_0\cdots u_k\in\cL(X)$ with $u_0=u_k$,
\[
c(u)=\sum_{i=0}^{k-1} c(u_i)
    =\sum_{i=0}^{k-1}\rho(u_{i+1})-\rho(u_i)
    =\rho(u_k)-\rho(u_0)=0,
\]
so in particular $c(v)=0$ for every return word $v$ to a letter.

Conversely, suppose that $c$ is a letter-coboundary.
Without loss of generality, we also assume that $\cA \subseteq \cL(X)$.
Fix $x \in X$ and define $r_k = c(x_0 x_1 \dots x_{k-1})$ for $k \ge 1$.
Observe that $r_k = r_j$ for any $j > k \geq 1$ such that $x_k = x_j$, for $x_k x_{k+1} \dots x_{j-1}$ is a return word to $x_k$ in $X$ in this case, which implies
\[  c(x_0 x_1 \dots x_{j-1}) =
    c(x_0 x_1 \dots x_{k-1}) + 
    c(x_j x_{j+1} \dots x_{k-1}) = 
    c(x_0 x_1 \dots x_{k-1}).    \]
Then, since $X$ is minimal, we can define $\rho\colon\cA \to \R$ by $\rho(a) = r_j$ where $j \ge 1$ is any position at which $a$ occurs in $x$.

Let us check that $\rho$ satisfies the desired property.
Consider $ab \in \cL(X)$ of length 2.
By minimality, $x_j x_{j+1} = ab$ for some $j \ge 1$.
Then, $\rho(x_{j+1}) = c(x_0\dots x_j) = c(x_0\dots x_{j-1}) + c(x_j) = \rho(x_j) + c(x_j)$, that is, $c(a) = \rho(b) - \rho(a)$.
\end{proof}

\begin{lemma}
    \label{lem:cob_zero_integral}
    Let $X \subseteq \cA^\Z$ be a minimal subshift and let $c \colon \cA^* \to \R$ be a letter-coboundary in $X$. Then,
    \begin{equation*}
        \sum_{a \in \cA} \mu([a]) \, c(a) = 0
        \enspace \text{for every $\mu \in \cM(X,S)$.}
    \end{equation*}
\end{lemma}
\begin{proof}
\Cref{lem:cobord&rho} provides a map $\rho \colon \cA \to \R$ satisfying $c(a) = \rho(b) - \rho(a)$ for all $a,b \in \cA$ with $ab \in \cL(X)$.
So, if we define the maps 
\begin{equation}\label{eq:chro}
\overline{c}, \overline{\rho} \colon X \to \R \mbox{ by }\overline{c}(x) = c(x_0)  \mbox{ and }\overline{\rho}(x) = \rho(x_0),
\end{equation}
then $\overline{c} = \overline{\rho} \circ S - \overline{\rho}$.
Therefore, for any $\mu \in \cM(X,S)$,
\begin{equation*}
    \sum_{a \in \cA} \mu([a]) \, c(a) = 
    \int \overline{c} \, \mathrm{d}\mu = 
    \int (\overline{\rho} \circ S - \overline{\rho}) \, \mathrm{d}\mu = 0.
\end{equation*}
\end{proof}

\begin{remark}  \label{rem:coboundary}  
Letter-coboundaries as defined here are related to the classical notion of a  coboundary in topological dynamics in the following way. 
Given a minimal system $(X,T)$, the {\em coboundary map $\beta$} is the endomorphism of the additive group $C(X,\R)$ of real valued continuous function from $X$ to $\R$ defined by $\beta f = f\circ T - f$.
A map $f \in C(X,\R)$ is said to be a {\em real coboundary of $(X,T)$} if it lies in the image of $\beta$, that is, if there exists $g\in C(X,\R)$ such that $f=g\circ T-g$. 
As  an illustration, the map $\overline{c}$ defined in \eqref{eq:chro} is  a  real coboundary of the subshift  $(X,S)$.
Suppose now that $X \subseteq \cA^\Z$ is a subshift and that $c \colon \cA^* \to \R$ is a letter-coboundary on $X$. 
One can check that the map $c \mapsto \overline{c}$ defines a one-to-one correspondence between letter-coboundaries on $(X,S)$ and real coboundaries
$f= \beta g$ of $(X,S)$ such   that both $f(x)$  and $g(x)$  depend only on the first letter  $x_0$ of $x$,  for every $x \in X$.
\end{remark}

We recall a classical theorem from topological dynamics that highlights the role played by   coboundaries.

\begin{theorem}
[Gottshalk--Hedlund's Theorem \cite{GotHed:55}]
\label{theo:GH} 
Let $(X,T)$ be a minimal topological dynamical system, and fix a continuous map $f \colon X \to \R$.
Then, the following conditions are equivalent:
\begin{enumerate}
    \item there exists a continuous map $g \colon X \to \R$ such that $f = g \circ T - g$;
    \item there exists $x_0 \in X$ such that  the Birkhoff sums $ \sum_{0 \leq n \lt N} f(T^n x_0) $
    remain bounded as $N \to \infty$.
\end{enumerate}
\end{theorem}

Sequences with bounded  Birkhoff sums  are at the heart of  Gottshalk--Hedlund's Theorem. 
We end this section by recalling the related notions  of
\emph{symbolic discrepancy}  and \emph{balancedness}, 
which we then revisit with    \Cref{prop:balance_char_Sadic}; see also  \Cref{balanced=>spaces_decomposition} and \Cref{ex:TM:freqs&eigs}.
For the equivalence below, see \cite{unimodular}.

\begin{definition}\label{def:balance}
Let $X$ be a minimal subshift.
We say  that $f \in C (X , \mathbb{R} )$ is {\em balanced} for $X$    whenever  there exists a  constant $C_f>0$ 
such that 
\[
\Big|\sum_{i=0}^n f(S^ix) - f(S^iy) \Big| \leq C_f 
\enspace \text{for all $x,y \in X$ and $n \ge 1$,} 
\]
or equivalently, if it has {\em finite discrepancy}, \emph{i.e.}, if there exists  a constant $\alpha_f$ such that 
\[  \sup_{n \ge 1} \Big| \sum_{i=0}^n f(S^i x) -  \alpha_f \Big|  < \infty  \mbox { for all  }x \in X. \]
\end{definition}

The notion of balancedness for continuous functions  was introduced in \cite{unimodular}.
It was inspired by the classical notion of balancedness in word combinatorics  which is stated in combinatorial terms for subshifts as follows: a minimal subshift $(X,S)$ on the alphabet $\cA$ is said to be {\em balanced on the factor} $v\in\cL(X)$ whenever there exists a constant $C_v$ such that for all words $w,w'\in\cL(X)$ of the same length, the number of occurrences of $v$ in $w$ and $w'$ is bounded from above by $C_v$. 
This is equivalent to the fact that the characteristic function $\boldsymbol{1}_{[v]}:X\to\{0,1\}$ is balanced in the sense of \Cref{def:balance}. 
The subshift $(X,S)$ is said to be {\em balanced on letters} if it is balanced on $a$ for all $a\in \cA$ and {\em balanced on factors} if it is balanced on $v$ for all $v \in \cL(X)$. 
As a consequence of \Cref{theo:GH}, if $(X,S)$ is a minimal subshift balanced on a factor $v\in \cL(X)$, then the frequency $\alpha$ of $v$ equals $\mu([v])$ for every ergodic measure $\mu$.
The relations between balance and  symbolic discrepancy have been developed in \cite{adam03,adam05}.

\begin{remark} \label{rem:discrepancy}
One checks that  if a minimal subshift $X$  is balanced on all its factors, then it is uniquely ergodic. Let $\mu$ stand for its unique invariant measure.   For every $u \in \cL(X)$, the {\em normalized characteristic function} $\boldone_{[u]} - \mu([u])$ has bounded Birkhoff sums, and thus,  by Theorem  \ref{theo:GH},  $\boldone_{[u]} - \mu([u])$ is a real coboundary.   This implies  that $\mu([u])$ is an  eigenvalue, by considering the  cohomological equation modulo 1, and also that $E(X,S) = I(X,S)$.
In this case, the system has the maximal continuous eigenvalue group property according to the terminology of \cite{CORTEZ_DURAND_PETITE_2016,DURAND_HOST_SKAU_1999}.
\end{remark}

\subsection{Extension graphs}\label{subsec:extensiongraphs}
We recall here the definition of the extension graphs associated to subshifts, and then use the extension graph of the empty word to introduce the notion of {\em decisiveness} of a directive sequence, which will play an important role in Section \ref{sec:carac_finite_rank}.

Let $X \subseteq \cA^\Z$ be a subshift.
We consider two copies $\cA_L$ and $\cA_R$ of the alphabet $\cA$, with canonical bijections $a \mapsto a_L$ and $a \mapsto a_R$ from $\cA$ onto $\cA_L$ and $\cA_R$, respectively.
For $w \in \cL(X)$, define 
\[  L_X(w) = \{ a_L \in \cA_L : a w \in \cL(X) \}
    \enspace \text{and} \enspace
    R_X(w) = \{ b_R \in \cA_R : w b \in \cL(X) \}. \]
The {\em extension graph} $\Gamma_X(w)$ of $w$ in $X$ is the bipartite undirected graph with vertex set $L_X(w) \cup R_X(w)$ and edge set
\[  \big\{ (a_L,b_R) \in \cA_L \times \cA_R  :  
       a w b \in \cL(X)  \big\}.  \]
Note that we use ordered pairs to represent undirected edges. 
This introduces no ambiguity because edges always connect a vertex in $\cA_L$ to a vertex in $\cA_R$, allowing us to canonically represent any edge as a tuple $(a_L,b_R)$ with $a_L \in \cA_L$ and $b_R \in \cA_R$.
The number of edges in $\Gamma_X(w)$ is denoted $|\Gamma_X(w)|$.

Extension graphs allow the characterization of wide  families of subshifts, and in particular, the class of dendric  subshifts  defined  below.
This  class of subshifts includes Sturmian and Arnoux--Rauzy subshifts, as well as codings of regular interval exchange transformations, and exhibits strong rigidity properties \cite{dendric}.
\begin{definition} \label{def:dendric}
A minimal subshift $X$
is said  \emph{dendric}
if all the extension graphs of all its factors are trees (\emph{i.e.},  they are connected and acyclic).
On a two-letter alphabet,  dendric shifts are called 
\emph{Sturmian  shifts}. A Sturmian shift $X$ is  a symbolic coding of an irrational translation  of the form $x \mapsto x +\alpha \mod 1$, acting on ${\mathbb T}={\mathbb R}/{\mathbb Z}$; the  irrational number $\alpha$ is called the 
\emph{parameter} of  the Sturmian shift.
\end{definition}
For a  Sturmian shift with parameter $\alpha$, one  has $E(X,S)=\alpha {\mathbb Z}+{\mathbb Z}= I(X,S)$, and for dendric shifts $E(X,S)=I(X,S)$ by \cite{unimodular}. For more  on Sturmian shifts, see \emph{e.g.} \cite{Pytheas}.

We are mostly interested in the extension graph $\Gamma_X(\varepsilon)$ of the empty word. 
A connected component of $\Gamma_X(\varepsilon)$ is an inclusion-maximal set $K \subseteq \cA_L \cup \cA_R$ such that every pair of vertices in $K$ is connected by a path in $\Gamma_X(\varepsilon)$.
We denote by $\cP_X^R$ (resp.\@ $\cP_X^L$) the collection of all subsets $C$ of $\cA$ such that $C_R = K \cap \cA_R$ (resp.\@ $C_L = K \cap \cA_L$) for some connected component $K$ of $\Gamma_X(\varepsilon)$.
Note that $\cP_X^R$ and $\cP_X^L$ are partitions of $\cA$.

\begin{remark}\label{rem:connectedminimal}
Let $X$ be a subshift with a point $x$ whose orbit is dense.
Then, for every $K \in \cK_X(\varepsilon)$, neither of the sets $K \cap \cA_L$ and $K \cap \cA_R$ is contained in the other (seen as subsets of $\cA$), unless $\Gamma_X(\varepsilon)$ is connected.
Indeed, if one inclusion held, then $x$ would eventually see only letters in $K \cap \cA_R$ or in $K \cap \cA_L$.
But when $\Gamma_X(\varepsilon)$ is disconnected, these sets are strictly contained in $\cA_L$ and $\cA_R$, respectively, which contradicts the density of the orbit of $x$.
\end{remark}


\begin{example}\label{ex:extension_graph_CM}
    Consider the substitution $\sigma \colon \cA^* \to \cA^*$ on $\cA = \{0,1,2\}$ that is given by
    \[   \sigma : \begin{cases} 
            0 & \mapsto 010 \\
            1 & \mapsto 21  \\
            2 & \mapsto 210
        \end{cases}     \]
    Denote by $X_\sigma$ the substitutive subshift generated by $\sigma$. 
    One can check by hand that $\{ 01, 10, 21, 02 \}$ is the set of words of length 2 in the language of $X_\sigma$.
    So, the extension graph $\Gamma_{X_\sigma}(\varepsilon)$ has two connected components, as illustrated in Figure \ref{fig:extensiongraph}.
One has $\cP_X^L=\{\{0,2\},\{1\}\}$  and $\cP_X^R=\{\{0\},\{1,2\}\}.$
    Moreover observe that  any letter-coboundary  $c$
    satisfies  that $\rho(1)=\rho(2)$,      with the notation of  Lemma
    \ref{lem:cobord&rho}.
    Indeed, $c(0)=\rho(1)-\rho(0)=\rho(2)-\rho(0).$ This example is continued in \Cref{ex:CM}.
    
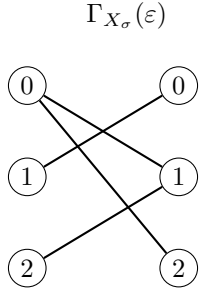
\begin{figure}[ht]
 \tikzset{node/.style={circle,draw,minimum size=0.5cm,inner sep=0pt}}
 \tikzset{title/.style={minimum size=0.5cm,inner sep=0pt}}

 \begin{center}
  \begin{tikzpicture}
   \node[title](ee) {$\Gamma_{X_{\sigma}}(\varepsilon)$};
   \node[node](eal) [below left= 0.5cm and 0.6cm of ee] {$0$};
   \node[node](ebl) [below= 0.7cm of eal] {$1$};
   \node[node](ecl) [below= 0.7cm of ebl] {$2$};
   \node[node](ear) [right= 1.5cm of eal] {$0$};
   \node[node](ebr) [below= 0.7cm of ear] {$1$};
   \node[node](ecr) [below= 0.7cm of ebr] {$2$};
   \path[draw,thick]
    (eal) edge node {} (ebr)
    (ebl) edge node {} (ear)
    (ecl) edge node {} (ebr)
    (eal) edge node {} (ecr);
  \end{tikzpicture}
 \end{center}

 \caption{The extension graph $\Gamma_{X_{\sigma}}(\varepsilon)$ of $X_{\sigma}$ admits two connected components.}
 \label{fig:extensiongraph}
\end{figure}
\end{example}


We now  quantify the size of the  letter-coboundary space using the subshift’s extension graph of the empty word.  Our main result in this section is Theorem \ref{thm:manifold},
from which, 
in particular, we characterize in \Cref{cor:connectedtrivial} when a minimal shift has trivial coboundaries in terms of this graph. It is based on the following simple observation.
\begin{remark}\label{rem:constantK}
    Let $c$ be a letter-coboundary  and $\rho$  the associated map given by Lemma \ref{lem:cobord&rho}.  The map $\rho$  takes the  same value in all the letters of  $K \cap \cA_R$, for any connected component $K$
of $\Gamma_X(\varepsilon)$, that is, $\rho$ is constant in the elements of the partition $\cP_X^R$.
\end{remark}

To state \Cref{thm:manifold}, we first introduce some notation.
Let $X \subseteq \cA^\Z$ be a minimal subshift.
Denote by $\cC_X$ the set of letter-coboundaries $c \colon \cA^* \to \R$ in $X$.
We let $\cF_X$ be the set of maps $\rho \colon \cA \to \R$ that are constant in each $C \in \cP_X^R$ ({\em i.e.}, $\rho(a) = \rho(b)$ for all $a,b \in C$ and $C \in \cP_X^R$) and $\cF^0_X \subseteq \cF_X$ be the subset consisting of constant maps ({\em i.e.}, $\rho(a) = \rho(b)$ for all $a,b \in \cA$).
These three sets have the structure of $\R$-vector space under the operations of point-wise addition and scalar multiplication.

For $\rho \in \cF_X$, let $c_\rho \colon \cA^* \to \R$ be the morphism defined by $c_\rho(a) = \rho(b) - \rho(a)$, where $b$ is any letter such that $ab \in \cL(X)$.
This definition does not depend on the chosen $b$ because, since $\rho$ is constant in each $C \in \cP_X^R$, $\rho$ attains the same value for every different choice of $b$.
One easily checks that the map $\rho \mapsto c_\rho$ is $\R$-linear.

The following result is inspired by discussions with Clemens M\"ullner; see  \cite{Sadic7}.

\begin{theorem}
\label{thm:manifold}
Let $X \subseteq \cA^\Z$ be a transitive subshift.
Then, there is a short exact sequence
\[  0 \longrightarrow \cF^0_X
      \overset{\iota}{\longrightarrow} \cF_X 
      \overset{\eta}{\longrightarrow} \cC_X
      \longrightarrow 0   \]
where $\iota$ is the inclusion map and $\eta$ maps $\rho$ to $c_\rho$.
In particular, as an $\R$-vector space, $\cC_X$ has dimension $r-1$, where $r$ is the number of connected components of $\Gamma_X(\varepsilon)$.
\end{theorem}

\begin{proof}
The map $\eta$ is onto by \Cref{lem:cobord&rho}, and it is immediate from the definition of $\eta$ that $\cF^0_X \subseteq \ker \eta$.
So, we only need to prove that $\ker \eta \subseteq \cF^0_X$.

Let $\rho \in \ker \eta$.
Then, $\rho(b) - \rho(a) = c_\rho(a) = 0$ for all edges $(a_L,b_R)$ of $\Gamma_X(\varepsilon)$.
This implies that if $x$ is any point of $X$, then $\rho(x_{i+1}) = \rho(x_i)$ for every $i \in \Z$, and thus that $\rho(x_i) = \rho(x_0)$.
The subshift $X$  being transitive, we may choose $x$ to have a dense orbit; we obtain $\rho(a) = \rho(x_0)$ for all $a \in \cA$.
Therefore, $\rho \in \cF^0_X$, finishing the proof.
\end{proof}

There are natural families of subshifts where the extension graph of the empty word has several connected components, and hence the space $\cC_X$ is not one-dimensional. One of them is the class of  {\em specular subshifts} which  arise  as symbolic codings of linear involutions,   the latter  generalizing  interval exchange transformations (see \cite{BDDPRR:17}).  In a specular subshift, the extension graph of every non-empty word is a tree, and the extension graph of the empty word is a union of two connected components. 
For these subshifts, \Cref{thm:manifold} implies that the space $\cC_X$ has dimension $2$. We handle specular subshifts in more detail in \Cref{sec:specular}.

\begin{corollary}
\label{cor:connectedtrivial}
Let $X \subseteq \cA^\Z$ be a minimal shift. 
Then, the extension graph $\Gamma_X(\varepsilon)$ is connected if and only if every letter-coboundary in $X$ is trivial.
\end{corollary}

Note that  this article, we define letter-coboundaries with values in $\R$, in contrast to other works where the codomain is $\R/\Z$ (equivalently, the unit circle in the complex plane). 
The following lemma, which is a consequence of \Cref{thm:manifold}, shows that this distinction makes no difference.

\begin{lemma}\label{lem:cob a la host}
Let $X \subseteq \cA^\Z$ be a transitive subshift.
Suppose that $\tilde c \colon \cA^*\to\R/\Z$ is a morphism such that $\tilde c(u) = 0\pmod\Z$ for every return word $u$ to some letter.
Then there exists a letter-coboundary $c\colon \cA^*\to\mathbb R$ in $X$ with $c(a) = \tilde c(a) \pmod{\Z}$ and $|c(a)| \le 1$ for all $a \in \cA$
\end{lemma}
\begin{proof}
By \cite[Lemma 3.17]{BCBY} (cf. \Cref{lem:cobord&rho}), there is $\tilde\rho:\cA\to\R/\Z$ such that
$\tilde c(a) = \tilde\rho(b)-\tilde\rho(a)\pmod{\Z}$ for every $ab\in\cL(X)$ of length $2$.
As in the proof of \Cref{thm:manifold}, $\tilde\rho$ is constant modulo $\mathbb Z$ on each connected component of $\Gamma_X(\varepsilon)$.
Choose a map $\rho:\cA\to\mathbb R$, constant on each connected component of $\Gamma_X(\varepsilon)$, such that $\rho(a) = \tilde \rho(a) \pmod{\Z}$ and $|\rho(a)| \le 1/2$ for every $a \in \cA$.
Then, by \Cref{lem:cobord&rho} there is a coboundary $c\colon\cA^*\to\R$ with
$c(a)=\rho(b)-\rho(a)$ for all $ab\in\cL(X)$ of length $2$.
Therefore $c(a) = \tilde c(a)\pmod{\Z}$ and $|c(a)| \le 1$ for every $a\in\cA$.
\end{proof}

\subsection{Decisiveness}\label{subsec:decisive}
We now introduce the  concept of {\em decisiveness} for directive sequences, which plays a crucial role in \Cref{sec:carac_finite_rank}.
This notion is inspired by the work of Karpel and Downarowicz \cite{karpel_downa} and their notion of decisiveness for Bratteli diagrams. In their setting, an ordered Bratteli diagram is called decisive if the corresponding Vershik map  can be extended in a unique way to a homeomorphism of the whole path space of the Bratteli diagram, where attention must be paid to maximal and minimal paths.  The  definition  introduced below is not  equivalent to the notion of decisiveness of  \cite{karpel_downa}, but  can be considered as  a  combinatorial counterpart.

In the present $S$-adic setting, it is natural to require for our directive sequences that each unidirectional limit point $\lim_{n\to\infty} \tau_{0,n}(a_n)$ in $\cA^\N$ admits a unique  extension $\lim_{n\to\infty} \tau_{0,n}(b_n.a_n)$   as a  two-sided sequence in 
$\cA^\Z$. 
The following definition aims to provide a sufficient condition for  guaranteeing this  extension property.

\begin{definition} \label{def:decisive}
Let $\btau = (\tau_n \colon \cA_{n+1}^* \to \cA_n^* : n \geq 0)$ be an everywhere growing  directive sequence.
We call $\btau$ {\em decisive} if, for every $n \geq 0$, there exist maps $\ell,r \colon \cA_{n+1} \to \cA_n$ such that $\tau_n(b)$ begins with $r(a)$ and $\tau_n(a)$ ends with $\ell(b)$ for all $ab \in \cL(X_{\btau}^{(n+1)})$.
\end{definition}


Large classical families of subshifts admit decisive $S$-adic representations as a consequence of  Proposition \ref{prop:decisive:suff_conds} below. 
In particular, dendric shifts  (see Definition \ref{def:dendric}) admit decisive $S$-adic representations  since  they are proper by \cite{dendric}.
Natural non-decisive examples include, for instance,  
$\bt$ taking as a constant value a \emph{left-permutative} substitution $\tau$, that is, a substitution such that   for any letter $a$, the set of  first letters of $\tau(a)$ is in bijection with the alphabet $\cA,$
as for instance the \emph{Thue-Morse  substitution} $$\tau: a\mapsto ab, \ b \mapsto ba$$ (see also Section  \ref{subsec:Itau}).


Decisiveness  has a natural dynamical interpretation, as shown by next proposition.
\begin{proposition}
    \label{prop:decisive:dyn_inter}
Let $\btau = (\tau_n \colon \cA_{n+1}^* \to \cA_n^* : n \geq 0)$ be an everywhere growing directive sequence.
We define $B_{n,m}(a) = \{\tau_{n,m}(x) : x \in X_{\btau}^{(m)},\, x_0 = a\}$ for $m > n \geq 0$ and $a \in \cA_m$. 
Then, the following conditions are equivalent:
\begin{enumerate}
    \item there exists a contraction of $\btau$ that is decisive;
    \item for any $(u_n,a_n)_{n\geq0} \in \Add(\btau)$ and $n \geq 0$, the nested sequence of sets
    \[  \left( S^{\sum_{n \leq j \lt m} |\tau_{n,j}(u_j)|} B_{n,m}(a_m) : m \gt n\right) \]
    contains a single element in its intersection;
    \item the sequence of Kakutani--Rohklin covers $$\left(\{ S^k B_{n,m}(a) : a \in \cA_m,\, 0 \leq k < h_m(a),\ m > n \}\right)_{m}$$ generates the topology of $X_{\btau}^{(n)}$ for each $n \ge 0$.
\end{enumerate}
\end{proposition}
\begin{proof}
Throughout this proof, we use the notation $[u \boldsymbol{.} v]_n = \{x \in X_{\btau}^{(n)} : x_{[-|u|, |v|)} = uv\}$ for $u,v \in \cA_{n}^*$ and $n \geq 0$.

Assume first that some contraction $\btau'$ of $\btau$ is decisive.
Note that if Item (2) holds for $\btau'$, then it also holds for $\btau$ by \eqref{eq:Sadic:connection_levels}.
Thus, we may assume without loss of generality that $\btau$ itself is decisive.

Let $\ell_n,r_n \colon \cA_{n+1} \to \cA_n$ be the maps for which $\tau_n(b)$ starts with $r_n(a)$ and $\tau_n(a)$ ends with $\ell_n(b)$ for all $ab \in \cL(X_{\btau}^{(n+1)})$.
Then, $B_{n,n+1}(a) \subseteq [\ell_n(a) \boldsymbol{.} \tau_n(a) \, r_n(a)]$ for any $a \in \cA_{n+1}$.
Therefore,
\begin{equation}
    \label{eq:prop:decisive:dyn_inter:1}
    B_{n,m}(a) \subseteq 
    [\, \tau_{n,m-1}(\ell_n(a)) \boldsymbol{.}
     \tau_{n,m}(a) \, \tau_{n,m-1}(r_n(a))\,]
\end{equation}
for every $m \gt n \geq 0$ and $a \in \cA_m$.

Set $q_{n,m} =\min\{ |\tau_{n,m}(a)| : a \in \cA_m \}$.
\Cref{eq:prop:decisive:dyn_inter:1} implies that for any $a \in \cA_m$, $0 \leq k \lt |\tau_{n,m}(a)|$ and $x,x' \in S^k B_{n,m}(a)$, we have $x_{[-q_{n,m},q_{n,m})} = x'_{[-q_{n,m},q_{n,m})}$. 
Since $\btau$ is everywhere growing, $q_{n,m} \to \infty$ as $m \to \infty$.
Therefore, two points in the same $S^k B_{n,m}(a)$ for infinitely many $m$ must be equal.
Item (2) follows.
\medskip

Next, we assume that Item (2) holds.
A compactness argument shows that for every $n \geq 0$ one can find $m(n) > n$ such that, for all $a \in \cA_{m(n)}$, $B_{n,m}(a) \subseteq [\ell_n(a) \boldsymbol{.} \tau_{n,m}(a) r_n(a)]$ for some letters $\ell_n(a), r_n(a) \in \cA_n$.
This implies that $\tau_{n,m(n)}(b)$ starts with $r_n(a)$ and that $\tau_{n,m(n)}(a)$ ends with $\ell_n(b)$ for every $ab \in \cL(X_{\btau}^{(m(n))})$.
So, with $(n_k)_{k\geq0}$ defined by $n_0 = 0$ and $n_{k+1} = m(n_k)$ for $k \geq 1$, the contraction of $\btau$ along $(n_k)_{k\geq0}$ is decisive.

Items (2) and (3) are clearly equivalent since, for each fixed $n \ge 0$, the covers $\cP_{n,m} = \{ S^k B_{n,m}(a) : a \in \cA_m,\, 0 \leq k < h_m(a)\}$ of $X_{\btau}^{(n)}$ are nested for $m > n$; that is, every element of $\cP_{n,m'}$ is a union of elements of $\cP_{n,m}$ for every $m' \ge m > n$.
\end{proof}

\begin{remark} \label{rem:anyminimal}
\emph{Every} minimal  subshift $X$ admits  an  $S$-adic expansion  with a directive sequence being  everywhere growing, decisive,  and  even recognizable, if $X$  is assumed to be infinite.  Indeed, any minimal subshift has a proper and everywhere growing  $S$-adic representation  by \cite{GPS92} (see also \cite{DDMP,espinoza22}). Recognizability then comes from \cite{BSTY}. However this  $S$-adic representation might be  of infinite alphabet rank. 
\end{remark}
Next proposition  provides a  further  decisive $S$-adic representation in terms of centered  $1$-block presentations, also possibly with   infinite alphabet rank.   We recall that centered $k$-block presentations are introduced in  Section \ref{subsec:block_presentations}. 
\begin{proposition}
    \label{prop:decisive:suff_conds}
Let $\btau = (\tau_n \colon \cA_{n+1}^* \to \cA_n^* : n \geq 0)$ be an everywhere growing directive sequence. 
\begin{enumerate}
    \item If $\btau$ is proper, then it is decisive.
    \item If $|\tau_n(a)| \geq 2$ for all $n \geq 0$ and $a \in \cA_{n+1}$, then the centered $1$-block presentation $\btau^{[1]}$ is decisive.
\end{enumerate}
\end{proposition}

\begin{proof}
Assume $\btau$ is proper, so that there exist letters $\ell_n, r_n \in \cA_n$ such that $\tau_n(a)$ starts with $\ell_n$ and ends with $r_n$ for any $a \in \cA_{n+1}$.  
Fix $n$ and any $ab\in\cL(X_{\btau}^{(n+1)})$.
Then in the concatenation $\tau_n(a)\tau_n(b)$, the last letter of $\tau_n(a)$ is always $r_n$ and the first letter of $\tau_n(b)$ is always $\ell_n$, independently of $a$ and $b$. Thus the boundary symbol at each cut is fixed from both sides, which is exactly decisiveness.

Next, we prove (2).
Let $[a_0a_1a_2][a_1a_2a_3]\in\cL(X_{\btau^{[1]}}^{(n+1)})$, so that $a_0a_1a_2a_3\in\cL(X_{\btau}^{(n+1)})$. 
Write $\tau_n(a_0 a_1 a_2 a_3) = b_0b_1\cdots b_{\ell-1}$.
By  definition, $\tau_n^{[1]}([a_1a_2a_3])$ begins with $[b_{s-1}b_sb_{s+1}]$, where $s = |\tau_n(a_0a_1)|$.  
In particular, $b_{s-1}$ is the last letter of $\tau_n(a_1)$ and $b_s b_{s+1}$ is the length-2 prefix of $\tau_n(a_2)$ (since $|\tau_n(a_2)| \ge 2$). 
Therefore, the first letter $[b_{s-1}b_sb_{s+1}]$ of $\tau_n^{[1]}([a_1a_2a_3])$ depends only on $[a_0 a_1 a_2]$.
A similar argument works at the left-hand side, so $\btau^{[1]}$ is decisive.
\end{proof}

We finish this section with an interpretation of decisiveness in terms of extension graphs.

\begin{lemma}
    \label{lemma:decisive:same_first_letter_in_each_CC}
    Let $\btau=(\tau_n:\cA_{n+1}^*\to\cA_n^*)_{n\ge0}$ be an everywhere growing directive sequence.
    Then, $\btau$ is decisive if and only if, for every $n\ge0$, the map sending $a \in \cA_{n+1}$ to the first (resp.\ last) letter of $\tau_n(a)$ is constant on the right (resp.\ left) vertices of each connected component of $\Gamma_{\!X_{\btau}^{(n+1)}}(\varepsilon)$.
\end{lemma}
\begin{proof}
Assume $\btau$ is decisive and fix $n\ge0$.
If $(a_L,b_R)$ and $(a_L,b'_R)$ are edges with the same left vertex, decisiveness gives that $\tau_n(b)$ and $\tau_n(b')$ start with the same letter.
Along any path between $b_R$ and $b'_R$ the left vertex is shared at each step, so the first letter is preserved; thus it is constant on each connected component of right vertices.
The proof for left vertices $b_L$ and $b'_L$ is symmetric.

Assume that $\btau$ is not decisive.
Then, without loss of generality, there exists a level $n \ge 0$ and words $ab,ab' \in \cL(X_{\btau}^{(n)})$ such that $\tau_n(b)$ and $\tau_n(b')$ do not start with the same letter.
The letters $b_R$ and $b'_R$ being in the same connected component, we deduce that the first letter of $\tau_n(\cdot)$ is  not constant on the right vertices of some connected component of $\Gamma_{X_{\btau}^{(n)}}(\varepsilon)$.
\end{proof}

\section{General characterization of eigenvalues}
\label{sec:carac}

We provide characterizations of eigenvalues
for minimal subshifts described by very general directive sequences, including  subshifts with infinite alphabet rank.
This requires introducing suitable sequences of real numbers $\brho = (\rho_n(a) : n \geq 0,\, a \in \cA_n)$ that allows one  to compensate 
the  oscillatory behaviour of the morphisms $\alpha h_n$ modulo $1$, for $\alpha$  being an eigenvalue.
In \Cref{sec:carac_finite_rank}, we show that in many situations,  under further assumptions on the directive sequence $\bt$, $\brho$ can be replaced by a sequence of letter-coboundaries $(c_n : n \geq 0)$ on the different subshifts $X_{\btau}^{(n)}$ generated by $\btau$,  with the sharpness of 
 our assumptions  on $\bt$  being  proved in \Cref{thm:NoCobsInftyRank}.

\subsection{Statements}\label{subsec:eigenvalues:results}

For $x \in \R$, let $\|x\|$ denote the distance from $x$ to the nearest integer, and define the {\em centered fractional part} $\{x\}$ as the unique element of $[-1/2,1/2)$ such that $x + \{x\}$ is an integer.
We have in particular $|\{x\}| = \|x\|$.

\begin{theorem} \label{theo:EigCharac:Mult}
Let $\btau = (\tau_n \colon \cA_{n+1}^* \to \cA_n^*)_{n\geq0}$ be a primitive and recognizable directive sequence.
A real number $\alpha$ is a continuous eigenvalue of $X_{\btau}$ if and only if there exists  a sequence of real numbers $\brho = (\rho_n(a) : n \geq 0, a \in \cA_n)$ such that
\begin{equation} \label{eq:convergence:theo:EigCharac:Mult}
    \sup \Big\{
    \big\| \rho_n(u_0) - \rho_n(u_k) + 
        \alpha h_n(u_0 \cdots u_{k-1}) \big\| :
        u_0 \cdots u_k \in \cL(X_{\btau}^{(n)})
        \Big\}
\end{equation}
converges to $0$ as $n \to \infty$.
\end{theorem}

With a stronger primitivity assumption on the directive sequence (expressed in terms of a positivity assumption), we can replace \eqref{eq:convergence:theo:EigCharac:Mult} by a simpler criterion that relies only on  bounds that depend  on finite sets  defined  for each   substitution $\tau_n \circ \tau_{n+1}$ independently; we do not have  to consider the language globally $\cL(X_{\btau}^{(n)})$
as in \eqref{eq:convergence:theo:EigCharac:Mult}, whereas  we  sum over a finite set in 
\Cref{eq:theo:EigCharac:positive}.
This refinement is carried out in \Cref{theo:EigCharac:positive} below.
It is important to note that Condition \eqref{eq:theo:EigCharac:positive} in this theorem is strictly stronger than  Condition   \eqref{eq:convergence:theo:EigCharac:Mult} in \Cref{theo:EigCharac:Mult}.

We recall that a   directive sequence $\btau = (\tau_n \colon \cA_{n+1}^* \to \cA_n^*)_{n \geq 0}$ is  \emph{positive} if for all $n \ge 0$, $a \in \cA_n$ and $b \in \cA_{n+1}$, $a$ occurs in $\tau_n(b)$.
We recall that for a substitution $\tau$, the notation $\prefixes(\tau(a))$ below stands for the set of prefixes of $\tau(a)$.

\begin{theorem} \label{theo:EigCharac:positive}
Let $X$ be a subshift generated by a recognizable and positive directive sequence $\btau = (\tau_n \colon \cA_{n+1}^* \to \cA_n^*)_{n \geq 0}$.
Then, a real number $\alpha$ is a continuous eigenvalue of $X$ if and only if there exists a sequence of  real numbers $\brho = (\rho_n(a) : n \geq 0, a \in \cA_n)$ such that
\begin{equation}
    \label{eq:theo:EigCharac:positive} 
    \sum_{n\geq0} \max \big\{
    \|\rho_n (u_0) - \rho_n(u_k)+ \alpha h_n(u_0\cdots u_{k-1}) \| :
    a \in \cA_{n+2},\, u_0\cdots u_k \in \prefixes(\tau_n(\tau_{n+1}(a)))
    \big\}  < \infty.
\end{equation}
\end{theorem}

As observed in the proof of \Cref{theo:EigCharac:positive} (see \Cref{subsec:proofsofsec4}), Condition \eqref{eq:theo:EigCharac:positive} is sufficient for $\alpha$ to be an eigenvalue of $X_{\btau}$ even if one replaces positive by primitive.

In \Cref{theo:EigCharac:proper} below, we show that, both in \Cref{theo:EigCharac:Mult} and \Cref{theo:EigCharac:positive}, if the directive sequence $\btau$ is additionally proper, then the numbers $\rho_n(a)$ can always be chosen to be zero.
The resulting criterion for positive directive sequences is precisely the one obtained in \cite[Theorem 2]{DFM19}.
Thus, \Cref{theo:EigCharac:Mult} and \Cref{theo:EigCharac:positive} can be seen as natural generalizations of \cite[Theorem 2]{DFM19} to the non-proper and non-negative
case.

\begin{theorem} \label{theo:EigCharac:proper}
Let $X_{\btau}$ be a subshift generated by a recognizable and primitive directive sequence $\btau = (\tau_n \colon \cA_{n+1}^* \to \cA_n^*)_{n \geq 0}$.
Suppose that each $\tau_n$ is left-proper.
Then, a real number $\alpha$ is a continuous eigenvalue of $X_{\btau}$ if and only if
\begin{equation} \label{eq:theo:EigCharac:proper}
    \sup \Big\{
    \big\| 
    \alpha h_n(u) \big\| :
    u \in \cL(X_{\btau}^{(n)})
    \Big\}
\end{equation}
converges to $0$ as $n \to \infty$.
If $\btau$ is additionally positive, then $\alpha$ is an additive eigenvalue of $X_{\tau}$ if and only if
\begin{equation} 
    \label{eq2:theo:EigCharac:proper}
    \sum_{n\geq0} \max\big\{
    \| \alpha h_n(u) \| :
    a \in \cA_{n+2}, u \in \prefixes(\tau_n\tau_{n+1}(a))
    \big\}  < \infty. 
\end{equation}
\end{theorem}

From \eqref{eq:theo:EigCharac:proper} we recover the classical necessary condition for $\alpha$ to be an additive eigenvalue of $X_{\btau}$ when $\btau$ is proper (cf.~\cite{Host86,Dekking1978,DFM19,BCBY}):
\begin{equation} \label{eq:EigCharac:necessary:proper}
    \lim_{n\to\infty} 
    \max\{ \|\alpha h_n(a)\| : a \in \cA_n \} = 0.
\end{equation}

From \Cref{theo:EigCharac:Mult} and Item (2) in \Cref{lem:straight_rho} we also recover the following characterization of eigenvalues in terms of return words (cf. \cite{Ferenczi-Mauduit-Nogueira}, \cite[Proposition 4.5]{BCBY}).

\begin{theorem} \label{theo:EigCharac:mot_retour}
    Let $X_{\btau}$ be a subshift generated by a recognizable and primitive directive sequence $\btau = (\tau_n \colon \cA_{n+1}^* \to \cA_n^*)_{n \geq 0}$.
    A real number $\alpha$ is a continuous eigenvalue of $X_{\btau}$ if and only if
    \[  \sup \Big\{
        \big\| 
        \alpha h_n(u_0 \cdots u_{k-1}) \big\| :
        u_0 \cdots u_{k-1} u_0 \in \cL(X_{\btau}^{(n)})
        \Big\}
    \]  
    converges to $0$ as $n \to \infty$.
\end{theorem}

\subsection{General strategy} \label{subsec:strategy}

It is convenient to first introduce the following notation, which is used throughout the lemmas and proofs.
We recall that $\First(u)$ stands for the first letter of a non-empty word $u$.
Given a sequence of  real numbers $\brho = (\rho_n(a) : n \ge 0,\, a \in \cA_n)$ and $\alpha \in \R$, we let
\begin{equation}
	\label{defi_eps_nm}
    \varepsilon_{n,m}(\brho) = \max \big\{
    \|\rho_n(\First(ub)) + \alpha h_n(u) - \rho_n(b)\| : 
    a \in \cA_m,\, ub \in \prefixes(\tau_{n,m}(a))
    \big\}
\end{equation}
for $m > n \ge 0$, and 
\begin{equation}
	\label{defi_eps_n}
    \varepsilon_n(\brho) = \sup \big\{
        \|\rho_n(\First(ub)) + \alpha h_n(u) - \rho_n(b)\| : 
        ub \in \cL(X_{\btau}^{(n)})
    \big\}
\end{equation}
for $n \ge 0$.
The dependence of $\varepsilon_{n,m}(\brho)$ and $\varepsilon_n(\brho)$ on $\btau$ and $\alpha$ will always be clear from the context.
Note that $\varepsilon_n(\brho)$ is the quantity appearing in \eqref{eq:convergence:theo:EigCharac:Mult}, while $\varepsilon_{n,n+2}$ is the $n$-th term of the sum in \eqref{eq:theo:EigCharac:positive}.
We also remark that 
\begin{equation}
	\label{trivial_bound_eps_n_ge_eps_nm}
	\varepsilon_{n,m}(\brho) \leq \varepsilon_n(\brho)
    \leq 2 \sup_{\ell \gt n} \varepsilon_{n,\ell}(\brho)
    \enspace \text{for all $m > n \ge 0$.}
\end{equation}

The  sequence  $\brho$  in \Cref{theo:EigCharac:Mult} and \Cref{theo:EigCharac:positive} is {\em a priori} used to  approximate  the values taken by an eigenfuction $g$   on  the  basis $B_n(a)$ of the Kakutani--Rohklin towers ${\mathcal P}_n$ defined in \Cref{def:recognizable}, \emph{i.e.},   
for a point $x \in B_n(a)$, with $n \ge 0$, one  has $g(x) \approx \rho_n (a)$.  The values  on the other levels are given by using the eigenvalue relation $g(S^k(x)) - g(x) \equiv  k \alpha \mod \Z$. 
Suppose now that we want to test  whether  some $\alpha $ is an eigenvalue.
It  is  then natural to  consider  the quantities involved in the definition of   $\varepsilon_n (\brho)$ in \eqref{defi_eps_n}.
Indeed, if $ x \in  X_{\btau}^{(n)}$, then 
$\rho_n(x_0) \approx g(\tau_{0,n}(x))$, 
whereas  $$\rho_n(x_k) \approx  g( S^{h_n(x_0 \cdots x_{k-1})}\tau_{0,n}(x)) \equiv 
\alpha h_n(x_0 \cdots x_{k-1})+  g(\tau_{0,n}(x)),$$  which translates into
$$\rho_n(x_0) - \rho_n(x_k) + \alpha h_n(x_0 \cdots x_{k-1}) \approx 0. $$ 

In other words, if  $\brho$ is a good approximation of $g$, then  $\varepsilon_{n,m}(\brho)$, and thus $\varepsilon_n(\brho)$ by \eqref{trivial_bound_eps_n_ge_eps_nm}, is small.

A simple  lower bound on how small we expect $\varepsilon_{n,m}(\brho)$ to be for a given $\alpha$ is  first given by \Cref{lem:rho_controls_retwords} below, in terms of return words.
For a directive sequence $\btau = (\tau_n \colon \cA_{n+1}^* \to \cA_n^* : n \ge 0)$, let $R_{n,m}(\btau)$ be the set of non-empty words $u \in \cA_n^*$ such that $u \, \First(u)$ occurs in $\tau_{n,m}(a)$ for all $a \in \cA_m$.
In other words, this is the set of  concatenations of return words to letters that occur in every $\tau_{n,m}(a)$.
Observe that the union   $\cup_{m < n} R_{n,m} (\btau)$ is equal to
the set of all  concatenations of return words 
in $X_{\btau}^{(n)}$.




\begin{lemma}
\label{lem:rho_controls_retwords}
Let $\btau = (\tau_n \colon \cA_{n+1}^* \to \cA_n^*)_{n\geq 0}$ be a primitive directive sequence and $\alpha \in \R$.
For $n \ge 0$, let $r(n) > n$ be the least integer such that every letter $a \in \cA_n$ appears in $\tau_{n,r(n)}(b)$ for every $b \in \cA_{r(n)}$.
Denote 
\begin{equation}
    \label{eq:defi_delta_ret_words}
    \delta_{n,m}(\btau) = \max\big\{ \|\alpha h_n(u)\| : u \in R_{n,r(m)}(\btau) \big\}
    \enspace \text{for $m > n \ge 0$.}
\end{equation}
Then, for any sequence $\brho = (\rho_n(a) : n \ge 0,\, a \in \cA_n)$, we have
\begin{equation}
    \label{eq:rho_controls_retwords}
    \delta_{n,r(m)}(\btau) \leq 2\varepsilon_{n,r(m)}(\brho)
    \enspace \text{for all $n \ge 0$, $m \ge r(n)$.}
\end{equation}
\end{lemma}
\begin{proof}
Let $\brho=\bigl(\rho_n(a): n\ge 0,\, a\in \cA_n\bigr)$ be arbitrary, and fix $n\ge 0$ and $m\ge r(n)$. 
Consider $u\in R_{n,r(m)}(\btau)$ and set $b=\First(u)$.
Since $ub$ occurs in $\tau_{n,r(m)}(a_{r(m)})$, there exists a prefix $v$ of $\tau_{n,r(m)}(a_{r(m)})$ such that $v u b\in \prefixes\bigl(\tau_{n,r(m)}(a_{r(m)})\bigr)$.
By the definition of $\varepsilon_{n,r(m)}$, one has
\[
\bigl\|\rho_n(a_n)+\alpha\, h_n(v)-\rho_n(b)\bigr\|\le \varepsilon_{n,r(m)}
\quad\text{and}\quad
\bigl\|\rho_n(a_n)+\alpha\, h_n(vu)-\rho_n(b)\bigr\|\le \varepsilon_{n,r(m)}.
\]
Subtracting the two expressions inside and applying the triangle inequality gives $\|\alpha h_n(u)\|\le 2\varepsilon_{n,r(m)}$.
\end{proof}

The fact that the sequence $\brho$ approximates an eigenfunction $g$ imposes further structural constraints on it.  
In particular, one must take into account the relationship between the bases of the partition at different levels. 
Since  $$B_m(a) \subseteq B_n (\First(\tau_{n,m}(a))) \mbox{ for any }m \gt n,$$ then the eigenfunction $g$ satisfies that $g(B_m(a))$ is close to $g(B_n (\First(\tau_{n,m}(a))))$.
This translates into $\rho_m(a) \approx \rho_n(\First(\tau_{n,m}(a)))$.
Note  that \eqref{eq:convergence:theo:EigCharac:Mult} alone is not sufficient to guarantee  such  a  constraint
on $\brho$. 
However, \Cref{lem:straight_rho} below  shows that it is always possible to construct a sequence  $\brhobar$ for which both the structural constraints relating  various levels
(see Item (1)) and the  quantities $\varepsilon_{n,m}(\brhobar)$  (see Item (2)) are controlled by return words. 
As a consequence, by \Cref{lem:rho_controls_retwords}, this sequence $\brhobar$ is optimal with respect to the convergence in \eqref{eq:convergence:theo:EigCharac:Mult}, in the sense that $\varepsilon_{n,m}(\brhobar) \leq 4\varepsilon_{n,r(m)}(\brho)$ for any given sequence $\brho$ (see \Cref{eq:compare} below). In other words,  the optimality of the sequence 
$\brhobar$ allows one to inherit convergence properties   
on  the sequence $\brhobar$  from  corresponding 
convergence properties on  a  sequence $\brho$.

\begin{lemma}
\label{lem:straight_rho}
Let $\btau = (\tau_n \colon \cA_{n+1}^* \to \cA_n^*)_{n\ge0}$ be a primitive directive sequence, $\alpha \in \R$.
Consider the notation $r(n)$ and $\delta_{n,m}(\btau)$ from \Cref{lem:rho_controls_retwords}.
There exists a sequence $\brhobar = (\bar{\rho}_n(a) : n \ge 0,\, a \in \cA_n)$ satisfying the following:
\begin{enumerate}
    \item $\|\bar{\rho}_m(b) - \bar{\rho}_n(\First(\tau_{n,m}(b)))\| \leq \delta_{n,r(m)}(\btau)$ for all $n \ge 0$, $m \ge r(n)$ and $b \in \cA_m$;
    \item $\|\bar{\rho}_n(v_0) + \alpha h_n(v_0 \cdots v_{k-1})- \bar{\rho}_n(v_k)\| \leq 2\delta_{n,r(m)}(\btau)$ for all $n \ge 0$, $m \ge r(n)$ and $v_0\cdots v_k$ occurring in $\tau_{n,m}(b)$ for some $b \in \cA_m$.
\end{enumerate}
\end{lemma}
\begin{proof}
We start by defining $\brhobar$.
Using \Cref{lem:existence_limit_points} we can find letters $(a_n)_{n\ge 0}$ with $\tau_n(a_{n+1})$ beginning with $a_n$ for every $n\ge 0$.
By the definition of $r(n)$, for each $n\ge 0$ and $b\in \cA_n$ there exists a word $u_n(b)$ such that $u_n(b)\,b\in \prefixes(\tau_{n,r(n)}(a_{r(n)}))$. 
As $\tau_{n,r(n)}(a_{r(n)})$ starts with $a_n$, we can take $u_n(a_n)$ to be the empty word.
Define $\bar\rho_n(b) = \alpha h_n\bigl(u_n(b)\bigr)$.
Remark that $\bar\rho_n(a_n) = 0$ for every $n \ge 0$.

Next, we prove Item (1).
Fix $n\ge 0$, $m\ge r(n)$, $a\in \cA_m$, and a prefix $u b\in \prefixes\bigl(\tau_{n,m}(a)\bigr)$.
Both $u_n(b)\,b$ and $\tau_{n,m}(u_m(a))\,u\,b$ are prefixes of $\tau_{n,r(m)}(a_{r(m)})$, with the first one shorter; hence there exists a word $w$ starting with $b$ such that
\[
\tau_{n,m}(u_m(a))\,u\,b = u_n(b)\,w\,b.
\]
In particular, $w \in R_{n,r(m)}(\btau)$, therefore
\[
\bigl\|\alpha\, h_n\bigl(\tau_{n,m}(u_m(a)\,u)\bigr)-\alpha h_n(u_n(b))\bigr\|
=\bigl\|\alpha\, h_n(w)\bigr\|
\le \delta_{n,r(m)}(\btau).
\]
Using $h_n\bigl(\tau_{n,m}(u_m(a)\,u)\bigr) = h_m\bigl(u_m(a)\bigr)+h_n(u)$ and $\bar\rho_n(b) = \alpha h_n(u_n(b))$ (by definition), the last inequality rewrites as
\begin{equation}
\label{eq:1:straight_rho}
\bigl\|\bar\rho_m(a)+\alpha\, h_n(u)-\bar\rho_n(b)\bigr\|\le \delta_{n,r(m)},
\end{equation}
which holds for every $n \ge 0$, $m \ge r(n)$, $a \in \cA_m$ and $u b\in \prefixes\!\bigl(\tau_{n,m}(a)\bigr)$.
Specializing \eqref{eq:1:straight_rho} to $ub$  being the length-1 prefix of $\tau_{n,m}(a)$, one gets  
\[
\bigl\|\bar\rho_m(a)-\bar\rho_n\bigl(\First(\tau_{n,m}(a))\bigr)\bigr\|\le \delta_{n,r(m)},
\]
proving Item (1).
Moreover, if $v = v_0\cdots v_k$ occurs in $\tau_{n,m}(a)$, then $v' v \in \prefixes(\tau_{n,m}(a))$ for some $v' \in \cA_n^*$, so \eqref{eq:1:straight_rho} gives that both
\[  \bigl\|  \bar\rho_m(a) + \alpha h_n(v') - \bar\rho_n(v_0)   \bigr\|
    \enspace \text{and} \enspace 
    \bigl\| \bar \rho_m(a) + \alpha h_n(v' \, v_0\cdots v_{k-1}) - \bar\rho_n(v_k)   \bigr\|
\]
are bounded by $\delta_{n,r(m)}$.
Hence, by taking differences,
\[  \bigl\| \bar \rho_n(v_k) + \alpha h_n(v_0\cdots v_{k-1}) - \rho_n(v_0)   \bigr\| \le 
    2 \delta_{n,r(m)},
\]
which proves Item (2).
\end{proof}

One useful consequence of \Cref{lem:rho_controls_retwords} is that one can control the sequence $\brhobar$ constructed in \Cref{lem:straight_rho} using any  given sequence $\brho$:
by Item (2) in \Cref{lem:straight_rho} and \eqref{eq:rho_controls_retwords}, we have
\begin{equation} \label{eq:compare}
    \varepsilon_{n,m}(\brhobar) \leq 2\delta_{n,m}(\btau) \leq 4\varepsilon_{n,r(m)}(\brho)
\end{equation}
for any $n \ge 0$ and $m \ge r(n)$, and by Item (1) in \Cref{lem:straight_rho} and \eqref{eq:rho_controls_retwords},
\begin{equation}
    \|\bar{\rho}_m(a) - \bar{\rho}_n(\First(\tau_{n,m}(a)))\| \leq 
    \delta_{n,m}(\btau) \leq 2\varepsilon_{n,r(m)}(\brho)
\end{equation}
for any $n \ge 0$, $m \ge r(n)$ and $a \in \cA_m$.

\subsection{\texorpdfstring{Proofs of \Cref{theo:EigCharac:Mult}, \Cref{theo:EigCharac:positive} and \Cref{theo:EigCharac:proper}}{}}\label{subsec:proofsofsec4}
We now prove the main results of this section.

\begin{proof}[Proof of \Cref{theo:EigCharac:Mult}]
Suppose that $\alpha$ is an eigenvalue of $X_{\btau}$.
Let $g \colon X_{\btau} \to \R/\Z$ be an eigenfunction associated to $\alpha$.
We  first define a sequence  $\brho$ by taking any value $\rho_n(a) \in g(B_n(a))$, for $n \ge 0$ and $a \in \cA_n$.
We are going to show that $\varepsilon_n(\brho)$ converges to $0$ as $n \to \infty$.
The first step is to show that  the sequence  $(\diam(g(B_n(a)))_n$ converges to $0$, where $\diam$ denotes the diameter of a subset of $\R/\Z$.

For $n \ge 0$ and $a \in \cA_n$, let $r_n(a) \coloneqq \lfloor h_n(a)/2 \rfloor$. 
Then any $x,y \in S^{r_n(a)} B_n(a)$ agree on a window of length $2r_n(a)$ centered at $0$.
Since $\btau$ is primitive, $\min\{r_n(a) : a \in \cA_n\} \to \infty$ as $n \to \infty$, so, by the uniform continuity of $g$,
\[
\varepsilon'_n \coloneqq \sup\bigl\{ \diam\bigl(g(S^{r_n(a)} B_n(a))\bigr) : a \in \cA_n \bigr\}
\]
converges to $0$ as $n \to \infty$.
Now, $g(S^k B_n(a)) = \{g(x) + k \alpha : x \in B_n(a)\}$ and the map $z \mapsto z + \alpha$ is an isometry of $\R/\Z$.
Therefore,
\begin{equation}
\label{eq:1:theo:EigCharac:Mult}
\varepsilon'_n = \sup\bigl\{ \diam\bigl(g(B_n(a))\bigr) : a \in \cA_n \bigr\}.
\end{equation}

We now prove that $\varepsilon_n(\brho) \leq 2\varepsilon'_n$, which would prove $\varepsilon_n(\brho) \to 0$ as $n \to \infty$.
Let $n \ge 0$, $ub \in \cL(X_{\btau}^{(n)})$ with $b \in \cA$, and let $a$ denote the first letter of $ub$.
Then there exists $y \in X_{\btau}^{(n)}$ such that $y_{[0,|u|)} = u$.
We set $x \coloneqq \tau_{0,n}(y)$ and observe that $x \in B_n(a)$ and $S^{h_n(u)} x \in B_n(b)$.
By \eqref{eq:1:theo:EigCharac:Mult}, $\|g(x) - \rho_n(a)\| \leq \varepsilon'_n$ and $\|g(S^{h_n(u)} x) - \rho_n(b)\| \leq \varepsilon'_n$.
Since $g(S^{h_n(u)} x) = x + \alpha h_n(u)$, we obtain
\[	\|\rho_n(a) + \alpha h_n(u) - \rho_n(b)\| \leq 2\varepsilon'_n,	\]
which holds for every $n \ge 0$ and $ub \in \cL(X_{\btau}^{(n)})$.
We conclude that $\varepsilon_n(\brho) \leq 2\varepsilon'_n$, and thus that $\varepsilon_n(\brho) \to 0$ as $n \to \infty$.
\medskip

Conversely, let $\brho = (\rho_n(a) : n \ge 0,\, a \in \cA_n)$ be a family of real numbers satisfying $\varepsilon_n(\brho) \to 0$ as $n \to \infty$.
We will now  construct an eigenfunction $g$ as the limit of continuous maps $g_n \colon X \to \R/\Z$.  A  priori    they  could be    defined from the coefficients $(\rho_n(a) : a \in \cA_n)$.
However, to ensure that the limit $g$ is continuous, we need the family $(g_n : n \ge 0)$ to be equicontinuous.
Establishing this equicontinuity requires estimates across levels of $\btau$, so we work instead with the optimal sequence $\brhobar$ from \Cref{lem:straight_rho}.

Let $\brhobar = (\bar{\rho}_n(a) : n \ge 0,\, a \in \cA_n)$ be given by \Cref{lem:straight_rho}.
Define $g_n \colon X_{\btau} \to \R/\Z$ by $g_n(x) = \bar{\rho}_n(a) + k\alpha$ for $x \in S^k B_n(a)$, where $a \in \cA_n$ and $0 \le k < h_n(a)$.
Since $\btau$ is recognizable, the collection $\{S^k B_n(a) : a \in \cA_n,\, 0 \le k < h_n(a)\}$ is a partition of $X$ into clopen sets, hence $g_n$ is well-defined and continuous.

We now show that $\{g_n : n \ge 0\}$ is equicontinuous.
Fix $n \ge 0$, $m \ge r(n)$, and $x,x' \in X$ lying in the  atom $S^k B_n(a)$, for some $a \in \cA_n$ and $0 \le k < h_n(a)$.
It suffices to show that $\| g_m(x) - g_m(x') \| \le 12 \varepsilon_n(\brho)$.
Since $\|g_m(x) - g_m(x')\| = \| g_m(S^{-k}x) - g_m(S^{-k}x') \|$, we may assume $k = 0$, \emph{i.e.}, $x,x' \in B_n(a)$.

Choose $b \in \cA_m$ and $0 \le \ell < h_m(b)$ with $x \in S^\ell B_m(b)$.
Then, there exists a prefix $v_0 \dots v_k \in \prefixes(\tau_{n,m}(b))$ with $v_k = a$ and 
$h_n(v_0 \dots v_{k-1}) = \ell$.
By Item (2) of \Cref{lem:straight_rho},
\begin{equation}
\label{eq:2.0:theo:EigCharac:Mult}
    \bigl\|\bar\rho_n(v_0) + \alpha\ell - \bar\rho_n(a)\bigr\|
    = \bigl\|\bar\rho_n(v_0) + \alpha h_n(v_0 \dots v_{k-1}) - \bar\rho_n(v_k)\bigr\|
    \le 2 \delta_{n,r(m)}(\btau).
\end{equation}
Since $v_0$ is the first letter of $\tau_{n,m}(b)$, Item (1) of \Cref{lem:straight_rho} gives
$\|\bar\rho_m(v_0) - \bar\rho_m(b)\| \le \delta_{n,r(m)}(\btau)$.
Substituting this into \eqref{eq:2.0:theo:EigCharac:Mult} yields
\[
   \bigl\|\bar\rho_m(b) + \alpha\ell - \bar\rho_m(a)\bigr\|
   \le 3 \delta_{n,r(m)}(\btau).
\]
By definition, $g_m(x) = \bar\rho_m(b) + \alpha\ell$, so $g_m(x)$ lies within 
$3 \delta_{n,r(m)}(\btau)$ of $\bar\rho_m(a)$.
A similar argument shows the same for $g_m(x')$.
Thus,
\[
   \|g_m(x) - g_m(x')\| \le 6 \delta_{n,r(m)}(\btau).
\]

By \Cref{lem:rho_controls_retwords}, 
$\delta_{n,r(m)}(\btau) \le 2 \varepsilon_{n,r(m)}(\brho)$, and since
$\varepsilon_{n,r(m)}(\brho) \le \varepsilon_n(\brho)$ by definition,
we obtain the desired bound
\[
   \|g_m(x) - g_m(x')\| \le 12 \varepsilon_n(\brho).
\]

\smallskip 

We have shown that $(g_n : n \ge 0)$ is an equicontinuous family.
By the Arzelà–Ascoli Theorem, there is a convergent subsequence $(g_{n_k} : k \ge 0)$ converging to a continuous map $g \colon X \to \R/\Z$.
We show next that $g$ is an eigenfunction.

Let $x \in X$, with $\btau$-address $(u_n,a_n)_{n\ge0}$ (see \Cref{def:address}), and set $p_{n_k} = h_{n_k}(u_{n_k})$ for $k \ge 0$.
We check that $g(Sx) = g(x) + \alpha$ by considering two cases.
\begin{itemize}
\item
If there are infinitely many $n \in \{n_k : k \ge 0\}$ such that $p_n + 1 < h_n(a_n)$, then the definition of $g_n$ ensures that $g_n(S x) = g_n(x) + \alpha$ for such values of $n$, thus $g(Sx) = g(x) + \alpha$ by taking the limit.
\item
Assume that $p_n + 1 = h_n(a_n)$ for all large enough $n \in \{n_k : k \ge 0\}$.
For any such $n$, there exist letters $a'_n \in \cA_n$ such that $a_n a'_n \in \cL(X_{\btau}^{(n)})$ and $Sx \in B_n(a'_n)$.
In particular, $g_n(Sx) = \rho_n(a'_k)$, which, together with $p_n = h_n(a_n) - 1$, gives
\[  g_n(Sx) - g_n(x) - \alpha  = 
    \bar{\rho}_n(a'_n) - \bar{\rho}_n(a_n) - \alpha h_n(a_n).  \]
The last term is bounded by $4\varepsilon_n(\brhobar)$ by Item (2) of \Cref{lem:straight_rho} and \Cref{lem:rho_controls_retwords}.
Therefore, $g(Sx) = g(x) + \alpha$ by taking the limit.
\end{itemize}
We conclude that $g$ is a continuous eigenfunction of $X_{\btau}$ for the eigenvalue $\alpha$.
\end{proof}


We continue with the proof of \Cref{theo:EigCharac:positive}.
A key ingredient is the following lemma, which adapts the argument used in the proof of \cite[Theorem 3]{Bressaud-Durand-Maass:2005}.
We recall that the set $R_{n,m}(\btau)$ consists of all non-empty words $u \in \cA_n^*$ such that $u \, \First(u)$ occurs in $\tau_{n,m}(a)$ for every $a \in \cA_m$.

\begin{lemma}
\label{lem:EigCharac:positive:returnWords}
Let $\btau = (\tau_n \colon \cA_{n+1}^* \to \cA_n^* : n \ge 0)$ be a recognizable, positive directive sequence.
Assume that $\alpha$ is an eigenvalue of $X_{\btau}$.
Then, for every $k \ge 1$,
\begin{equation}  
\label{eq:theo:EigCharac:positive:returnWords} 
    \sum_{n \ge 0} \delta_{n,n+k}(\btau) < \infty,
\end{equation}
with 
$\delta_{n,m} = 
\max\big\{ \|\alpha  h_n(u)\| : u \in R_{n,r(m)}(\btau) \big\}$  being
the same quantity defined in \eqref{eq:defi_delta_ret_words} in \Cref{lem:rho_controls_retwords}.
\end{lemma}
\begin{proof}
It is enough to prove \eqref{eq:theo:EigCharac:positive:returnWords} with $k = 1$, since then we can apply this particular case to the $k$ different contractions of $\bt$ along $(nk+i)_{n\geq 1}$, for $0\leq i\leq k-1$
So, assume that $k = 1$ and set $\delta_n = \delta_{n,n+1}$ for $n \ge 0$.

We argue by contradiction and assume that $(\delta_n)_{n\geq0}$ has infinite sum.
Let us start the proof with some simplifications.
Note that there exists a sequence $(n_k)_{k\geq0}$ of positive integers such that $n_{k+1} \geq n_k + 2$ and  $(\delta_{n_k})_{k\geq0}$ has infinite sum.
Also, we can find $a_k \in \cA_{n_k+1}$ and $u_k \in R_{n_k,r(n_k+1)}(a_k)(\bt)$ such that $\delta_{n_k} = \|\alpha h_{n_k}(u_k)\|$.
For $x \in \R$, we recall that  $\{x\}$ is  the unique real number in $[-1/2,1/2)$ such that $\| x \| = |\{ x \}|$.
Then, up to taking a subsequence of $(n_k)_k$, we may assume all the numbers $\{\alpha h_{n_k}(u_k)\}$ are nonnegative or all of them are nonpositive.
Let us assume that all of them are nonnegative; the other case can be handled in a similar way.
\smallskip

Next, 
we construct two special $\btau$-addresses $(v_n,c_n)_{n\geq0}$ and $(v'_n,c'_n)_{n\geq0}$ (see \Cref{def:address}).
In a first step, we define the $n$-th term of the addresses for $n \in \{n_k : k \geq 0\}$; then, for $n \geq 0$ such that $n_k + 1 < n < n_{k+1}$ for some $k \geq 0$; finally, for $n \in \{n_k + 1 : k \geq 0\}$.
Since $n_{k+1} \geq n_k + 2$, the three steps do not overlap and cover every $n \geq 0$, which ensures that the addresses will be well-defined.

Let $n = n_k$, with $k \geq 0$.
\begin{enumerate}
\item 
Let us  handle the first step.
Since $u_k \in R_{n_k}(a_k)$, we have that $u_k$ starts with a letter $b_k$ for which there exist prefixes $v_{n_k} b_k, v'_{n_k} b_k$ of $\tau_{n_k}(a_k)$ such that $v'_{n_k} u_k = v_{n_k}$.
This defines $v_{n_k}$ and $v'_{n_k}$.
We also set $c_{n_k} = c'_{n_k} = b_k$.
\item
For the second step, we take $k \geq 0$ and inductively define, for $n_k + 1 < n < n_{k+1}$, $c_n,c'_n$ as the first letter of $\tau_n(c_{n+1})$ and $v_n,v'_n$ as the empty word.
\item
Finally, for $n = n_k + 1$ with $k \geq 0$, we note that since $\btau$ is positive, $a_k$ occurs in every $\tau_n(c)$, $c \in \cA_{n_k+2}$.
Hence, we can put $c_n = c'_n = a_{n_k}$ and let $v_n, v'_n$ to be such that $v_n a_k = v'_n a_k \in \prefixes(\tau_n(c_{n+1}))$.
\end{enumerate}
Let us show that $(v_n,c_n)_{n\geq0}$ is indeed a $\btau$-address.
If $n \geq 0$ satisfies $n_k < n < n_{k+1}$ for some $k \geq 0$, then $v_n$ and $c_n$ were defined in Step (2); thus, it directly follows that $v_n c_n \in \prefixes(\tau_n(c_{n+1}))$.
Consider next the case $n = n_k$ for some $k \geq 0$.
Then, $c_n = b_k$ and $v_n = u_k$, where $u_k$ and $b_k$ satisfy $u_k b_k \in \prefixes(\tau_n(a_k))$.
Hence, $v_n c_n \in \prefixes(\tau_n(a_k))$.
Now, since $n_{k+1} \geq n_k + 2$, $c_{n+1}$ was defined in Step (3), so $c_{n+1} = a_k$.
We obtain $v_n c_n \in \prefixes(\tau_n(c_{n+1}))$.
Therefore, $(v_n,c_n)_{n\geq0}$ is a $\btau$-address.
A similar argument shows that $(v'_n,c'_n)_{n\geq0}$ is a $\btau$-address as well.
\smallskip

The final part of the proof is to use \Cref{theo:EigCharac:Mult} with the addresses that we have constructed to get a contradiction.

Let $\First_{n,m}(a) = \First(\tau_{n,m}(a))$ for $m > n \ge 0$ and $a \in \cA_m$, and define 
$\eta_{n,m} = \rho_n(\First_{n,m}(c_m)) + \alpha h_m(v_m) - \rho_n(c_n)$ and 
$\eta'_{n,m} = \rho_n(\First_{n,m}(c'_m)) + \alpha h_m(v'_m) - \rho_n(c'_n)$.
If follows from the construction that $c_n = c'_n$ for every $n \ge 0$, $v_n = v'_n$ if $n \notin \{n_k : k \ge 0\}$, and $v_{n_k} = u_{n_k} \, v_{n_k}$ for every $k \ge 0$.
Therefore, $\eta_{n,m} - \eta'_{n,m} = 0$ if $m \notin \{n_k : k \ge 0\}$ and $\eta_{n,n_k} - \eta'_{n,n_k} = \alpha h_{n_k}(u_k)$ for $k \ge 0$, yielding
\begin{equation}
    \label{proof:lem:EigCharac:positive:returnWords:1}
    \sum_{n_k \leq m < n_\ell} \!\!
    \eta_{n_k,m} - \eta'_{n_k,m} =
    \sum_{k \leq i < \ell} 
    \alpha h_{n_i}(u_i)
    \quad \text{for any $\ell > k \geq 0$}.
\end{equation}

We now relate the quantities $\eta_{n,m}$ with the hypothesis on $\alpha$.
Since $\alpha$ is an eigenvalue of $X_{\btau}$, we can apply \Cref{theo:EigCharac:Mult}, which gives a family $\brho = (\rho_n(a) : n \ge 0,\, a\in\cA_n)$ such that $\varepsilon_n(\brho)$ tends to 0 as $n \to \infty$,  with $\varepsilon_n(\brho)$ being defined in   (\ref{defi_eps_n}).
Then, we  get
\begin{equation}
    \label{proof:lem:EigCharac:positive:returnWords:2}
    \sum_{n_k \leq m < n_\ell} \!\! \eta_{n_k,m} = 
    \rho_{n_k}(\First_{n_k,n_\ell}(c_{n_\ell})) + \alpha h_{n_k}(\boldsymbol{v}_{n_k,n_\ell}) - \rho_{n_k}(c_{n_k}),
\end{equation}
where 
\[  \boldsymbol{v}_{n,m} =
    \tau_{n,m}(v_m) \,  \tau_{n,m-1}(v_{m-1}) \dots 
    \tau_n(v_{n+1}) \, v_n \, c_n. \]
Since $(v_n,c_n)_{n\geq0}$ is a $\btau$-address, $\boldsymbol{v}_{n_k,n_\ell}$ is in the language of $X_{\btau}^{(n_k)}$.
This permits to bound \eqref{proof:lem:EigCharac:positive:returnWords:2}  as
\[  \big\| \sum_{n_k \leq m < n_\ell} \!\! \eta_{n_k,m} \big\| \leq 
    \varepsilon_{n_k}(\brho). \]
A similar argument shows that $\varepsilon_{n_k}(\brho)$ also  provides an upper bound for  $\|\sum_{n_k \leq m < n_\ell} \eta'_{n_k,m}\|$. Hence, by \eqref{proof:lem:EigCharac:positive:returnWords:1},
\[  \big\| \!
    \sum_{k \leq i < \ell} 
    \alpha h_{n_i}(u_i) \, \big\| 
    \leq 
    \big\|\sum_{n_k \leq m < n_\ell} \!\! \eta_{n_k,m}\, \big\| + 
    \big\|\sum_{n_k \leq m < n_\ell} \!\! \eta'_{n_k,m}\, \big\|
    \leq 2 \varepsilon_{n_k}(\brho)
    \]
for all $\ell > k \geq 0$.
In particular, since $\varepsilon_n(\brho)$ converges to $0$ as $n$ tends to infinity, there exists $\tilde{k} \geq 0$ such that $\| \sum_{\tilde{k} \leq i < \ell} \alpha h_{n_i}(u_i) \| < 1/16$ for every $\ell > \tilde{k}$.
This implies that 
\[  \sum_{\tilde{k} \leq i < \ell} \{ \alpha h_{n_i}(u_i) \} = 
        \Bigl\{\sum_{\tilde{k} \leq i < \ell} \alpha h_{n_i}(u_i) \Bigr\}.
\]
Hence, as the numbers $\{ \alpha h_{n_i}(u_i) \}$ are nonnegative, it follows that
    \[  \sum_{\tilde{k} \leq i < \ell} \{ \alpha h_{n_i}(u_i) \} = 
        \Bigl|\Bigl\{ \sum_{\tilde{k} \leq i < \ell} \alpha h_{n_i}(u_i) \Bigr\} \Bigr| =
         \big\| \sum_{\tilde{k} \leq i < \ell} \alpha h_{n_i}(u_i) \ \big\| 
         \leq 1/8  \quad \text{for all $\ell > \tilde{k}$.} \]
But $\delta_{n_i} = \{ \alpha h_{n_i}(u_i) \}$, contradicting the fact  that $(\delta_{n_i})_{i\geq0}$ has infinite sum.
\end{proof}

We can now prove \Cref{theo:EigCharac:positive}.

\begin{proof}[Proof of \Cref{theo:EigCharac:positive}.]
Let $\brho = (\rho_n(a) : n \ge 0,\, a \in \cA_n)$ be a sequence such that $(\varepsilon_{n,n+2}(\brho))_{n \ge 0}$ has finite sum.
Denote by $\brhobar = (\bar{\rho}_n(a) : n \ge 0,\, a \in \cA_n)$ the sequence provided by \Cref{lem:straight_rho}.
Since $\btau$ is positive, we have $r(n) = n+1$ for all $n \ge 0$ in \Cref{lem:rho_controls_retwords}, and therefore $\delta_{n,n+2}(\btau) \le 2 \varepsilon_{n,n+2}(\brho)$ for all $n$.
Combined with \Cref{lem:straight_rho}, this yields:
\begin{enumerate}[label=(\roman*)]
    \item $\|\bar{\rho}_m(b) - \bar{\rho}_n(\First(\tau_{n,m}(b)))\| \le 2\varepsilon_{n,n+2}(\brho)$ for all $n \ge 0$ and $b \in \cA_{n+2}$;
    \item $\|\bar{\rho}_n(v_0) + \alpha h_n(v_0 \cdots v_{k-1}) - \bar{\rho}_n(v_k)\|
           \le 4\varepsilon_{n,n+2}(\brho)$ 
          for all $n \ge 0$ and all $v_0\cdots v_k$ occurring in $\tau_{n,n+2}(b)$ for some $b \in \cA_m$.
\end{enumerate}

We now show that
\begin{equation}
\label{eq:1:theo:EigCharac:positive}
    \varepsilon_{n,m}(\brhobar)
    \le 6 \sum_{i \ge n} \varepsilon_{i,i+2}(\brho)
    \qquad\text{for all } m > n \ge 0.
\end{equation}
By \eqref{trivial_bound_eps_n_ge_eps_nm}, this implies $\varepsilon_n(\brhobar) \le 12 \sum_{i \ge n} \varepsilon_{i,i+2}(\brho)$, which tends to $0$ since $(\varepsilon_{n,n+2}(\brho))$ has finite sum.
Thus, once \eqref{eq:1:theo:EigCharac:positive} is established, \Cref{theo:EigCharac:Mult} implies that $\alpha$ is an eigenvalue.

Fix $m > n \ge 0$, $a \in \cA_m$, and a prefix $v = v_0 \cdots v_k \in \prefixes(\tau_{n,m}(a))$.
Then there is a Dumont--Thomas type decomposition (see \cite{Dumont-Thomas}) of $v_0 \cdots v_{k-1}$
\[  v_0 \cdots v_{k-1} = \tau_{n,m-1}(u_{m-1}) \, \tau_{n,m-2}(u_{m-2}) \dots \tau_n(u_{n+1})\, u_n,  \]
where $u_i a_i \in \cA_i^*$ (with $a_i$ a letter),   $u_i a_i \in \prefixes(\tau_i(a_{i+1}))$ for $n \le i < m$, and $a_m = a$.
For $m \ge i > n$, set $b_i = \First(\tau_{n,i}(a_i)) \in \cA_n$, and let $b_n = a_n$.
Since
\[
   h_n(u) = \sum_{n \le i < m} h_i(u_i),
\]
we obtain
\[
   \rho_n(b_m) + \alpha h_n(u) - \rho_n(b_n)
   = \sum_{n \le i < m} \rho_n(b_{i+1}) + \alpha h_i(u_i) - \rho_n(b_i).
\]
Now, by Item (i), $\|\rho_i(a_i) - \rho_i(b_i)\| \le 2\varepsilon_{i,i+2}(\brho)$, and by Item (ii), $\|\rho_i(a_i) + \alpha h_i(u_i) - \rho_i(b_i)\| \le 4\varepsilon_{i,i+2}(\brho)$.
Combining these bounds in the last identity yields
\[
   \|\rho_n(b_m) + \alpha h_n(u) - \rho_n(b_n)\|
   \le 6 \sum_{i \ge n} \varepsilon_{i,i+2}(\brho).
\]
This proves \eqref{eq:1:theo:EigCharac:positive}, and therefore that $\varepsilon_n(\brhobar) \to 0$ as $n \to \infty$.
By \Cref{theo:EigCharac:Mult}, $\alpha$ is an eigenvalue of $X$, which ends the proof.
\medskip

Assume now that $\alpha$ is an additive eigenvalue.
Let $\brhobar = (\bar{\rho}_n(a) : n \ge 0,\, a \in \cA_n)$ be the sequence given by \Cref{lem:straight_rho}.
Since $\btau$ is positive, $r(n) = n+1$ in that lemma.
Item (2) then gives $\varepsilon_{n,n+2}(\brhobar) \le \delta_{n,n+2}(\btau)$ for every $n$.
By \Cref{lem:EigCharac:positive:returnWords}, the series $(\delta_{n,n+2}(\btau))_{n \ge 0}$ is summable, 
and hence so is $(\varepsilon_{n,n+2}(\brhobar))_{n \ge 0}$.
\end{proof}

\begin{proof}[Proof of \Cref{theo:EigCharac:proper}.]
The fact that \eqref{eq:theo:EigCharac:proper} converges to $0$ as $n \to \infty$ means that the hypothesis of \Cref{theo:EigCharac:Mult} are satisfied by $\brho = (\rho_n(a) \coloneqq 0 : n \geq 0, a \in \cA_n)$; hence $\alpha$ is an additive eigenvalue in this case.

Conversely, let us assume that $\alpha$ is an additive eigenvalue.
By \Cref{theo:EigCharac:Mult}, there exists a sequence of real numbers $\brho = (\rho_n(a) : n \geq 0, a \in \cA_n)$ such that $\varepsilon_n(\brho)$ converges to $0$ as $n \to \infty$.
By \Cref{lem:rho_controls_retwords}, this implies  that for every $n \ge 0$ and $m \ge r(n)$, $\delta_{n,r(m)}(\btau) \le \varepsilon_{n,r(m)}(\brho)$.
Since $\varepsilon_{n,r(m)}(\brho) \le 2 \varepsilon_n(\brho)$ by \eqref{trivial_bound_eps_n_ge_eps_nm}, we obtain that
\begin{equation}
    \label{eq:00:theo:EigCharac:proper}
    \text{$\sup_{m \ge r(n)} \delta_{n,r(m)}(\btau)$ converges to $0$ as $n \to \infty$.}
\end{equation}

Let now $\brhobar = (\bar{\rho}_n(a) : n \ge 0, \, a \in \cA_n)$ be the  optimal  sequence given by \Cref{lem:straight_rho}.
Since $\btau$ is left-proper, for every $i \geq 0$, there exist  an integer $n_0(i) > i$ and $\alpha_i \in \cA_i$ such that $\First(\tau_{i,n}(a)) = \alpha_i$ for all $n \geq n_0(i)$ and $a \in \cA_n$.
Then, for every $n \geq \max(n_0(i), r(i))$ and $a,b \in \cA_n$, Item (1) of \Cref{lem:straight_rho} gives the bound
\[  \| \rho_n(a) - \rho_n(b) \| \leq 
    \| \rho_n(a) - \rho_i(\alpha_k) \| + 
    \| \rho_n(b) - \rho_i(\alpha_k) \| \leq 
    2\delta_{i,r(n)}(\btau).
\]
From this and  Item (2) of \Cref{lem:straight_rho} we get, for any $m \ge r(n)$, $a \in \cA_m$ and $v_0 \cdots v_k \in \prefixes(\tau_{n,m}(a))$,
\begin{multline}
    \label{eq:2:theo:EigCharac:proper}
    \big\| \alpha h_n(v_0 \cdots v_{k-1}) \big\| \leq 
    \big\| \rho_n(v_0) - \rho_n(v_k) + \alpha h_n(v_0 \cdots v_{k-1}) \big\| +
    \big\| \rho_n(v_0) - \rho_n(v_k) \big\| \\ \leq 
    2\delta_{i,r(n)}(\btau) + 2\delta_{i,r(m)}(\btau) \le 
    4 \delta_{i,r(m)}(\btau).
\end{multline}
Since $\sup_{m\ge r(i)} \delta_{i,m}(\btau) \to 0$ as $m \to \infty$ by \eqref{eq:00:theo:EigCharac:proper}, we deduce that \eqref{eq:theo:EigCharac:proper} converges to $0$ as $n \to \infty$.

It is left to consider the case in which $\btau$ is positive and every $\tau_n$ is left-proper.
Observe that \eqref{eq2:theo:EigCharac:proper} is equivalent to \eqref{eq:theo:EigCharac:positive} with $\brho = (\rho_n(a) \coloneqq 0 : n \geq 0, a \in \cA_n)$, so $\alpha$ is an additive eigenvalue of $X_{\btau}$ whenever \eqref{eq2:theo:EigCharac:proper} holds.
Assume now that $\alpha$ is an additive eigenvalue.
We use a variation of the argument used in the non-positive case.
Let $\brhobar = (\bar{\rho}_n(a) : n \ge 0, \, a \in \cA_n)$ be the sequence given by \Cref{lem:straight_rho}.
Then, in \eqref{eq:2:theo:EigCharac:proper}, the current hypothesis gives $n_0(i) = i + 1$ and $r(n) = n + 1$, so 
\[  \big\| \alpha h_n(v_0 \cdots v_{k-1}) \big\| \leq 
    4\delta_{n-1,n+3}(\btau)
\]
for all $n \geq 1$ and $v_0 \cdots v_k \in \prefixes(\tau_{n,n+2}(a))$ with $a \in \cA_{n+2}$.
Since $(\delta_{n,n+4}(\btau) : n \ge 0)$ has finite sum by \Cref{lem:EigCharac:positive:returnWords}, \eqref{eq2:theo:EigCharac:proper} follows.
\end{proof}


\section{Characterization of eigenvalues through letter-coboundaries}
\label{sec:carac_finite_rank}

In this section, we impose further  mild assumptions on the directive sequence $\btau$. \Cref{theo:EigCharac:decisive:Mult,theo:EigCharac:FAR:Mult,theo:EigCharac:decisive:positive}  then state that  the sequences $\brho=( \rho_n(a): n \geq 0, \ a \in \cA_n )$ appearing in \Cref{theo:EigCharac:Mult,theo:EigCharac:positive} can be chosen to define letter-coboundaries on the corresponding levels $X_{\btau}^{(n)}$. 
In particular, this allows one to exploit linear-algebra duality, as shown in \Cref{cor:EigCharac:FAR:eigs&freqs}, which relates   the possible  values that eigenvalues can take  with   measures of  the basis $\mu(B_n(a))$. We also derive in Proposition~\ref{prop:balance_char_Sadic}  a characterization of balancedness.

We consider two wide subclasses of directive sequences: those with bounded alphabet rank, and those that are decisive. 
The former corresponds, in dynamical terms, to subshifts of \emph{finite topological rank}, an important class of low-complexity subshifts \cite{DM:08}. These are, roughly speaking,   systems  that  admit a Bratteli-Vershik  representation with a bounded number of vertices at each level; see \emph{e.g.} \cite{BDM10,espinoza22}. Concerning the later, 
we recall that decisiveness is a mild hypothesis, satisfied for example by proper directive sequences and by the  centered $1$-block presentations of primitive directive sequences (see Proposition \ref{prop:decisive:suff_conds}).  As  stressed in Remark \ref{rem:anyminimal},
every  infinite minimal subshift is generated by a primitive, 
 decisive, and  recognizable  directive sequence.

Our results  can be compared with  those in \cite{BCBY}, where a single, global coboundary is considered on $X_{\bt}$. 
Here we focus on local coboundaries, that is, coboundaries defined on  level-$n$ subshifts $X_{\bt}^{(n)}$.

We first state the corresponding characterizations in Section \ref{subsec:statements_FAR}.
We then discuss  relations with  symbolic discrepancy  in Section \ref{subsec:balance_FAR}. After  having stated a series of prepatory lemmas in Section \ref{subsec:proofsLemmaFAR},
proofs are   given in Section \ref{subsec:proofsFAR}.

\subsection{Characterizations of  continuous eigenvalues}
\label{subsec:statements_FAR}

We begin with the most general criterion of this section, which only requires the directive sequence to be decisive.  


\begin{theorem}
\label{theo:EigCharac:decisive:Mult}
Let $X$ be a subshift generated by a primitive and recognizable directive sequence $\btau = (\tau_n \colon \cA_{n+1}^* \to \cA_n^*)_{n\geq0}$.
Assume that $\btau$ is decisive, or that $(|\cA_n| : n \ge 0)$ is  bounded.
Then, a given real number $\alpha$ is an eigenvalue of $X$ if and only if there exist letter-coboundaries $c_n \colon \cA_n^* \to \R$ in $X_{\btau}^{(n)}$ for each $n \ge 0$ such that
\begin{equation}
    \label{eq:1:theo:EigCharac:decisive:Mult}
    \sup \big\{
        \big\| c_n(u) - \alpha h_n(u)  \big\| :
        u \in \cL(X_{\btau}^{(n)})  \big\}
\end{equation}
converges to $0$ as $n \to \infty$.
\end{theorem}

In the case where the alphabets of the  directive sequences  have bounded cardinality, the criterion  from \Cref{theo:EigCharac:decisive:Mult} can be simplified  in such  a way that it only involves  a single letter-coboundary $c_n$ at some level $n \ge 0$.  
Moreover, in Item~(3) of next theorem, we show that $c_n$ captures all the oscillatory behavior modulo $1$  of the quantity $\alpha h_n$ arising from jumps between  tops of  towers of the Kakutani--Rohklin partitions. In other words, $\alpha h_n$ decomposes into a letter-coboundary $c_n$ and a residue function $\gamma_n$, where the  residue function $\gamma_n$ satisfies Item (3) below, a condition that is analogous to the criterion from \Cref{theo:EigCharac:proper} for proper directive sequences.

\begin{theorem}
\label{theo:EigCharac:FAR:Mult}
Let $X$ be a subshift generated by a primitive and recognizable directive sequence $\btau = (\tau_n \colon \cA_{n+1}^* \to \cA_n^*)_{n\geq0}$.
Assume that $(|\cA_n| : n \geq 0)$ is bounded.
Then, for a given real number $\alpha$, the following conditions are equivalent:
\begin{enumerate}
    \item $\alpha$ is an additive eigenvalue of $X$;
    
    \item there exist an integer $n \geq 0$ and a letter-coboundary $c_n \colon \cA_n^* \to \R$ in $X_{\btau}^{(n)}$ such that
    \begin{equation}
        \label{eq:hip_cob:theo:EigCharac:FAR:Mult}
        \sup \big\{
            \big\| c_n(\tau_{n,m}(u)) - \alpha h_m(u)  \big\| :
            u \in \cL(X_{\btau}^{(m)})  \big\}
    \end{equation}
    converges to $0$ as $m \to \infty$;

    \item there exist $n \geq 0$ and a letter-coboundary $c_n \colon \cA_n^* \to \R$ in $X_{\btau}^{(n)}$ such that the morphism $\gamma_n \colon \cA^*_n \to \R$, given by $\gamma_n(a) = \{c_n(a) - \alpha h_n(a)\}$ for  $a \in \cA_n$, satisfies that 
    \begin{equation}
        \label{eq:hip_gamma:theo:EigCharac:FAR:Mult}
        \sup\bigl\{ |\gamma_n(\tau_{n,m}(u))| : u \in \cL(X_{\btau}^{(m)}) \bigr\}
    \end{equation}
    converges to 0 as $m \to \infty$.
\end{enumerate}
\end{theorem}


Any of Items (2) or (3) of \Cref{theo:EigCharac:FAR:Mult} implies Item (1), even without assuming that the alphabets of $\btau$ are bounded, as shown next. 
Indeed, the map $c_n \circ \tau_{n,m} \colon \cA_m^* \to \mathbb{R}$ is a letter-coboundary in $X_{\btau}^{(m)}$,   by  Lemma \ref{lem:compositioncoboundary}.
Therefore, by \eqref{eq:hip_cob:theo:EigCharac:FAR:Mult} or \eqref{eq:hip_gamma:theo:EigCharac:FAR:Mult}, the hypothesis of \Cref{theo:EigCharac:Mult} is satisfied for the sequence $(\rho_n(a) : m > n,\, a \in \cA_n)$ where $\rho_n \colon \cA_n \to \R$ is given by \Cref{lem:cobord&rho} when applied with $c_n \circ \tau_{n,m}$.

In both Items (2) and (3), we impose a condition on a morphism,  which is either $c_n - \alpha h_n$, or $\gamma_n$.
The main difference is that in Item (2) we consider the distance to the nearest integer $\|\cdot\|$, whereas in Item (3) we use the absolute value $|\cdot|$.
This crucial distinction allows the convergence in Item (3) to be analyzed using tools from linear algebra; see \Cref{subsec:nontrivial,subsec:nontrivialdynamical}
for illustrations.

Note that the properties stated in  Items (2) and (3) are asymptotic properties in the sense that if they hold for some $n \geq 0$, then they also hold for every $m \geq n$.
Finally, we note that Item (3) of \Cref{theo:EigCharac:FAR:Mult} generalizes \cite[Corollary 3.6]{mercat} by removing the finitary condition, namely that the substitutions in the directive sequence take only finitely many values.

As it happens with \Cref{theo:EigCharac:positive}, the criterion of \Cref{theo:EigCharac:decisive:Mult} can be simplified for directive sequences exhibiting strong recurrence properties.
We carry out this simplification in \Cref{theo:EigCharac:decisive:positive} below.
Note that this time we require a property slightly stronger than the strong primitivity of the directive sequence.

\begin{theorem}
\label{theo:EigCharac:decisive:positive}
Let $X$ be a subshift generated by a recognizable directive sequence $\btau = (\tau_n \colon \cA_{n+1}^* \to \cA_n^*)_{n\geq0}$ that is decisive.
Assume that for all $n \geq 0$ and $a \in \cA_{n+1}$, every word $u \in \cL(X_{\btau}^{(n)})$ of length $2$ occurs in $\tau_n(a)$.
Then, a real number $\alpha$ is an eigenvalue of $X$ if and only if there exist letter-coboundaries $c_n \colon \cA_n^* \to \R$ in $X_{\btau}^{(n)}$, for each $n \ge 0$, such that 
\begin{equation}
    \label{eq:1:theo:EigCharac:FAR:positive}
    \sum_{n \ge 0} \max \big\{
        \big\| c_n(u) - \alpha h_n(u) \| :
        a \in \cA_{n+2},\, 
        u \in \prefixes(\tau_{n,n+2}(a)) \big\}
        < \infty.
\end{equation}
\end{theorem}

\begin{remark} \label{rem:FAR_crit=>Host}
Condition \eqref{eq:1:theo:EigCharac:FAR:positive} implies Item (2) of \Cref{theo:EigCharac:FAR:Mult}, which in turn implies
\begin{equation}
    \label{eq:theo:EigCharac:FAR:Mult:necessary}
        \lim_{n\to\infty} 
        \max\big\{ \big\| c_n(u_n)
        - \alpha h_n(u_n) \big\| :
        a \in \cA_{n+1}, \, u_n \in \prefixes(\tau_n(a)) \big\}
        = 0.
\end{equation}
Condition \eqref{eq:theo:EigCharac:FAR:Mult:necessary} is analogous to the classical criterion of Host for substitutions \cite{Host86}.
However, as noted in \cite{eigenvalues_LR_2003}, \eqref{eq:theo:EigCharac:FAR:Mult:necessary} is not sufficient for $\alpha$ to be an additive eigenvalue in the $S$-adic setting.
\end{remark}
An interesting aspect of Theorems \ref{theo:EigCharac:FAR:Mult} and \ref{theo:EigCharac:decisive:positive} is that they cannot be extended to non-decisive directive sequences of infinite alphabet rank; see Section~\ref{subsec:ex:nocoboundary} for such an example.

\subsection{More on  eigenvalues and finite discrepancy}\label{subsec:balance_FAR}
In this section,  we  revisit the condition in Item~(3) of \Cref{theo:EigCharac:FAR:Mult} in terms of  finite discrepancy  and balancedness (see Definition~\ref{def:balance}).

We first recall the notation $B_n(a) = \{\tau_{0,n}(x) : x \in X_{\btau}^{(n)},\, x_0 = a\}$ as defined in \eqref{eq:Sadic:defi_Bn(a)}.  Assume that $\alpha$ is an eigenvalue.
The morphism $\gamma_n$ in Item~(3) of \Cref{theo:EigCharac:FAR:Mult} satisfies, by Item~(2), that $$|\gamma_n(v)| = \|c_n(v) - \alpha h_n(v)\|$$ for all $v \in \cL(X_{\btau}^{(n)})$.
Let  $w_n(a)$ be the nearest integer to $c_n(a) - \alpha h_n(a)$, \emph{i.e.},
$$c_n(a) - \alpha h_n(a)= w_n(a)+ \gamma_n(a).$$
According to to  Lemma \ref{lemma:EigCharac:gamma2freqs&eig} below,  the continuous map $f_n \colon X_{\btau} \to \Z$ defined by
\begin{equation}\label{eq:fna}
   f_n(x) = \begin{cases} 
   w_n(a) & \text{if $x \in B_n(a)$} \\
            0      & \text{otherwise}
        \end{cases}     
\end{equation} 
satisfies $\int f_n \,\mathrm{d}\mu = \alpha$ for all $\mu \in \cM(X,S)$,  and 
moreover,  $f_n$ is balanced, according to  Lemma \ref{lemma:EigCharac:gamma2freqs&eig}.
We recall that the connection between balancedness of integer functions and eigenvalues is an old topic, and the explicit description of such functions $f_n$ given here is similar to the one found in \cite[Theorem 3.2]{ghh18} for minimal Cantor systems and their proper Bratteli-Vershik models, where it is  proved  that,  for any eigenvalue $\alpha$,
there exists a clopen set $U$, with $\mu(U) = \alpha$ for all $\mu \in \cM(X,S)$, such that ${\mathbf 1}_U - \alpha {\mathbf 1}$ is a real coboundary, see also \cite{CORTEZ_DURAND_PETITE_2016}.
We recover here  the fact that any eigenvalue can be  realized  as 
the integral of a  balanced function. 
This  connection is  at the core  of next proposition 
that shows that eigenvalues can be described as  linear integer combinations of measures of the bases.

\begin{proposition}
\label{cor:EigCharac:FAR:eigs&freqs}
Let $X$ be a subshift generated by a primitive and recognizable directive sequence $\btau = (\tau_n \colon \cA_{n+1}^* \to \cA_n^*)_{n\geq0}$.
Assume that $\btau$ is decisive or that $(|\cA_n| : n \geq 0)$ is bounded.
Then, for every eigenvalue $\alpha$ of $X$ and sufficiently large $n$, there exists an integer vector $(w_a)_{a\in\cA_n} \in \Z^{\cA_n}$ such that
\begin{equation*}
    \alpha = \sum_{a\in\cA_n} w_a \, \mu(B_n(a))
    \enspace \text{for all $\mu \in \cM(X,S)$.}
\end{equation*}
\end{proposition}

\begin{remark}\label{rem:measures}
It is important to compare \Cref{cor:EigCharac:FAR:eigs&freqs} with \cite{itzaortiz,CORTEZ_DURAND_PETITE_2016,ghh18}, where it is shown that for a minimal Cantor system $(X,T)$ the group of eigenvalues $E(X,T)$ is contained in the image subgroup $I(X,T)$ (see \Cref{subsec:topologicalds}).
In the case of a minimal subshift on the alphabet $\cA$, the image subgroup turns out to be related to the additive group  generated by the measures of 
cylinders, that is,
\[
I(X,S) = 
\bigcap_{\mu \in \cM(X,T)}
\bigl\{  \sum_{a \in \cA} w_a \, \mu\bigl([u]\bigr)
: (w_a)_{a \in \cA} \in \Z^{\cA} \bigr\};
\]
see \emph{e.g.} \cite[Prop. 2.6]{unimodular}.

\Cref{cor:EigCharac:FAR:eigs&freqs} yields a sharper statement for minimal subshifts.  
Indeed, if a subshift $(X,S)$ admits an $S$-adic structure given by a primitive and recognizable directive sequence $\btau=(\tau_n \colon \cA_{n+1}^* \to \cA_n^*)_{n \ge 0}$, then
\begin{equation}\label{eq:Itau}
E(X,S) \subseteq I(\btau)
\coloneqq
\bigcap_{\mu \in \cM(X,S)} \ \bigcup_{n\ge 0}
\Bigl\{
\sum_{a \in \cA_n} w_a \, \mu\bigl(B_n(a)\bigr)
: (w_a)_{a \in \cA_n} \in \Z^{\cA_n}
\Bigr\}.
\end{equation}
It always holds that $I(\btau) \subseteq I(X,S)$. 
In certain natural cases (including the Thue--Morse substitution; see Section~\ref{subsec:Itau}), the inclusion is strict, $I(\btau) \subsetneq I(X,S)$. 
Thus, in the setting of minimal subshifts, \Cref{cor:EigCharac:FAR:eigs&freqs} provides tighter bounds than the general inclusion $E(X,S) \subseteq I(X,S)$.
\end{remark}

\Cref{cor:EigCharac:FAR:eigs&freqs} highlights the importance of linear algebra duality in the study of eigenvalues:  
the measures $\mu(B_n(a))$ arise as generalized right eigenvectors of the incidence matrices of the substitutions  $\tau_n$ (see \Cref{theo:EigCharac:positive}), while our criterion in \Cref{theo:EigCharac:FAR:Mult} relies on the heights $h_n$, corresponding to the left  action of the  row vector $\vec{\boldone}$ whose  entries are all equal to $1$ on these matrices.  
Exploiting this duality is essential for computing eigenvalues in concrete situations.  
We illustrate this in \Cref{ex:CM}, for a specific non-proper substitution.  
Remark that this connection is obtained under the additional assumptions of decisiveness or finite alphabet rank. In fact, it is only in these cases  that we can establish  a criterion based on coboundaries. \Cref{thm:NoCobsInftyRank} shows that   this is  not possible for all directive sequences. 
We need in particular the fact that coboundaries vanish   under integration, which is what allows us to express $\alpha$ in terms of a linear combination of measures $\mu(B_n(a))$ of the bases $B_n(a)$.

We even provide a characterization of balanced locally-constant functions $f \colon X \to \R$ in terms of coboundaries.
We remark that the hypotheses of the next proposition are satisfied for any continuous function $f \colon X \to \Z$ provided that the sequence of Kakutani–-Rohklin partitions defined by $\btau$ generates the topology (which occurs when $\btau$ is decisive by Proposition \ref{prop:decisive:dyn_inter}).

\begin{proposition}
    \label{prop:balance_char_Sadic}
Let $X$ be a minimal subshift generated by a primitive, recognizable directive sequence $\btau = (\tau_n \colon \cA^*_{n+1} \to \cA^*_n : n \geq 0)$, with $(|\cA_n|)_{n \ge 0}$ bounded. Let $f \colon X \to \R$ such that  there exists
$\ell \ge 0$  for which  $f$ takes  constant value,  denoted as $f_{\ell}(a,k) $, on each $S^k B_{\ell}(a)$, where $a \in \cA_{\ell}$ and $0 \leq k < h_{\ell}(a)$. 
Define the morphism $f_{\ell} \colon \cA_{\ell}^* \to \R$ by
\[  f_{\ell}(a) = 
    \sum_{0\leq k < h_{\ell}(a)} f_{\ell}(a,k) 
    \quad \text{for all } a \in \cA_{\ell}.
\]
Then, $f$ is balanced on $(X,S)$ if and only if there exist $n \ge \ell$, a letter-coboundary $c_n$ in $X_{\btau}^{(n)}$, and  an invariant measure $\mu$ of $X$ such that $\alpha \coloneqq \int f \, \mathrm{d}\mu$ satisfies that
\begin{equation}
    \label{eq:hip_gamma:prop:balance_char_Sadic}
    \sup\bigl\{ |f_\ell(\tau_{\ell,n}(u)) - \alpha h_n(u) + c_n(u)| : u = \tau_{n,m}(v),\, v \in \cL(X_{\btau}^{(m)}) \bigr\}
\end{equation}
tends to 0 as $m \to \infty$.
In this case, $\alpha = \int f \, \mathrm{d}\nu$ for every invariant measure $\nu$ of $X$, and $\alpha$ is an eigenvalue of $X$.
\end{proposition}

We will deduce  a simple characterization  of letter-balancedness
when $X$ is generated by a primitive aperiodic  substitution in 
 \Cref{balanced=>spaces_decomposition}.
Furthermore,  we present  in \Cref{ex:sturmian_avec_cobord} an example of a directive  sequence and  a function  $f$  which requires a non-trivial letter-coboundary  so that \eqref{eq:hip_gamma:prop:balance_char_Sadic} holds.

\subsection{Preparatory lemmas}\label{subsec:proofsLemmaFAR}


We start by introducing some terminology and notation.
Fix a directive sequence $\btau = (\tau_n \colon \cA_{n+1}^* \to \cA_n^* : n \ge 0)$.
In \Cref{sec:carac} we defined, for each $n \ge 0$, $r(n)$ as the least $m \gt n$ such that every letter $b \in \cA_n$ occurs in $\tau_{n,m}(a)$ for every $a \in \cA_m$.
In this section, we use instead $r'(n)$, defined as the least $m \gt n$ such that every length-2 word $u \in \cL(X_{\btau}^{(n)})$ occurs in $\tau_{n,m}(a)$ for every $a \in \cA_m$.
Clearly, $r'(n) \ge r(n)$.

Many of the arguments used to deduce Item (1) from Item (3) of \Cref{theo:EigCharac:FAR:Mult} are likewise needed for \Cref{cor:EigCharac:FAR:eigs&freqs}. 
We therefore present a single lemma that encapsulates these arguments and yields both results.


\begin{lemma}
\label{lemma:EigCharac:gamma2freqs&eig}
Let $X$ be a subshift generated by a primitive and recognizable directive sequence $\btau = (\tau_n \colon \cA_{n+1}^* \to \cA_n^*)_{n\geq0}$.
Fix $\alpha \in \mathbb{R}$ and $n \geq 0$, and let $c_n \colon \cA_n^* \to \mathbb{R}$ be a letter-coboundary in $X_{\btau}^{(n)}$. 
Let $\gamma_n \colon \cA_n^* \to \R$ be the morphism defined by $\gamma_n(a) = \bigl\{\alpha h_n(a) - c_n(a)\bigr\}$ for $a \in \cA_n$, and assume that
\begin{equation}
    \label{eq:hip:lemma:EigCharac:gamma2freqs&eig}
    \gamma_n(v) = 
    \bigl\{\alpha h_n(v) - c_n(v)\bigr\}
    \quad \text{for all $v \in \cL(X_{\btau}^{(n)})$}.
\end{equation}
Then, $\alpha$ is an eigenvalue of $X$, and
\begin{equation}
    \label{eq:statement:lemma:EigCharac:gamma2freqs&eig}
    \alpha = \sum_{a\in\cA_n} w_a \, \mu(B_n(a))
    \quad \text{for every $\mu \in \cM(X, S)$,}
\end{equation}
where $w_a = \alpha h_n(a) - c_n(a) - \gamma_n(a) \in \mathbb{Z}$.
Moreover,  the continuous map $f_n \colon X_{\btau} \to \Z$ defined in \Cref{eq:fna})
satisfies $\int f_n \,\mathrm{d}\mu = \alpha$ for every $\mu \in \cM(X,S)$ and  is balanced.
\end{lemma}

\begin{proof}
Let $\overline{\gamma}_n \colon X_{\btau}^{(n)} \to \R$ be the continuous map defined by $\overline{\gamma}_n(x) = \gamma_n(x_0)$ for $x \in X_{\btau}^{(n)}$.
Observe that, for any $x \in X_{\btau}^{(n)}$ and $N > 0$, \eqref{eq:hip:lemma:EigCharac:gamma2freqs&eig} gives
\[ \Bigl| \sum_{0 \leq j < N} \overline{\gamma}_n (S^j x) \Bigr| = 
    \bigl| \gamma_n(x_{[0,N)})\bigr| =
    \| \alpha h_n(x_{[0,N)})- c_n(x_{[0,N)}) \| \le 1/2.
    \]
This implies, by the Gottschalk-Hedlund Theorem (see \Cref{theo:GH}) that $\overline{\gamma}_n = \xi_n - \xi_n \circ S$ for some continuous map $\xi_n \colon X_{\btau}^{(n)} \to \R$.
Hence, for any $\mu \in \cM(X,S)$, the invariant measure $\mu_n$ on $X_{\btau}^{(n)}$ given by \Cref{prop:measure_transfer} satisfies
\begin{equation}
    \label{eq:0:lemma:EigCharac:gamma2freqs&eig}
    \sum_{a \in \cA_n} \gamma_n(a) \, \mu_n([a]_n) = 
    \int \overline{\gamma} \, \mathrm{d}\mu_n = 
    \int (\xi_n - \xi_n \circ S)\,\mathrm{d}\mu_n = 0.
\end{equation}
Further, \Cref{prop:measure_transfer} states that there exists  $\lambda_n$ such that  $\mu(B_n(a)) = \lambda_n \mu_n([a]_n) $ for every $b \in \cA_n$, hence $\sum_{a \in \cA_n} \gamma_n(a) \, \mu(B_n(a))=0.$
Now, if we let $w_a \in \Z$ be the nearest integer to $\alpha h_n(a) - c_n(a)$, then $ \alpha h_n(a) = c_n(a)+ \gamma_n(a) +w_a$ for every $a \in \cA_n$.
This yields
\begin{equation*}
    \alpha =  \alpha( \sum_{a \in \cA_n} \! h_n(a) \, \mu(B_n(a)))=
    \sum_{a \in \cA_n} \! c_n(a) \, \mu(B_n(a)) +
    \sum_{a \in \cA_n} \! w_n(a) \, \mu(B_n(a)).
\end{equation*}

The first term in the right-hand side is zero by \Cref{lem:cob_zero_integral}, so \eqref{eq:statement:lemma:EigCharac:gamma2freqs&eig} follows.
This  also implies that  the function $f_n$ defined in \Cref{eq:fna} 
satisfies $\int f_n \,\mathrm{d}\mu = \alpha$ for all $\mu \in \cM(X,S)$,  and  is balanced.
\smallbreak

Next, we use $\xi_n$ to construct an eigenfunction $g$ of $X_{\btau}$ with eigenvalue $\alpha$.
By \Cref{lem:cobord&rho}, there exists $\rho_n \colon \cA_n \to \R$ such that $c_n(a) = \rho_n(b) - \rho_n(a)$ for all $ab \in \cL(X_{\btau}^{(n)})$.
Then, we define $g \colon X_{\btau} \to \R/\Z$ by 
\[  g(S^k \tau_{0,n}(x)) = \xi_n(x) + \rho_n(x_0) + k \alpha
    \pmod \Z
    \quad\text{for $x \in X_{\btau}^{(n)}$ and $0 \leq k < h_n(x_0)$.}
\]
Because $\btau$ is recognizable, $g$ is well-defined and continuous.

Let us prove that $g$ is an eigenfunction.
Assume that $y = S^k \tau_{0,n}(x)$ for some $x \in X_{\btau}^{(n)}$ and $0 \leq k < h_n(x_0)$.
If $k \neq h_n(x_0) - 1$, then $g(S y) = \xi_n(x) + \rho_n(x_0) + (k+1) \alpha = g(y) + \alpha \pmod \Z$.
Suppose now that $k = h_n(x_0) - 1$.
In this case, $Sy = \tau_{0,n}(S x)$, so $g(S y) = \xi_n(S x) + \rho_n(x_1) \pmod \Z$.
Hence,
\begin{multline*}
    g(Sy) - g(y) = 
    (\rho_n(x_1) - \rho_n(x_0) - \alpha h_n(x_0)) - (\xi_n(x) - \xi_n(Sx)) + \alpha \\ = 
    c_n(x_0) - \alpha h_n(x_0) - \overline{\gamma}_n(x) + \alpha 
    \pmod \Z
\end{multline*}
where in the last step we used that $c_n(x_0) = \rho_n(x_1) - \rho_n(x_0)$ and $\overline{\gamma}_n(x) = \xi_n(x) - \xi_n(Sx)$.
This, together with $\overline{\gamma}_n(x) = \{c_n(x_0) - \alpha h_n(x_0)\}$ (by definition of $\overline{\gamma}_n$) gives $g(Sy) - g(y) = \alpha \pmod{\Z}$.
We conclude that $g$ is a continuous eigenfunction with eigenvalue $\alpha$.
\end{proof}

The following lemma is used in the proofs of \Cref{theo:EigCharac:decisive:Mult,theo:EigCharac:decisive:positive} with $\chi_n \coloneqq \alpha h_n $ and $\vertiii{\cdot} = \|\cdot\|$, 
and in the proof of \Cref{lem:balance=>small_abs_val} (as a preparatory step before establishing \Cref{prop:balance_char_Sadic}) with $\chi_\ell \coloneqq f_\ell - \alpha h_\ell $ and $\vertiii{\cdot} = |\cdot|$.

\begin{lemma}
\label{lemma:EigCharac:rho2cobs}
Let $\btau = (\tau_n \colon \cA_{n+1}^* \to \cA_n^*)_{n\geq0}$ be a primitive and recognizable directive sequence that is either decisive or has $(|\cA_n| : n \ge 0)$ bounded. 
Consider a family of real numbers $\brho = (\rho_n(a) : n \ge 0,\, a\in\cA_n)$ and morphisms $(\chi_n \colon \cA^*_n \to \R : n\ge0)$.
Let $\vertiii{\cdot}$ denote either the absolute value $|\cdot|$ or the distance to the nearest integer $\|\cdot\|$, and define
\begin{equation*}
    \varepsilon_{n,m} = 
    \max\Big\{
    \vertiii{\rho_n(v_0) + \chi_n(v_0 \cdots v_{k-1}) - \rho_n(v_k)} :
    a\in\cA_m, \, 
    v_0 \cdots v_k \in \prefixes(\tau_{n,m}(a))
    \Big\}.
\end{equation*}
Then, there exist letter-coboundaries $c_n \colon \cA_n^* \to \R$ in $X_{\btau}^{(n)}$, for each $n \ge 0$, and a constant $Q > 0$ such that
\begin{equation}
    \label{eq:statement:lemma:EigCharac:rho2cobs}
    \vertiii{ c_n(v) - \chi_n(v) } \le 
    Q\, \varepsilon_{n-1,m}
\end{equation}
for all $n \ge 0$, $m \ge r'(n)$, $a \in \cA_m$, and all $v$ occurring in $\tau_{n,m}(a)$.
\end{lemma}

\begin{proof}
We construct the letter-coboundaries $c_n$ using \Cref{lem:cobord&rho}.  
The key point is to show that there is a global constant $Q$ such that, for every $n \ge 1$ and any two vertices $b_R,b'_R$ of $\Gamma_{X_{\btau}^{(n)}}(\varepsilon)$ lying in the same connected component, the following holds:
\begin{equation}
    \label{eq:CLAIM:theo:EigCharac:FAR:Mult}
    \vertiii{\rho_n(b) - \rho_n(b')} 
    \le 
    Q\,\varepsilon_{n-1,r'(n)}.
\end{equation}

To prove \eqref{eq:CLAIM:theo:EigCharac:FAR:Mult}, we begin with the following observation.  
For any word $v$ occurring in $\tau_{n-1,m}(d)$ with $m \ge n$ and $d\in\cA_m$, we have 
$\chi_{n-1}(v) = \chi_{n-1}(u) - \chi_{n-1}(u')$ for suitable $u,u' \in \prefixes(\tau_{n-1,m}(d))$.
Thus,
$ \vertiii{ \rho_{n-1}(v_0) + \chi_{n-1}(v) - \rho_{n-1}(v_k) }
    \le 2\varepsilon_{n-1,m}$
for all $m \ge n \ge 1$ and all $v = v_0 \cdots v_k$ occurring in $\tau_{n-1,m}(d)$ for some $d \in \cA_m$.
Now consider $v_0\cdots v_k$ occurring in $\tau_{n,m}(d)$ for some $d\in\cA_m$.  
Define $\bar\rho_n(a) = \rho_{n-1}(\First(\tau_{n-1}(a)))$ for $a\in\cA_n$.  
Applying the previous inequality to the word 
$\tau_{n-1}(v_0\cdots v_{k-1})\,\First(\tau_{n-1}(v_k))$ yields
\begin{equation}
    \label{eq:00:theo:EigCharac:FAR:Mult}
    \vertiii{ \bar\rho_n(v_0) + \chi_n(v_0\cdots v_{k-1}) - \bar\rho_n(v_k) }
    \le 2\varepsilon_{n-1,m}
\end{equation}
for all $m \ge n \ge 1$ and all $v_0\cdots v_k$ occurring in $\tau_{n,m}(a)$.
Since $r'(n)$ is chosen so that every length-2 word $ab$ in $\cL(X_{\btau}^{(n)})$ appears in $\tau_{n,r'(n)}(d)$ for every $d \in \cA_{r'(n)}$, we obtain from \eqref{eq:00:theo:EigCharac:FAR:Mult} that
\begin{equation}
    \label{eq:01:theo:EigCharac:FAR:Mult}
    \vertiii{ \bar\rho_n(a) + \chi_n(a) - \bar\rho_n(b) }
    \le 2\varepsilon_{n-1,r'(n)}
    \quad\text{for all } ab \in \cL(X_{\btau}^{(n)}).
\end{equation}

We now prove \eqref{eq:CLAIM:theo:EigCharac:FAR:Mult} in the case in which $(|\cA_n| : n \geq 0)$ is bounded.
For any $ab,ab'\in\cL(X_{\btau}^{(n)})$ of length $2$, \eqref{eq:01:theo:EigCharac:FAR:Mult} shows that $\bar\rho_n(b)$ and $\bar\rho_n(b')$ lie within $4\varepsilon_{n-1,r'(n)}$ of each other.  
Thus, if $b_R$ and $b'_R$ have a common neighbor in $\Gamma_{\!X_{\btau}^{(n)}}(\varepsilon)$, then the same bound holds. 
By an inductive argument along a path inside each connected component of $\Gamma_{\!X_{\btau}^{(n)}}(\varepsilon)$, we get
\[
    \vertiii{ \bar\rho_n(b) - \bar\rho_n(b') }
    \le 4 |\cA_n|\, \varepsilon_{n-1,r'(n)}
\]
for all $b_R,b'_R$ in the same connected component.  
Since $(|\cA_n| : n \ge 0)$ is bounded, this proves \eqref{eq:CLAIM:theo:EigCharac:FAR:Mult} with $Q = 4\sup_n |\cA_n|$.

Now assume that $\btau$ is decisive.
For each $n \ge 1$, let $f_n\colon\cA_n\to\cA_{n-1}$ be such that $\tau_{n-1}(b)$ begins with $f_n(a)$ whenever $ab\in\cL(X_{\btau}^{(n)})$.  
By \Cref{lemma:decisive:same_first_letter_in_each_CC}, if $b_R$ and $b'_R$ lie in the same connected component of $\Gamma_{\!X_{\btau}^{(n)}}(\varepsilon)$, then $\tau_{n-1}(b)$ and $\tau_{n-1}(b')$ begin with the same letter, so $f_n(b)=f_n(b')$.  
Thus $\bar\rho_n(b)=\bar\rho_n(b')$, and \eqref{eq:CLAIM:theo:EigCharac:FAR:Mult} holds with $Q=0$.
\medskip

Having established \eqref{eq:CLAIM:theo:EigCharac:FAR:Mult}, we can define a map $\tilde\rho_n\colon\cA_n\to\R$ such that
\begin{enumerate}[label=(\roman*)]
    \item $\tilde\rho_n(b)=\tilde\rho_n(b')$ whenever $b_R,b'_R$ lie in the same connected component of $\Gamma_{X_{\btau}^{(n)}}(\varepsilon)$;
    \item $\vertiii{\tilde\rho_n(b)-\bar\rho_n(b)} \le Q \varepsilon_{n-1,r'(n)}$ for all $b\in\cA_n$.
\end{enumerate}
By (i) and \Cref{lem:cobord&rho}, there exists a letter-coboundary $c_n$ on $X_{\btau}^{(n)}$ such that $c_n(v_0\cdots v_{k-1}) = \tilde\rho_n(v_k)-\tilde\rho_n(v_0)$ for every $v_0\cdots v_k\in\cL(X_{\btau}^{(n)})$ with $k\ge1$.
Then  (ii) gives, for all such $v_0\cdots v_k$, that 
\[
    \vertiii{c_n(v_0\cdots v_{k-1}) - \bar\rho_n(v_k) + \bar\rho_n(v_0)}
    \le 2Q \varepsilon_{n-1,r'(n)},
\]
and thus 
\[
    \vertiii{ c_n(v_0\cdots v_{k-1}) - \chi_n(v_0\cdots v_{k-1}) }
    \le 
    \vertiii{ \chi_n(v_0\cdots v_{k-1}) - \bar\rho_n(v_k) + \bar\rho_n(v_0) }
    + 2Q \varepsilon_{n-1,r'(n)}.
\]
Now, if $v_0\cdots v_k$ occurs in $\tau_{n,m}(a)$ for some $m\ge r'(n)$ and $a \in \cA_m$, then \eqref{eq:00:theo:EigCharac:FAR:Mult} shows that $\bar\rho_n(v_k)-\bar\rho_n(v_0)$ is within $2\varepsilon_{n-1,m}$ of $\chi_n(v_0\cdots v_{k-1})$.  
Plugging this into the inequality above yields
\[
    \vertiii{ c_n(v_0\cdots v_{k-1}) - \chi_n(v_0\cdots v_{k-1}) }
    \le
    2\varepsilon_{n-1,m} + 2Q \varepsilon_{n-1,r'(n)}.
\]
Since $m\ge r'(n)$ implies $\varepsilon_{n-1,r'(n)} \le \varepsilon_{n-1,m}$, we get that \eqref{eq:statement:lemma:EigCharac:rho2cobs} holds with constant $2Q+1$.
\end{proof}

We remark that the next lemma is completely general and can be stated for any subshift.
We present it in the setting of directive sequences to make its later use more transparent.

\begin{lemma}
\label{lem:small_gamma=>morphism}
Let $\btau = (\tau_n \colon \cA_{n+1}^* \to \cA_n^* : n\geq0)$ be a directive sequence.
Fix $n \ge 0$, and suppose that $\gamma'_n \colon \cA_n^* \to \R$ is a morphism satisfying
\begin{equation}
\label{eq:hip:lem:small_gamma=>morphism}
    \|\gamma'_n(v)\| < 1/3
    \quad
    \text{for all $v \in \cL(X_{\btau}^{(n)})$.}
\end{equation}
Define the morphism $\gamma_n \colon \cA_n^* \to \R$ by $\gamma_n(a) = \{\gamma'_n(a)\}$ for each $a \in \cA_n$.
Then,
\begin{equation}
\label{eq:conclu:lem:small_gamma=>morphism}
    \gamma_n(v) = \{\gamma'_n(v)\}
    \quad
    \text{for all $v \in \cL(X_{\btau}^{(n)})$.}
\end{equation}
\end{lemma}
\begin{proof}
For each $v \in \cL(X_{\btau}^{(n)})$, there is a unique decomposition $\gamma'_n(v) = w_n(v) + \delta_n(v)$, with $w_n(v) \in \Z$ and $\delta_n(v) = \{\gamma'_n(v)\}$.
These decompositions define maps  $w_n \colon \cL(X_{\btau}^{(n)}) \to \Z$ and $\delta_n \colon \cL(X_{\btau}^{(n)}) \to \R$.
Note that, by definition, $\gamma_n(a) = \delta_n(a)$ for all $a \in \cA_n$.

We now show that
\begin{equation}
    \label{eq:1:lem:small_gamma=>morphism}
    w_n(uv) = w_n(u) + w_n(v)
    \quad\text{and}\quad
    \delta_n(uv) = \delta_n(u) + \delta_n(v)
    \qquad 
    \text{for all $uv \in \cL(X_{\btau}^{(n)})$.}
\end{equation}
Once this is proved, we can get \eqref{eq:conclu:lem:small_gamma=>morphism} as follows. 
For any $v = v_0 \cdots v_{k-1} \in \cL(X_{\btau}^{(n)})$, we compute using that $\gamma_n$ is a morphism:
\[  \delta_n(v) = 
    \sum_{0 \le i < k} \delta_n(v_i) = 
    \sum_{0 \le i < k} \gamma_n(v_i) =
    \gamma_n(v).
\]
Therefore, $\gamma_n(v) = \delta(v) = \{\gamma'_n(v)\}$ for all $v \in \cL(X_{\btau}^{(n)})$.

Let us prove \eqref{eq:1:lem:small_gamma=>morphism}.
By definition of $w_n$ and $\delta_n$, and using that $\gamma'_n$ is a morphism, we can write
\begin{equation}
    \label{eq:2:lem:small_gamma=>morphism}
    w_n(uv) + \delta_n(uv)
    = \gamma'_n(uv)
    = \gamma'_n(u) + \gamma'_n(v)
    = w_n(u) + \delta_n(u) + w_n(v) + \delta_n(v)
\end{equation}
whenever $uv \in \cL(X_{\btau}^{(n)})$.
Hence,
\[
    |w_n(uv) - w_n(u) - w_n(v)| \le
    \|\delta_n(uv)\| + \|\delta_n(u)\| + \|\delta_n(v)\| \lt 1,
\]
where in the last step we used the hypothesis \eqref{eq:hip:lem:small_gamma=>morphism}.
Since $w_n(uv) - w_n(u) - w_n(v)$ is an integer, it must be zero; thus, $w_n(uv) = w_n(u) + w_n(v)$.
Substituting this identity back into \eqref{eq:2:lem:small_gamma=>morphism} yields $\delta_n(uv) = \delta_n(u) + \delta_n(v)$, completing the proof of \eqref{eq:1:lem:small_gamma=>morphism} and hence of the lemma.
\end{proof}

The following lemma, which relies on Lemma \ref{lemma:EigCharac:rho2cobs}, is used to prove \Cref{prop:balance_char_Sadic}.

\begin{lemma}
\label{lem:balance=>small_abs_val}
Let $X$ be a subshift generated by a primitive, recognizable directive sequence 
$\btau = (\tau_n \colon \cA_{n+1}^* \to \cA_n^* : n \ge 0)$ 
that is either decisive or satisfies that $(|\cA_n| : n \ge 0)$ is bounded. 
Fix $\ell \ge 0$, $\alpha \in \R$, and an integer-valued morphism 
$f_\ell \colon \cA_\ell^* \to \R$ such that
\begin{equation}
\label{eq:hip:lem:balance=>small_abs_val}
    \sup\bigl\{
        |f_\ell(v) - \alpha h_\ell(v)| : 
        v \in \cL(X_{\btau}^{(\ell)})
    \bigr\}
    < \infty .
\end{equation}
Then, there exist letter-coboundaries $c_n \colon \cA_n^* \to \R$ in $X_{\btau}^{(n)}$, for each $n \ge \ell$, such that
\begin{equation}
\label{eq:conclu:lem:balance=>small_abs_val}
    \sup\bigl\{
        | f_\ell(\tau_{\ell,n}(v)) - \alpha h_n(v) + c_n(v) | :
        v \in \cL(X_{\btau}^{(n)})
    \bigr\}
\end{equation}
converges to $0$ as $n \to \infty$.
\end{lemma}

\begin{proof}
Let $\chi_\ell \colon \cA_\ell^* \to \R$ be the morphism defined by 
$\chi_\ell(v) = f_\ell(v) - \alpha h_\ell(v)$ for $v \in \cA_\ell^*$, and let 
$\overline{\chi}_\ell \colon X_{\btau}^{(\ell)} \to \R$ be the continuous map 
$\overline{\chi}_\ell(x) = \chi_\ell(x_0)$ for $x \in X_{\btau}^{(\ell)}$.
For any $x \in X_{\btau}^{(\ell)}$ and $N > 0$, we have
\[
    \Bigl| \sum_{0 \le j < N} \overline{\chi}_\ell(S^j x) \Bigr| =
    |\chi_\ell(x_{[0,N)})| =
    \big| f_\ell(x_{[0,N)}) - \alpha h_\ell(x_{[0,N)}) \big|,
\]
which is uniformly bounded by Hypothesis \eqref{eq:hip:lem:balance=>small_abs_val}.
This implies, by the Gottschalk--Hedlund Theorem (see \Cref{theo:GH}), that
$\overline{\chi}_\ell = \xi_\ell - \xi_\ell \circ S$ for some continuous map 
$\xi_\ell \colon X_{\btau}^{(\ell)} \to \R$.

Recall that $[a]_n = \{x \in X_{\btau}^{(n)} : x_0 = a\}$ for $n \ge 0$ and $a \in \cA_n$.
For each $n \ge \ell$ and $a \in \cA_n$, we choose any value 
$\rho_n(a) \in \xi_\ell(\tau_{\ell,n}([a]_n))$.
Since $X_{\btau}^{(\ell)}$ is compact, $\xi_\ell$ is uniformly continuous, so there exists a sequence $(\varepsilon_n : n \ge \ell)$ converging to zero such that 
\begin{equation}
    \label{eq:1:lem:balance=>small_abs_val}
    |\xi_\ell(x) - \rho_n(a)| \le \varepsilon_n
\end{equation}
for all $n \ge \ell$, $a \in \cA_n$, and $x \in \tau_{\ell,n}([a]_n)$.

Now, for any $n \ge \ell$, $x \in X_{\btau}^{(n)}$, and $N>0$, we can compute
\[
    \chi_\ell(\tau_{\ell,n}(x_{[0,N)})) =
    \sum_{0 \le j < h_n(x_{[0,N)})} \hspace{-0.4cm} \overline{\chi}_\ell(S^j x) =
    \xi_\ell(\tau_{\ell,n}(S^N x))  -  \xi_\ell(\tau_{\ell,n}(x)).
\]
By \eqref{eq:1:lem:balance=>small_abs_val}, the last expression is within $2\varepsilon_n$ of 
$\rho_n(x_N) - \rho_n(x_0)$.
This shows that the morphism $\chi_n \coloneqq \chi_\ell \circ \tau_{\ell,n}$ satisfies $|\chi_n(x_{[0,N)}) - \rho_n(x_N) + \rho_n(x_0)| \le 2\varepsilon_n$ for all $x \in X_{\btau}^{(n)}$ and $N \ge 0$.
Therefore,
\begin{equation}
\label{eq:2:lem:balance=>small_abs_val}
    |\chi_n(v) - \rho_n(v_k) + \rho_n(v_0)|
    \le 2\varepsilon_n
    \quad
    \text{for every $v_0 \cdots v_k \in \cL(X_{\btau}^{(n)})$.}
\end{equation}
This allows us to apply \Cref{lemma:EigCharac:rho2cobs} with  the choice $\vertiii{\cdot} = |\cdot|$, 
the morphisms $(\chi_n : n \ge \ell)$ and the family $(\rho_n(a) : n \ge \ell,\; a \in \cA_n)$.
This yields letter-coboundaries $c_n \colon \cA_n^* \to \R$ in $X_{\btau}^{(n)}$, for each $n > \ell$, and a constant $Q>0$ such that
\begin{equation*}
    |c_n(v) - f_\ell(\tau_{\ell,n}(v)) + \alpha h_n(v)| =
    |c_n(v) - \chi_n(v)| \le Q \cdot 2\varepsilon_{m-1}
\end{equation*}
for all $v \in \cL(X_{\btau}^{(n)})$.
(The bound $2\varepsilon_{n-1}$ comes from \eqref{eq:2:lem:balance=>small_abs_val}.)
Since $(\varepsilon_n : n \ge \ell)$ converges to zero, \eqref{eq:conclu:lem:balance=>small_abs_val} follows from the last inequality.
\end{proof}

\subsection{Proofs of the main statements} \label{subsec:proofsFAR}
We   now have all the necessary material to  prove   in this order Theorems 
\ref{theo:EigCharac:decisive:Mult},  \ref{theo:EigCharac:decisive:positive}, Proposition \ref{cor:EigCharac:FAR:eigs&freqs}, Theorem \ref{theo:EigCharac:FAR:Mult},  and lastly  Proposition \ref{prop:balance_char_Sadic}.

\begin{proof}[Proof of \Cref{theo:EigCharac:decisive:Mult}.]
Assume that for each $n \ge 0$ there exists a letter-coboundary $c_n \colon \cA_n^* \to \R$ such that $(\varepsilon_n : n \ge 0)$ converges to $0$ as $n \to \infty$, where 
\[
  \varepsilon_n =
    \sup \Bigl\{
    \bigl\| c_n(u) - \alpha h_n(u)  \bigr\| :
    u \in \cL(X_{\btau}^{(n)}) \Bigr\}.
\]
We show that $\alpha$ is an eigenvalue  by verifying the criterion in \Cref{theo:EigCharac:Mult}.

By \Cref{lem:cobord&rho}, for every $n \ge 0$ there exists a map $\rho_n \colon \cA_n^* \to \R$ such that  $c_n(ab) = \rho_n(b) - \rho_n(a)$ for every length-$2$ word $ab \in \cL(X_{\btau}^{(n)})$.
Thus, for any $n \ge 0$, any $a \in \cA_{n+1}$, and any prefix $v_0 \cdots v_k \in \prefixes(\tau_n(a))$, we have
\[
\bigl\|\rho_n(v_0) + \alpha h_n(v_0 \cdots v_{k-1}) - \rho_n(v_k)\bigr\|
= \bigl\|c_n(v_0 \cdots v_{k-1}) - \alpha h_n(v_0 \cdots v_{k-1})\bigr\|
\le \varepsilon_n.
\]
Since $(\varepsilon_n)$ converges to $0$, Condition \eqref{eq:convergence:theo:EigCharac:Mult} of \Cref{theo:EigCharac:Mult} holds.
Therefore, $\alpha$ is an eigenvalue of $X$.
\smallskip

Conversely, assume that $\alpha$ is an eigenvalue of $X$.
By \Cref{theo:EigCharac:Mult}, there exists a family of real numbers $\brho = (\rho_n(a) : n \ge 0,\, a \in \cA_n)$ such that $(\varepsilon_n : n \ge 0)$ converges to $0$, where 
\[
\varepsilon_n \coloneqq 
    \sup \Bigl\{
    \bigl\| \rho_n(v_0) - \rho_n(v_k) 
        + \alpha h_n(v_0 \cdots v_{k-1}) \bigr\| :
        v_0 \cdots v_k \in \cL(X_{\btau}^{(n)})
        \Bigr\}.
\]
Define the morphism $\chi_n \colon \cA_n^* \to \R$ by $\chi_n = \alpha h_n$ for each $n \ge 0$, which satisfies
\begin{equation}
    \label{eq:1theo:EigCharac:decisive:Mult}
    \|\rho_n(v_0) + \chi_n(v_0\cdots v_{k-1}) - \rho_n(v_k)\| 
    \le \varepsilon_n 
    \quad
    \text{for all $v_0 \cdots v_k \in \cL(X_{\btau}^{(n)})$ and $n \ge 0$.}
\end{equation}
Then, applying \Cref{lemma:EigCharac:rho2cobs} with $\vertiii{\cdot} = \|\cdot\|$, the morphisms $\chi_n = \alpha h_n$, and the family $\brho$, we obtain a constant $Q > 0$ and, for each $n \ge 0$, a letter-coboundary $c_n \colon \cA_n^* \to \R$ in $X_{\btau}^{(n)}$ such that $\|c_n(v) - \alpha h_n(v)\| \le Q \varepsilon_{n-1}$ for all $v \in \cL(X_{\btau}^{(n)})$.
(The term $\varepsilon_{n-1}$ comes from \eqref{eq:1theo:EigCharac:decisive:Mult}.)
Since $(\varepsilon_n : n \ge 0)$ converges to $0$, Condition \eqref{eq:1:theo:EigCharac:decisive:Mult} of \Cref{theo:EigCharac:decisive:Mult} follows.
\end{proof}

\begin{proof}[Proof of \Cref{theo:EigCharac:decisive:positive}.]
We follow the approach used in the proof of \Cref{theo:EigCharac:decisive:Mult}, relying this time on \Cref{theo:EigCharac:positive} instead of \Cref{theo:EigCharac:Mult}.

Assume that Condition \eqref{eq:1:theo:EigCharac:FAR:positive} holds for a sequence of letter-coboundaries $(c_n \colon \cA_n^* \to \R : n \ge 0)$.  
By \Cref{lem:cobord&rho}, there exists for each $n \ge 0$ a map $\rho_n \colon \cA_n^* \to \R$ such that $c_n(ab) = \rho_n(b) - \rho_n(a)$ for every length-$2$ word $ab \in \cL(X_{\btau}^{(n)})$.
Then, for every $n \ge 0$, $a \in \cA_{n+2}$ and $v_0 \dots v_k \in \prefixes(\tau_{n,n+2}(a))$, we have
\[  \bigl\|\rho_n(v_0) + \alpha h_n(v_0 \cdots v_{k-1}) - \rho_n(v_k)\bigr\|
    = \bigl\|c_n(v_0 \cdots v_{k-1}) - \alpha h_n(v_0 \cdots v_{k-1})\bigr\|.
\]
Hence Condition \eqref{eq:1:theo:EigCharac:FAR:positive} is equivalent to Condition \eqref{eq:theo:EigCharac:positive} of \Cref{theo:EigCharac:positive}.  
By that theorem, $\alpha$ is an eigenvalue of $X$.
Note that the only hypothesis required here is the primitivity of $\btau$, which suffices to apply the relevant implication of \Cref{theo:EigCharac:positive}.
\medskip

Conversely, assume that $\alpha$ is an eigenvalue of $X$.
It is enough to show that for each \emph{odd} $n \ge 0$ there exists a coboundary  
$c_n \colon \cA_n^* \to \R$ in $X_{\btau}^{(n)}$ and a constant $\varepsilon_n > 0$ such that $\|c_n(v) - \alpha h_n(v)\| \le \varepsilon_n$ for all $a \in \cA_n,\ v \in \prefixes(\tau_{n,n+2}(a))$, and such that  the series $\sum_{\text{$n$ odd}} \varepsilon_n$ is finite.  
The same argument then applies to the even levels, yielding Condition \eqref{eq:1:theo:EigCharac:FAR:positive}, since the sum of two convergent series is again convergent.

We consider the contraction $\btau' = (\tau_{2\ell,2\ell+2} : \ell \ge 0)$ of $\btau$ along the even levels.  
By hypothesis, each $\tau_{2\ell}(a)$ sees every length-$2$ word of $\cL(X_{\btau}^{(2\ell)})$, so $\btau'$ is positive.  
We may therefore apply \Cref{theo:EigCharac:positive} to $\btau'$, obtaining a family $\brho = (\rho_{2\ell}(a) : a \in \cA_{2\ell},\, \ell \ge 0)$ such that the sequence $(\varepsilon_{2\ell} : \ell \ge 0)$ has finite sum, where
\[  \varepsilon_n =
    \max \bigl\{
        \|\rho_n(v_0) + \alpha h_n(v) - \rho_n(v_k)\| :
        a \in \cA_{n+4},\ v_0 \cdots v_k \in \prefixes(\tau_{n,n+4}(a))
        \bigr\}
\]
for each $n \ge 0$.

Now, the hypothesis ensures that $\tau_n(a)$ sees every length-2 word $ab \in \cL(X_{\btau}^{(n)})$ for each $a \in \cA_n$; equivalently, $r'(n) = n+1$ for all $n \ge 0$.
Hence, applying \Cref{lemma:EigCharac:rho2cobs} with  
$\vertiii{\cdot} = \|\cdot\|$, the morphisms $\chi_n \coloneqq \alpha h_n$, and the family $\brho$ (together with the directive sequence $\btau$), we obtain a constant $Q > 0$ and, for each \emph{odd} $n \ge 0$, a coboundary $c_n \colon \cA_n^* \to \R$ such that $\|c_n(v) - \alpha h_n(v)\| \le Q \varepsilon_{n-1}$ for every odd $n \ge 0$, $a \in \cA_{n+3}$, and $v \in \prefixes(\tau_{n,n+3}(a))$.
In particular, this holds for all $v \in \prefixes(\tau_{n,n+2}(a))$ and $a \in \cA_{n+2}$.
Therefore, since the sequence $(\varepsilon_{2\ell} : \ell \ge 0)$ has finite sum,
\[ \sum_{\text{$n$ odd}} \max\bigl\{
    \|c_n(v) - \alpha h_n(v)\| :
    a \in \cA_n,\, v \in \prefixes(\tau_{n,n+2}(a))
    \bigr\} \le 
    \sum_{\text{$n$ odd}} \varepsilon_{n-1} \lt \infty,
    \]
which establishes the claim.
\end{proof}

\begin{proof}[Proof of \Cref{cor:EigCharac:FAR:eigs&freqs}.]
The proof is based on applying \Cref{lemma:EigCharac:gamma2freqs&eig}.
Since $\alpha$ is an eigenvalue, \Cref{theo:EigCharac:decisive:Mult} yields letter-coboundaries $c_n \colon \cA_n^* \to \R$ in $X_{\btau}^{(n)}$ for each $n \ge 0$, such that
\[  \varepsilon_n \coloneqq 
    \sup\Bigl\{
        \bigl\| c_n(u) - \alpha h_n(u)\bigr\| :
        u \in \cL(X_{\btau}^{(n)})
    \Bigr\}
\]
converges to $0$ as $n \to \infty$.
Choose $n$ large enough so that $\varepsilon_n < 1/3$, and define a morphism $\gamma'_n \colon \cA_n^* \to \R$ by $\gamma'_n \coloneqq c_n - \alpha h_n$.
Then, $\|\gamma'_n(v)\| < 1/3$ for all $v \in \cL(X_{\btau}^{(n)})$.
This permits to apply \Cref{lem:small_gamma=>morphism}, which states that the morphism $\gamma_n \colon \cA_n^* \to \R$ given by $\gamma_n(a) = \{\gamma'_n(a)\}$ for $a \in \cA_n$, satisfies $|\gamma_n(v)| = \|\gamma'_n(v)\|$ for all $v \in \cL(X_{\btau}^{(n)})$.
This is precisely the hypothesis of \Cref{lemma:EigCharac:gamma2freqs&eig}.
Therefore, by that lemma, $\alpha = \sum_{a \in \cA_n} w_a\, \mu(B_n(a))$ for all invariant measures $\mu \in \cM(X,S)$, where the coefficients $w_a \coloneqq \alpha h_n(a) - c_n(a) - \gamma_n(a)$ are integers.
\end{proof}

We now turn to the proof of \Cref{theo:EigCharac:FAR:Mult} and \Cref{prop:balance_char_Sadic}.  
The final element we need is the following lemma, which is the only one relying solely on $(|\cA_n| : n \ge 0)$ being bounded.
Note that the relabeling step is only necessary when the alphabets $\cA_n$ are not all identical.

\begin{lemma}
\label{existence_limit_coboundary}
Let $\btau = (\tau_n \colon \cA_{n+1}^* \to \cA_n^* : n\geq0)$ be a directive sequence such that $(|\cA_n| : n \ge 0)$ is bounded.
For each $n \ge 0$, consider letter-coboundaries $c_n \colon \cA_n^* \to \R$ in $X_{\btau}^{(n)}$, and assume that 
\[
\sup\{|c_n(a)| : a \in \cA_n,\ n \ge 0\} < \infty.
\]
Then, after possibly relabeling the alphabets $\cA_n$, there exist an increasing sequence $(n_\ell : \ell \ge 0)$ with $\cA_{n_\ell} = \cA_{n_0}$ for all $\ell \ge 0$, and a morphism $c \colon \cA_{n_0}^* \to \R$, which is a letter-coboundary in $X_{\btau}^{(n_\ell)}$ for all $\ell \ge 0$, such that:
\begin{enumerate}
    \item $c \circ \tau_{n_\ell,n_k} = c$ for all $k > \ell \ge 0$.
    \item $\sup\bigl\{ |c(v) - c_{n_\ell}(\tau_{n_{\ell},n_{\ell+1}}(v))| : v \in \cL(X_{\btau}^{(n_{\ell+1})}) \bigr\}$ converges to $0$ as $\ell \to \infty$.
\end{enumerate}
\end{lemma}

\begin{proof}
Since $(|\cA_n| : n \ge 0)$ is bounded, we may, after relabeling the alphabets $\cA_n$, select an increasing sequence $(n_\ell : \ell \ge 0)$ such that the alphabets $\cA_n$ and the extension graphs stabilize, \emph{i.e.,}
\begin{enumerate}
    \item[(i)] $\cA_{n_\ell} = \cA$ for all $\ell \ge 0$;
    \item[(ii)] $\Gamma_{\!X_{\btau}^{(n_\ell)}}(\varepsilon)$ equals a fixed graph $\Gamma$ for all $\ell \ge 0$.
\end{enumerate}
Then, applying \Cref{lem:cobord&rho} to levels $n_\ell$ gives, for each $\ell \ge 0$, a map $\rho_\ell \colon \cA \to \R$ satisfying  
$c_{n_\ell}(a) = \rho_\ell(b) - \rho_\ell(a)$ for all edges $(a_L,b_R)$ of $\Gamma$.

For $k \ge \ell \ge 0$, define $\First_{\ell,k} \colon \cA \to \cA$ by $\First_{\ell,k}(a) = \First(\tau_{\ell,k}(a))$, where $\First(t)$ denotes the first letter of a word $t$ and $\tau_{\ell,\ell}(a) = a$.
These maps satisfy $\First_{\ell,k} \circ \First_{k,m} = \First_{\ell,m}$ for all $m \ge k \ge \ell \ge 0$.
For each $\ell \ge 0$ and $a \in \cA$, the sequence $(\First_{\ell,k}(a) : k \ge \ell)$ is eventually constant; denote its eventual value by $\First_\ell(a)$.
Then, the set $\{ (\First_\ell(a))_{a \in \cA} : \ell \ge 0 \}$ is finite, so after passing to an inductively chosen subsequence we may assume that there exists $\First_\infty \colon \cA \to \cA$ such that
\begin{enumerate}
    \item[(iii)] $\First_{\ell,k}(a) = \First_\infty(a)$ for all $k > \ell \ge 0$ and $a \in \cA$.
\end{enumerate}
In particular,
\begin{equation}
    \label{eq:00:existence_limit_coboundary}
    \First_\infty \circ \First_{\ell,k} = \First_\infty
    \quad \text{for all $k \ge \ell \ge 0$,}
\end{equation}
since (iii) allows us to write, for any $k \ge m \ge \ell$,
\[
\First_\infty(\First_{\ell,k}(a)) =
\First_{\ell,m}(\First_{m,k}(a)) =
\First_{\ell,m}(a) =
\First_\infty(a).
\]

Next, we obtain uniform bounds on the values of $\rho_\ell$.
The hypothesis ensures that there exists $Q > 0$ that upper bounds $|c_n(a)|$ for all $n \ge 0$ and $a \in \cA_n$.
Since $\rho_\ell$ is defined up to an additive constant (see \Cref{thm:manifold}), we can fix any $\bar a \in \cA$ and normalize $\rho_\ell$ so that $\rho_\ell(\bar a) = 0$ for all $\ell \ge 0$.
Inspecting the extension graph $\Gamma_{\!X_{\btau}^{(n_\ell)}}(\varepsilon)$ then yields $|\rho_\ell(a)| \le |\cA|\,Q$ for all $\ell \ge 0$ and $a \in \cA$.

Thus, the family $\{\rho_\ell(a) : \ell \ge 0,\ a \in \cA\}$ is uniformly bounded.
Passing to a subsequence if needed, we may assume that there exists $\rho \colon \cA \to \R$ such that:
\begin{enumerate}
    \item[(iv)] For every $a \in \cA$, the sequence $(\rho_\ell(a) : \ell \ge 0)$ converges to $\rho(a)$.
\end{enumerate}

We now construct the limiting coboundary $c$.
On each connected component of $\Gamma$, every $\rho_\ell \circ \First_{\ell,k}$ is constant on its right-hand side, since $\rho_\ell$ is.
Taking the limit and using (iii) shows that $\rho \circ \First_\infty$ is also constant there.
Then, by \Cref{lem:cobord&rho}, the map $\rho \circ \First_\infty$ induces a morphism $c \colon \cA^* \to \R$ defined by 
$c(a) = \rho(\First_\infty(b)) - \rho(\First_\infty(a))$  
for every edge $(a_L,b_R)$ of $\Gamma$.
Property (ii) ensures that $c$ is a letter-coboundary in $X_{\btau}^{(n_\ell)}$ for all $\ell \ge 0$.
\smallskip 

We now verify Item (1).
Let $k > \ell \ge 0$ and let $(a_L,b_R)$ be an edge of $\Gamma$.
Unfolding the definitions gives the identity
\[
c(\tau_{n_\ell,n_k}(a)) =
\rho\big(\First_\infty(\First_{n_\ell,n_k}(b))\big)
-
\rho\big(\First_\infty(\First_{n_\ell,n_k}(a))\big).
\]
By \eqref{eq:00:existence_limit_coboundary}, $\First_\infty \circ \First_{n_\ell,n_k} = \First_\infty$, so
\[
c(\tau_{n_\ell,n_k}(a)) =
\rho(\First_\infty(b)) - \rho(\First_\infty(a)) = c(a).
\]

It remains to prove Item (2).
We set, for each $\ell \ge 0$, $\varepsilon_\ell = \max\{ |\rho(a) - \rho_\ell(a)| : a \in \cA \}$.
By (iv), $(\varepsilon_\ell : \ell \ge 0)$ tends to $0$ as $\ell \to \infty$.
Now, for any $v_0 \cdots v_l \in \cL(X^{(n_{\ell+1})}_{\btau})$ we can compute:
\[
c(v_0 \cdots v_{l-1}) -
c_{n_\ell} \circ \tau_{n_\ell,n_{\ell+1}}(v_0 \cdots v_{l-1})
=
\big(\rho(\First_\infty(v_l)) - \rho_{\ell,\ell+1}(v_l)\big)
-
\big(\rho(\First_\infty(v_0)) - \rho_{\ell,\ell+1}(v_0)\big).
\]
Since $\rho_{\ell,\ell+1} = \rho_\ell \circ \First_{\ell,\ell+1} =  \rho_\ell \circ \First_\infty$ by (iii), each term on the right-hand side has absolute value at most $\varepsilon_\ell$.
Thus ,
$|c(v) - c_\ell(\tau_{n_\ell,n_{\ell+1}}(v))| \le 2\varepsilon_\ell$  
for all $v \in \cL(X_{\btau}^{(n_\ell)})$, from which Item (2) follows
\end{proof}

\begin{proof}[Proof of \Cref{theo:EigCharac:FAR:Mult}.]
We prove $(1) \Rightarrow (2) \Rightarrow (3) \Rightarrow (1)$.

Assume first that $\alpha$ is an additive eigenvalue of $X$.
By \Cref{theo:EigCharac:decisive:Mult}, there exist letter-coboundaries  
$c_n \colon \cA_n^* \to \R$ in $X_{\btau}^{(n)}$ such that
\[
\varepsilon_n \coloneqq 
\sup \bigl\{
\| c_n(u) - \alpha h_n(u) \| :
u \in \cL(X_{\btau}^{(n)}) \bigr\}
\]
converges to $0$ as $n \to \infty$.
This condition depends only on the values of $c_n$ modulo $\Z$, so by \Cref{lem:cob a la host} we may assume that
\begin{equation}
\label{eq:0:theo:EigCharac:FAR:Mult}
\sup\{ |c_n(a)| : a \in \cA_n,\ n \ge 0 \} \le 1 < \infty.
\end{equation}
Observe that for any $m > n \ge 0$ and $u \in \cL(X_{\btau}^{(m)})$, we can estimate
\begin{equation}
\label{eq:1:theo:EigCharac:FAR:Mult}
\|c_n(\tau_{n,m}(u)) - c_m(u)\| \le 
\|c_n(\tau_{n,m}(u)) - \alpha h_n(\tau_{n,m}(u))\|  +
\|c_m(u) - \alpha h_m(u)\|  \le 
\varepsilon_n + \varepsilon_m.
\end{equation}
Now, since $(|\cA_n| : n \ge 0)$ is bounded and \eqref{eq:0:theo:EigCharac:FAR:Mult} holds, we may apply \Cref{existence_limit_coboundary} to the sequence $(c_n : n \ge 0)$.
We obtain an increasing sequence $(n_\ell : \ell \ge 0)$ with $\cA_{n_\ell} = \cA_{n_0}$ for all $\ell \ge 0$, and a morphism  
$c \colon \cA_{n_0}^* \to \R$ that is a letter-coboundary in $X_{\btau}^{(n_\ell)}$ for all $\ell \ge 0$, such that
\begin{enumerate}[label=(\arabic*)]
    \item $c = c \circ \tau_{n_\ell,n_k}$ for all $k > \ell \ge 0$;
    \item $\varepsilon'_\ell \coloneqq \sup\{ |c(v) - c_{n_\ell}(\tau_{n_\ell,n_{\ell+1}}(v))| : v \in \cL(X_{\btau}^{(n_\ell)}) \}$ tends to $0$ as $\ell \to \infty$.
\end{enumerate}

We now show that $c$ satisfies Item (2) of the theorem.
Given $m \ge n_1$, let $\ell(m)$ be the largest integer $\ell$ such that $n_{\ell+1} \le m$.
We claim that
\begin{equation}
\label{eq:claim:theo:EigCharac:FAR:Mult}
\|c \circ \tau_{n_0,m}(v) - \alpha h_m(v)\|
\le 
\varepsilon'_{\ell(m)} + \varepsilon_m + \varepsilon_{\ell(m)}
\quad\text{for all $m \ge n_1$ and $v \in \cL(X_{\btau}^{(m)})$.}
\end{equation}
Since $(\varepsilon'_\ell)$ and $(\varepsilon_m)$ both converge to $0$, this would prove Item (2).

Let us prove this claim. Fix $m \ge n_1$ and $v \in \cL(X_{\btau}^{(m)})$.
We abbreviate $u = \tau_{n_{\ell(m)+1},m}(v)$ and $\ell = \ell(m)$.
Then, by (i),
\[ c \circ \tau_{n_0,m}(v)
    = c \circ \tau_{n_0,n_{\ell+1}}(u)
    = c(u).
\]
Hence, by (ii)
\[  \|c \circ \tau_{n_0,m}(v) - c_{n_\ell}(\tau_{n_\ell,n_{\ell+1}}(u))\| = 
    \|c(u) - c_{n_\ell}(\tau_{n_\ell,n_{\ell+1}}(u))\| \le
    |c(u) - c_{n_\ell}(\tau_{n_\ell,n_{\ell+1}}(u))| 
    \le \varepsilon'_\ell.
\]
Since $c_{n_\ell}(\tau_{n_\ell,n_{\ell+1}}(u)) = c_{n_\ell}(\tau_{n_\ell,m}(v))$, the bound in \eqref{eq:1:theo:EigCharac:FAR:Mult} yields
\[  \|c_{n_\ell}(\tau_{n_\ell,n_{\ell+1}}(u)) - c_m(v)\|
    \le \varepsilon_{n_\ell} + \varepsilon_m.
\]
Combining these two bounds gives
\[  \|c \circ \tau_{n_0,m}(v) - c_m(v)\| \le 
    \varepsilon'_\ell + \|c_{n_\ell}(\tau_{n_\ell,n_{\ell+1}}(u)) - c_m(v)\| \le 
    \varepsilon'_\ell + \varepsilon_{n_\ell} + \varepsilon_m,
    \]
proving \eqref{eq:claim:theo:EigCharac:FAR:Mult} and thus Item (2).
\medskip

We continue with $(2) \Rightarrow (3)$.
Assume that there exist coboundaries $c_n \colon \cA_n^* \to \R$ in $X_{\btau}^{(n)}$, for each $n \ge 0$, such that
\[  \varepsilon_m \coloneqq 
    \sup\bigl\{
    \|c_n(\tau_{n,m}(u)) - \alpha h_m(u)\| :
    u \in \cL(X_{\btau}^{(m)}) \bigr\}
\]
converges to $0$ as $m \to \infty$.
By replacing $c_n$ and $n$ by $c_n \circ \tau_{n,m}$ and $m$ for some large $m > n$ if needed, we may assume $\varepsilon_m < 1/3$ for all $m \ge n$.
This enables us to apply, for each $m \ge n$, \Cref{lem:small_gamma=>morphism} with the morphism  
$\gamma'_m \coloneqq c_n \circ \tau_{n,m} - \alpha h_m$.
We deduce that the morphism $\gamma_m \colon \cA_m^* \to \R$ defined by  
$\gamma_m(a) = \{c_n \circ \tau_{n,m}(a) - \alpha h_m(a)\}$ for $a \in \cA_m$, satisfies 
$\gamma_m(v) = \{\gamma'_m(v)\}$ for all $v \in \cL(X_{\btau}^{(m)})$.
In particular, $|\gamma_m(v)| = \|\gamma'_m(v)\|$ for all such $v$, so, since $(\varepsilon_m)$ tends to $0$, Item (3) follows.
\medskip

Finally, we assume that Item (3) holds and prove that $\alpha$ is an eigenvalue of $X$.
By hypothesis, there exist $n \ge 0$ and a letter-coboundary $c_n$ such that the morphism  
$\gamma_n \colon \cA_n^* \to \R$, defined by $\gamma_n(a) = \{c_n(a) - \alpha h_n(a)\}$, satisfies
\[  \varepsilon_m \coloneqq 
    \sup\bigl\{ |\gamma_n(\tau_{n,m}(u))| : u \in \cL(X_{\btau}^{(m)}) \bigr\}
\]
converges to $0$ as $m \to \infty$.
By replacing $c_n$ and $n$ by $c_n \circ \tau_{n,m}$ and $m$ for some large $m > n$, we may assume $\varepsilon_n < 1/3$.
Thus, $|\gamma_n(u)| < 1/3$ for all $u \in \cL(X_{\btau}^{(n)})$.
We can thus apply \Cref{lem:small_gamma=>morphism} with $\gamma'_n \coloneqq \gamma_n$, obtaining 
$\gamma_n(v) = \{\gamma(v)\}$ for all $v \in \cL(X_{\btau}^{(n)})$.
This is exactly the hypothesis of \Cref{lemma:EigCharac:gamma2freqs&eig}, which then implies that $\alpha$ is an eigenvalue of $X$.
\end{proof}

We end this section with the proof of \Cref{prop:balance_char_Sadic}.

\begin{proof}[Proof of \Cref{prop:balance_char_Sadic}.]
We assume first that $f$ is balanced.
Set 
$\alpha \coloneqq \int f \,\mathrm{d}\mu 
    = \sum_{a \in \cA_\ell} f_\ell(a)\,\mu(B_{\ell}(a))$ and observe that for every $x \in X_{\btau}^{(\ell)}$ and $N>0$,
\[
  \sum_{0 \le k < h_\ell(x_{[0,N)})}
    f(S^k \tau_{0,\ell}(x)) 
    - \alpha\, h_\ell(x_{[0,N)})    
  = f_\ell(x_{[0,N)}) - \alpha\, h_\ell(x_{[0,N)}).
\]
The left-hand side is uniformly bounded because $f$ is balanced. Hence,
\[
  \sup\bigl\{
    |f_\ell(v) - \alpha h_\ell(v)| :
    v \in \cL(X_{\btau}^{(\ell)})
    \bigr\} < \infty.
\]
This is the hypothesis of \Cref{lem:balance=>small_abs_val}.
Thus, for each $n>\ell$, there exists a letter-coboundary 
$c_n \colon \cA_n^* \to \R$ in $X_{\btau}^{(n)}$ such that
\begin{equation}
    \label{eq:1:prop:balance_char_Sadic}
    \varepsilon_n \coloneqq 
    \sup\bigl\{ 
    | f_\ell(\tau_{\ell,n}(v)) - \alpha h_n(v) + c_n(v) |  :
    v \in \cL(X_{\btau}^{(n)}) 
    \bigr\}
\end{equation}
tends to $0$ as $n\to\infty$.

We now want to follow the strategy used in the proof of  the implication $(1)\Rightarrow (2)$ of \Cref{theo:EigCharac:FAR:Mult}.  
More precisely, we are going to apply \Cref{existence_limit_coboundary} to obtain a single letter-coboundary $c$ such that $c \circ \tau_{\ell,m}$ approximates $c_m$, and then derive \eqref{eq:hip_gamma:prop:balance_char_Sadic} from this.

Since $(|\cA_n|)$ is bounded and \eqref{eq:1:prop:balance_char_Sadic} holds, we can apply \Cref{existence_limit_coboundary} to obtain an increasing sequence $(n_\ell : \ell \ge 0)$ with $\cA_{n_\ell}=\cA_{n_0}$ for all $\ell$, and a morphism  
$c \colon \cA_{n_0}^* \to \R$, which is a letter-coboundary in each $X_{\btau}^{(n_\ell)}$, such that:
\begin{enumerate}[label=(\roman*)]
    \item $c = c\circ \tau_{n_\ell,n_k}$ for all $k>\ell\ge 0$;
    \item $\varepsilon'_\ell \coloneqq 
      \sup\bigl\{ |c(v) - c_{n_\ell}(\tau_{n_\ell,n_{\ell+1}}(v))| : 
        v\in \cL(X_{\btau}^{(n_\ell)}) \bigr\} \to 0$ as $\ell \to \infty$.
\end{enumerate}

For each $m \ge n_1$, we let $\ell(m)$ be the largest $\ell$ with $n_{\ell+1}\le m$.
We claim that
\begin{equation}
\label{eq:claim:prop:balance_char_Sadic}
\bigl| 
c \circ \tau_{n_0,m}(v)
- f_\ell(\tau_{\ell,m}(v)) 
+ \alpha\, h_m(v)
\bigr|
\le \varepsilon'_{\ell(m)} + \varepsilon_{n_{\ell(m)}}
\quad \text{for all $m\ge n_1$ and $v\in \cL(X_{\btau}^{(m)})$}.
\end{equation}
Since both $\varepsilon'_\ell \to 0$ and $\varepsilon_{n_{\ell}} \to 0$ as $\ell \to \infty$, \eqref{eq:hip_gamma:prop:balance_char_Sadic} follows from \eqref{eq:claim:prop:balance_char_Sadic}.

Let us prove \eqref{eq:claim:prop:balance_char_Sadic}.
Fix $m\ge n_1$ and $v\in \cL(X_{\btau}^{(m)})$.
Let $\ell=\ell(m)$ and set $u=\tau_{n_{\ell+1},m}(v)$.
By (i), $c\circ \tau_{n_0,m}(v) = c(u)$, so, by (ii),
\begin{equation}
\label{eq:2:prop:balance_char_Sadic}
\bigl|c\circ \tau_{n_0,m}(v) - c_{n_\ell}(\tau_{n_\ell,m}(v))\bigr|
 = \bigl|c(u) - c_{n_\ell}(\tau_{n_\ell,n_{\ell+1}}(u))\bigr|
\le \varepsilon'_\ell.
\end{equation}
Let now $u'=\tau_{n_\ell,m}(v)$.  
By \eqref{eq:1:prop:balance_char_Sadic}, $c_{n_\ell}(u')$ differs from
$f_\ell(\tau_{\ell,n_\ell}(u')) - \alpha h_{n_\ell}(u')$ by at most $\varepsilon_{n_\ell}$.
Using this to substitute $c_{n_\ell}(u')$ in \eqref{eq:2:prop:balance_char_Sadic} yields
\[
\bigl|
c\circ\tau_{n_0,m}(v)
 - f_\ell(\tau_{\ell,m}(v))
 + \alpha\, h_m(v)
\bigr|
\le \varepsilon'_\ell + \varepsilon_{n_\ell},
\]
proving \eqref{eq:claim:prop:balance_char_Sadic}.
\medskip

For the converse, suppose that there exist $n\ge \ell$, a letter-coboundary $c_n$ on $X_{\btau}^{(n)}$, and $\alpha\in\R$ such that the morphism $\gamma_n \colon \cA_n^* \to \R$, defined by $\gamma_n(a) = f_\ell(\tau_{\ell,n}(a)) - \alpha h_n(a) + c_n(a)$ for $a \in \cA_n$, satisfies that
\[
\varepsilon_m \coloneqq
\sup\bigl\{ |\gamma_n(\tau_{n,m}(u))| : 
    u\in \cL(X_{\btau}^{(m)}) \bigr\}
\]
tends to 0 as $m \to \infty$.
This yields, for any $x\in X_{\btau}^{(n)}$ and $N>0$,
\begin{equation*}
\Big|\sum_{0 \le k < h_n(x_{[0,N)})}
    f(S^k\tau_{0,n}(x)) 
 - \alpha\, h_n(x_{[0,N)})  \Big|
 = |\gamma_n(x_{[0,N)}) - c_n(x_{[0,N)})| \le 
 \varepsilon_n + |c_n(x_{[0,N)})|.
\end{equation*}
We can also provide   uniform  bounds for  $|c_n(x_{[0,N)})|$ as follows.
By \Cref{lem:cobord&rho}, there exists 
$\rho_n \colon \cA_n \to \R$ such that 
$c_n(a) = \rho_n(b) - \rho_n(a)$ for all $ab \in \cL(X_{\btau}^{(n)})$.
Hence
\[
|c_n(x_{[0,N)})|
 = |\rho_n(x_N) - \rho_n(x_0)|
 \le 2 \max\{|\rho_n(a)| : a \in \cA_n\}
 \eqqcolon Q_0.
\]
Therefore,
\begin{equation}
\label{eq:4:prop:balance_char_Sadic}
\Bigl|
\sum_{0 \le k < h_n(x_{[0,N)})}
    f(S^k\tau_{0,n}(x)) 
 - \alpha\, h_n(x_{[0,N)})
\Bigr|
\le Q_0 + \varepsilon_n
\end{equation}
for all $x\in X_{\btau}^{(n)}$ and $N>0$.
We use this to show that $f$ is balanced.

Let $y \in X$ and $M \ge 0$.
The ergodic sum $\sum_{0\le k<M} f(S^k y) - \alpha M$ may not have the form of that in \eqref{eq:4:prop:balance_char_Sadic}, but by extending this sum by at most  $H \coloneqq \max\{h_n(a) : a \in \cA_n\}$ terms on each side, we obtain one that does.
Each added term has absolute value at most $F \coloneqq \max\{|f_\ell(\tau_{\ell,n}(a))| : a \in \cA_n\} + |\alpha|$, so \eqref{eq:4:prop:balance_char_Sadic} gives the bound
\[
\Bigl|
\sum_{0\le k<M} f(S^k y) - \alpha M
\Bigr|
\le \varepsilon_n + Q_0 + 2HF.
\]
Thus, $f$ is balanced on $(X,S)$.
\end{proof}

\section{On the combinatorics of letter-coboundaries}
\label{sec:studyofcoboundaries}

We now  exploit the connection between letter-coboundaries and extension graphs.
This allows us to succinctly characterize every letter-coboundary that can appear in a subshift and to give natural conditions under which all such coboundaries are trivial.

\subsection{Discrepancy  of substitutive systems} \label{subsec:discrepancy_subst}

In this section, we derive from \Cref{prop:balance_char_Sadic} a characterization of primitive substitutions that generate a letter-balanced subshift.
This result provides an alternative to Adamczewski’s characterization~\cite{adam03,adam05};   prefix-suffix automata  used in  ~\cite{adam03,adam05} are replaced here  by extension graphs.

Let us note that the statements in this section follow from the results given  in \Cref{sec:carac_finite_rank}, and could be extended to primitive, recognizable directive sequences with bounded alphabets.
We refrain from doing it here to avoid further technical overhead.

Let us fix an aperiodic primitive substitution $\tau \colon \cA^* \to \cA^*$. 
See \Cref{susbsec:morphisms} for the definition  of  the (minimal) shift $X_{\tau}$.
Recall that its incidence matrix is the $\cA \times \cA$ matrix $M_\tau$ whose entry $M_\tau(b,a)$ counts the number of occurrences of $b$ in $\tau(a)$.
This matrix acts on $\R^\cA$ on the right: for any row vector $\vec{v}=(v_b)_{b\in\cA}\in\R^\cA$,
\[
  \vec{v} M_\tau = \bigl(\sum_{b\in\cA} v_b\, M_\tau(b,a) \bigr)_{a\in\cA}.
\]
Its {\em stable space} is the row-vector space
\[
  \vec{V}_{X_\tau} = \{\, \vec{v}\in\R^\cA : \vec{v} M_\tau^n \to 0 \text{ as } n\to\infty \,\}.
\]
Equivalently, $\vec V_{X_\tau}$ is the span of the left generalized eigenspaces of $M_\tau$ for eigenvalues $\eta$ with $|\eta|<1$.
In particular, for every $\vec{v}\in\vec V_{X_\tau}$, the convergence of  $\vec{v} M_\tau^n$ to the vector $\vec{0}$  is exponentially fast.

Next,  we define the {\em coboundary vector space} as
\[
  \vec{\cC}_{X_\tau} =
  \bigl\{ (c(a))_{a\in\cA}\in\R^\cA \mid
        c\colon \cA^*\to\R \text{ is a letter-coboundary on } X_\tau \bigr\}.
\]
For an illustration of the computations performed here, see  \emph{e.g.}
the example developed in \Cref{subsec:nontrivialdynamical}.

We first give a succinct description of $\vec{\cC}_{X_\tau}$ using \Cref{thm:manifold}.
For each connected component $K$ of $\Gamma_{X_\tau}(\varepsilon)$, let $\vec{\mathbf{1}}^R_K$ and  $\vec{\mathbf{1}}^L_K$ in $\R^\cA$ be  the vectors defined by
\[
  \vec{\mathbf{1}}^R_K(a) \!=\!
  \begin{cases}
    1 & \text{if } a_R \in K,\\
    0 & \text{otherwise,}
  \end{cases}
  \quad\text{and}\quad
  \vec{\mathbf{1}}^L_K(a) \!=\!
  \begin{cases}
    1 & \text{if } a_L \in K,\\
    0 & \text{otherwise.}
  \end{cases}
\]

\begin{proposition}
\label{subs:generators_for_coboundary_space}
The space $\vec{\cC}_{X_\tau}$ is generated by the row vectors
$\vec{\mathbf{1}}^R_K - \vec{\mathbf{1}}^L_K$, where $K$ ranges over the
connected components of $\Gamma_{X_\tau}(\varepsilon)$.
\end{proposition}

\begin{proof}
Let $\cK$ be the set of connected components of 
$\Gamma_{X_\tau}(\varepsilon)$. Let $\cF_{X_\tau}, \cC_{X_\tau}$ and $\eta\colon\cF_{X_\tau}\to\cC_{X_\tau}$ as in \Cref{thm:manifold}.
Since any letter-coboundary is uniquely determined by its value on letters, we can identify $\vec{\cC}_{X_\tau}$ with
$\cC_{X_\tau}$.

 Observe that the space $\cF_{X_\tau}$ is generated by the vectors $\vec{\mathbf{1}}^R_K$, $K \in \cK$.
Now, for any any $ab \in \cL(X_\tau)$, we have
\[  \eta(\vec{\mathbf{1}}^R_K)(a) = 
    \vec{\mathbf{1}}^R_K(b) - \vec{\mathbf{1}}^R_K(a) = 
    \vec{\mathbf{1}}^L_K(a) - \vec{\mathbf{1}}^R_K(a).
\]
Hence, as $\eta$ maps $\cF_{X_\tau}$ onto $\vec{\cC}_{X_\tau}$, the vectors
$\vec{\mathbf{1}}^R_K - \vec{\mathbf{1}}^L_K$ generate $\vec{\cC}_{X_\tau}$.
\end{proof}

We now state the main result of this section.

\begin{theorem}
\label{balanced=>spaces_decomposition}
Let $\tau \colon \cA^* \to \cA^*$ be a primitive aperiodic substitution.
Then, $X_\tau$ is letter-balanced if and only if the vector-space sum $\vec V_{X_\tau}+\vec{\cC}_{X_\tau}$ has codimension $1$ in $\R^\cA$.
\end{theorem}

We begin with a simple technical lemma.

\begin{lemma}
  \label{lem:subs:balanced_space_is_orth_to_mu}
Let $\mu$ be the unique invariant measure on $X_\tau$, and let $\vec\mu=(\mu([a]))_{a\in\cA}\in\R^\cA$ be the corresponding letter-frequency column vector.
Then $\vec{V}_{X_\tau} + \vec{\cC}_{X_\tau} \subseteq \vec\mu^\perp$, {\em i.e.}, $\vec{v}\cdot \vec\mu=0$ for all $\vec{v}\in \vec{V}_{X_\tau}+\vec{\cC}_{X_\tau}$.
\end{lemma}

\begin{proof}
It suffices to show that $\vec{v}\cdot\vec\mu = \vec{c}\cdot \vec{\mu} = 0$ for every $\vec{v}\in\vec V_{X_\tau}$ and every $\vec{c}\in\vec{\cC}_{X_\tau}$.
Let $c\colon\cA^*\to\R$ be  a letter-coboundary in $X_\tau$ and put $\bar c(x)=c(x_0)$ for $x \in X_\tau$. Then, by \Cref{lem:cob_zero_integral}, we have
\[
  (c(a))_{a\in\cA}\cdot \vec\mu = \int \bar c\,\mathrm d\mu = 0.
\]
Let now $\vec{v}\in\vec V_{X_\tau}$.
Since $\vec{v}$ lies in a left generalized eigenspace of $M_\tau$ for some eigenvalue $\eta$ with $|\eta|<1$, necessarily $\eta\neq \lambda$, where $\lambda$ is the Perron-Frobenius eigenvalue of $M_\tau$.
By a standard linear-algebra fact, the left eigenspace for $\eta$ is orthogonal to the right eigenspace for $\lambda$; the latter is spanned by $\vec\mu$.
Hence $\vec{v}\cdot \vec\mu=0$.
\end{proof}

\medskip

\begin{proof}[Proof of \Cref{balanced=>spaces_decomposition}]
We aim to derive the theorem from the results in \Cref{sec:carac_finite_rank}. 
We recall that  $\btau$ is the constant directive sequence $(\tau : n \ge 0)$, which is primitive and recognizable (see \Cref{rem:recog_substitution}). 
In this setup, $\tau_{n,m}$ simply denotes $\tau^{m-n}$ for all $m > n \ge 0$, and  $X_{\btau}^{(n)} = X_\tau$ for every $n \ge 0$.
Let $\vec{\mu} = (\mu([a]))_{a \in \cA} \in \R^\cA$ be the column vector of letter frequencies in $X_\tau$, where $\mu$ is the unique invariant measure of $X_\tau$.
For each $a \in \cA$, define $\mathbf{1}_a \colon \cA \to \R$ to be the indicator function of $a$, \emph{i.e.}, $\mathbf{1}_a(b) = 1$ if $b = a$ and $\mathbf{1}_a(b) = 0$ otherwise.
\smallskip

We now prove the theorem.
Assume first that $X_\tau$ is letter-balanced.
This means that, for each $a \in \cA$, the function $\bar{\mathbf{1}}_a \colon X_\tau \to \Z$ defined by $\bar{\mathbf{1}}_a(x) = \mathbf{1}_a(x_0)$, is balanced on $(X_\tau, S)$.
Thus, by \Cref{prop:balance_char_Sadic}, there exist integers $n_a \ge 0$ and letter-coboundaries $c_a \colon \cA \to \R$ in $X_{\btau}^{(n_a)} = X_\tau$ such that
\[
    \lim_{k \to \infty}
    \sup_{v \in \cL(X_\tau)}
    \left| \mathbf{1}_a(\tau^{n_a+k}(v)) - \alpha_a h_{n_a+k}(v) + c_a(\tau^k(v)) \right| = 0,
\]
where $\alpha_a = \int \bar{\mathbf{1}}_a \, \mathrm{d}\mu$. 
Without loss of generality, we may assume that $n_a = n$ for all $a \in \cA$, since increasing $n_a$ does not affect the convergence.
Let $\vec{c}_a = (c_a(b))_{b\in\cA}$ and $\vec{\mathbf{1}}_a = (\mathbf{1}_a(b))_{b\in\cA}$ denote the row vectors associated to $c_a$ and $\mathbf{1}_a$, respectively.
Then the convergence above implies that
\[
    \lim_{k \to \infty} \vec{\mathbf{1}}_a M_\tau^{n+k} - \alpha_a \vec{h}_{n+k} + \vec{c}_a M_\tau^k = 0,
\]
where $\vec{h}_n:= (h_n(a))_a.$
Equivalently, $\vec{\mathbf{1}}_a M_\tau^n - \alpha_a \vec{h}_n + \vec{c}_a \in \vec{V}_{X_\tau}$, and therefore $\vec{\mathbf{1}}_a M_\tau^n - \alpha_a h_n \in \vec{V}_{X_\tau} + \vec{\cC}_{X_\tau}$.
By letting $\vec{\mathbf{1}}\in \R^\cA$ be the row vector having all entries  equal to $1$, the last inclusion rewrites as
\begin{equation*}
    (\vec{\mathbf{1}}_a - \alpha_a  \vec{\mathbf{1}}) M_\tau^n \in \vec{V}_{X_\tau} + \vec{\cC}_{X_\tau}
    \enspace \text{for all $a \in \cA$.}
\end{equation*}
Note that since $\alpha_a=\vec{\mathbf{1}}_a \cdot \vec{\mu}$,  then $\vec{\mathbf{1}}_a - \alpha_a  \vec{\mathbf{1}}$ is the projection of $\vec{\mathbf{1}}_a$ onto the orthogonal space $\vec{\mu}^\perp$ of $\vec\mu$.
As the vectors $\vec{\mathbf{1}}_a$ generate $\R^\cA$, $\vec{\mu}^\perp M_\tau^n$ is contained in $\vec{V}_{X_\tau} + \vec{\cC}_{X_\tau}$, which implies that 
\[  I^\infty \coloneqq \bigcap_{m \ge 0} \vec{\mu}^\perp M_\tau^m 
    \subseteq \vec{V}_{X_\tau} + \vec{\cC}_{X_\tau}.
\]
Also, the definition of $\vec{V}_\tau$ ensures that 
\[  K^\infty \coloneqq 
    \bigl\{v \in \vec{\mu}^\perp : 
    \exists m \ge 0 \text{ s.t. } v \, M_\tau^m = 0 \bigr\}
    \subseteq \vec{V}_{X_\tau} \subseteq 
    \vec{V}_{X_\tau} + \vec{\cC}_{X_\tau}.
\]
Now, since $\vec{\mu}$ is a right-eigenvector of $M_\tau$, its orthogonal complement $\vec{\mu}^\perp$ is invariant under the right action of $M_\tau$.
This allows us to use  the Fitting lemma, yielding
\[  \vec{\mu}^\perp = 
    I^\infty + K^\infty \subseteq 
    \vec{V}_{X_\tau} + \vec{\cC}_{X_\tau}.
\]
By \Cref{lem:subs:balanced_space_is_orth_to_mu}, $\vec{V}_{X_\tau} + \vec{\cC}_{X_\tau}$ is contained in $\vec{\mu}^\perp$.
We conclude that $\vec{V}_{X_\tau} + \vec{\cC}_{X_\tau}=\vec{\mu}^{\perp}$, and thus it has codimension $1$ in $\R^{\cA}$.
\smallskip

Conversely, assume that $\vec{V}_{X_\tau} + \vec{\cC}_{X_\tau}$ has codimension 1.
By \Cref{lem:subs:balanced_space_is_orth_to_mu}, this space must coincide with $\vec{\mu}^\perp$.
Fix any $a \in \cA$, and consider $\vec{\mathbf{1}}_a - (\vec{\mathbf{1}}_a \cdot \vec{\mu}) \vec{\mathbf{1}}$, which lies in 
$\vec{\mu}^\perp$ and hence in $\vec{V}_{X_\tau} + \vec{\cC}_{X_\tau}$.
Therefore, we can write $\vec{\mathbf{1}}_a - (\vec{\mathbf{1}}_a \cdot \vec{\mu}) \vec{\mathbf{1}} = \vec{c} + \vec{v}$, where  the notation  $\vec{c} $ stands for $ (c(a))_{a\in\cA}$ for  $c$  letter-coboundary 
in $X_\tau$ and where  $\vec{v} \in \vec{V}_{X_\tau}$. 
Let $x \in X_\tau$, let $N \ge 1$, and let $\vec{u}_N$ denote the abelianization of the word $x_{[0,N)}$, \emph{i.e.}, $\vec{u}_N= (|\vec{u}_N|_a)_{a\in \cA}$. 
Then the scalar product $(\vec{\mathbf{1}}_a - (\vec{\mathbf{1}}_a \cdot \vec{\mu}) \vec{\mathbf{1}}) \cdot \vec {u}_N$ equals the discrepancy
\[
    \sum_{i=0}^{N-1} \bar{\mathbf{1}}_a(S^i x) - N \int \bar{\mathbf{1}}_a \, \mathrm{d}\mu.
\]
Since $\vec{v} M_\tau^n \to 0$ exponentially fast and $c$ is uniformly bounded on $\cL(X_\tau)$, classical arguments (\emph{e.g.}, via Dumont--Thomas decomposition \cite{Dumont-Thomas}) imply that this discrepancy is uniformly bounded over all $x \in X_\tau$ and $N \ge 1$; see also the proof of \Cref{cor:EigCharac:FAR:Mult:cocoboundary} below.
Thus, $\bar{\mathbf{1}}_a$ is balanced on $(X_\tau, S)$, and the subshift $X_\tau$ is letter-balanced.
\end{proof}

\subsection{Extension graphs, invariant measures and eigenvalues}\label{subsec:ExtgraphsInvmeasuresEig}

Let $X \subseteq \cA^\Z$ be a minimal subshift and $\mu \in \cM(X,S)$ be one of its invariant measures.
We recall that the factor complexity function of $X$ is  given by $p_X(n) = |\cL_n(X)|$,  and 
that $[u] = \{x \in X : x_{[0,|u|)} = u\}$ for $u \in \cL(X)$.
We  have considered so far the set $I(X,S)$   generated by  measures  of  cylinders endowed  with  an additive group structure. 
We now consider the $\Q$-vector space generated by the measures of cylinders $\mu([u])$ for $u \in \cL(X)$. 
For a set $J \subseteq \R$, we let $\dim_\Q J$ be the dimension of the $\Q$-vector space spanned by $J$.
Then, we set 
\[   d(\mu,n) = \dim_\Q(\{ \mu([u]) : u \in \cL_n(X)\}).    \]

Let $\cK_n$ denote the disjoint union of all the connected components of the extension graphs $\Gamma_X(u)$ for $u \in \cL_n(X)$ (see Section \ref{subsec:extensiongraphs}).

A classical computation shows that, for every $u \in \cL(X)$, one has (with the notation of Section \ref{subsec:extensiongraphs})
\begin{equation} 
    \label{eq:rem:sumofmeasures}
    \mu([u]) = 
    \sum_{a \in \cP_X^L(u)} \mu([au]) = 
    \sum_{b \in \cP_X^R(u)} \mu([ub]).
\end{equation}
These linear relations induce an upper  bound for $d(\mu,n)$  described in the  next lemma, based on a concept  similar to that of a {\em flow matrix},  as introduced  in \cite{andrieu-notes} for Rauzy graphs, developed here in  the language of extension graphs. See also \cite{Werneck} for further results in this direction.

\begin{lemma} \label{lem:dimfrequencies}
    Let $X$ be a transitive subshift and let $\mu\in\cM(X,S)$. 
    Then, for each $n\geq 1$, one has
    \begin{equation}\label{eq:dimfrequencies}
        d(\mu,n) \leq 
        p_X(n) - |\cK_{n-1}| + 1.
    \end{equation}
\end{lemma}
\begin{proof} Let $n\geq 1$ be fixed.
To each $w = w_0 w_1 \dots w_{n-1} \in \cL(X)$ of length $n$, we associate $p_w = w_0\ldots w_{k-n}$ and $s_w = w_1\ldots w_{n-1}$, both of length $n-1$.
Let $K_w$ and $\Bar{K}_w$ be the unique connected components in $\Gamma_X(s_w)$ and $\Gamma_X(p_w)$, respectively, such that $w_0$ appears on the left side of $K_w$ and $w_{n-1}$ appears on the right side of $\Bar{K}_w$.
Consider the linear transformation 

\[ M\colon \Q^{\cK_{n-1}}\to \Q^{\cL_n(X)},  \quad \  (q_K)_{K\in\cK_{n-1}} \mapsto
    (q_{K_w}-q_{\Bar{K}_w})_{w\in\cL_n(X)}.  \]
 
Consider also the linear transformation 
\[ L\colon \Q^{\cL_n(X)}\to \R, \quad  (q_w)_{w\in\cL_n(X)} \mapsto
    \sum_{w\in\cL_n(X)} q_w \, \mu([w]).   \]
Note that $d(\mu,n)=\rank(L)$. 

Let  $e_K$ be the canonical vector of $\Q^{\cK_{n-1}}$ having $1$ at position $K$ and $0$ elsewhere. Note also that if $K$ is a connected component of $\Gamma_X(u)$, with $u\in\cL_{n-1}(X)$, then, for any $w\in\cL_n(X)$, the $w$-th entry of $M(e_K)$ equals $1$ if $K=K_w$, $-1$ if $K=\Bar{K}_w$, and $0$ otherwise. 
Thus,
\[  L(M(e_K)) =
    \sum_{a \in K \cap \cA_L} \mu([au]) - 
    \sum_{b \in K \cap \cA_R} \mu([ub]) = 0,  \]
where the last step is due to \Cref{eq:rem:sumofmeasures}. 
This proves that $M(\Q^{\cK_{n-1}})\subseteq \ker L$. 
Hence, by the Rank-nullity Theorem,
\begin{align*}
d(\mu,n) &= 
p_X(n) - \dim_{\Q}\ker{L} \\ &\leq
p_X(n) - \dim_{\Q}M(\Q^{\cK_{n-1}}) = 
p_X(n) - |\cK_{n-1}| + \dim_{\Q}\ker{M}.
\end{align*}
We are going to prove that $\dim_{\Q}(\ker{M}) = 1$, from which the lemma follows. 

Observe that any $w \in \cL_n$ is represented exactly once in $L_{K_w}$, and exactly once in $R_{\Bar{K}_w}$. 
Hence, if $v = (1,\dots,1) \in \Q^{\cK_{n-1}}$, then by definition we have that $v \in \ker{M}$.
We will show that any $q \in \ker{M}$ is a scalar multiple of $v$. 
Let now $q \in \ker M$ and fix a transitive point $x\in X$. 
For each $\ell\in\Z$, let $K_\ell \colon = K_{x_\ell x_{\ell+1}\ldots x_{\ell+n-1}}$ and $\Bar{K}_\ell\colon =\Bar{K}_{x_\ell x_{\ell+1}\ldots x_{\ell+n-1}}$. 
Since $M(q)=0$, 
\[  q_{K_\ell} = q_{\Bar{K}_\ell} 
    \enspace \text{for all $\ell \in \Z$.}  \]
On the other hand, since $x_{\ell-1}$ and $x_{\ell+n-1}$ are located respectively on the left and right side of the same connected component of $\Gamma_X(x_{\ell}\ldots x_{\ell+n-2})$, $K_{\ell-1} = \Bar{K}_{\ell}$. 
Thus,
\[  q_{K_\ell} = q_{K_{\ell-1}} 
    \enspace \text{for all $\ell \in \Z$.}  \]
Since $x$  is a transitive point, every word of length $n$ appears in $x$, so $q_K = q_{K_0}$ for all $K \in \cK_{n-1}$. 
This proves that $q$ is a scalar multiple of $v$ and concludes the proof.
\end{proof}


We now use \Cref{lem:dimfrequencies} to provide an upper bound for the  dimension of the  group of  eigenvalues $E(X,S)$ involving the  cardinality of the  disjoint union $\cK_n$  of all the connected components of the  extension graphs $\Gamma_X(u)$ for $u \in \cL_n(X)$.
We recall that a directive sequence $\btau = (\tau_n : n \geq 0)$ is unimodular if each incidence matrix $M_{\tau_n}$ has determinant $\pm 1$. 

\begin{proposition}\label{prop:dimeigenvalues}
    Let $X \subseteq \cA^\Z$ be a minimal subshift and let $\mu \in \cM(X,S)$. 
    Then,
    $$\dim_\Q E(X,S) \leq 
    \liminf_{n\to\infty}
    (p_X(n) - |\cK_{n-1}| )+ 1.$$
    In particular, if $X$ is generated by a unimodular primitive directive sequence, then
    \begin{equation}\label{eq:dimEunimodular}
        \dim_\Q E(X,S) \leq |\cA|-r+1,
    \end{equation} 
    where $r$ is the number of connected components of the extension graph of the empty word.
\end{proposition}
\begin{proof}
    Since $E(X,S)$ is a subgroup of the image subgroup $I(X,S)$ by \cite{itzaortiz,CORTEZ_DURAND_PETITE_2016,ghh18}, we have $\dim_\Q E(X,S)\leq \dim_\Q I(X,S)$.
    For $n \ge 0$, let $I_n(X,S)$ be the $\Z$-span of
    \[  \{ \mu([u]) : u \in \cL_n(X) \}.  \]
    Then, $I_n(X,S) \subseteq I_{n+1}(X,S)$ and $I(X,S)$ is the union of all the $I_n(X,S)$ for $n \ge 0$ (see \Cref{rem:measures}).
    Since $d(\mu,n) = \dim_\Q(I_n(X,S))$, it follows that 
    \[  \dim_\Q I(X,S) = \lim_{n\to\infty} d(\mu,n). \]
    Combining this with the bound given by \Cref{lem:dimfrequencies} yields
    $$\dim_\Q(I(X,S)) \leq \liminf_{n\to\infty} (p_X(n)-|\cK_{n-1}|)+1.$$
    This proves the first statement. 
    If $X$ is generated by a unimodular primitive directive sequence, then
    \[  I(X,S) = 
        \bigcap_{\mu\in\cM(X,S)} 
        \Big\{ \sum_{a\in\cA} w_a \, \mu([a]) :
        (w_a)_{a\in\cA} \in \Z^{\cA} \Big\}, \]
    see \cite[Corollary 4.3]{unimodular} and the subsequent paragraphs.
    Thus, $\dim_\Q(I(X,S))\leq d(\mu,1)$. 
    Equation \eqref{eq:dimfrequencies} for $n = 1$ gives
    $$ d(\mu,1)\leq |\cA|-r+1. $$
    This proves the second statement and finishes the proof.
\end{proof}

\begin{remark}
    According to \cite[Theorem 9]{BDM10}, if $(X,T)$ is a minimal Cantor system of finite rank $d$ with $\ell$ ergodic measures, then $\dim_\Q E(X,S) \leq d - \ell + 1$. 
    Thus, both the number of ergodic measures and the rational dimension of their values bound the size of the group of eigenvalues. Note that in \Cref{eq:dimEunimodular}, while the left-hand side of the inequality is preserved under conjugacy, the right-hand side is not, so for a given minimal subshift, one can choose, among all the conjugate systems which are generated by a unimodular primitive directive sequence, the lowest possible bound. 
\end{remark}

\subsection{Extension graphs, factor complexity and eigenvalues}\label{subsec:eigen}

A further important relation between the dimension of the rational space of measures of letter-cylinders of a subshift and its  factor complexity,  due to Tijdeman \cite{tijdeman},  is described in  \Cref{thm:tijdeman} below.
We show in this section that it follows directly from Lemma~\ref{lem:dimfrequencies}. 
We also recover as a corollary of \Cref{lem:dimfrequencies} a result by Andrieu and Cassaigne \cite{andrieu-notes}, which links Tijdeman's theorem to the dendricity of a transitive subshift (see \Cref{def:dendric}). Tijdeman's original proof is based on studying linear relations in the Rauzy graph (with a different terminology though), and later Andrieu and Cassaigne refined the argument to obtain the conclusion on the dendricity. We follow here the same approaches, but use the language of extension graphs.

\begin{theorem}[\cite{tijdeman}] \label{thm:tijdeman}
Let $X \subseteq \cA^\Z$ be a transitive subshift, and let $\mu$ be one of its invariant measures. 
Denote by $t$ the dimension of the $\Q$-vector space generated by the set $\{\mu([a]) : a \in \cA\}$.
Then, one has 
\begin{equation} \label{eq:thm:tijdeman:1}
    p_X(n)\geq (n-1)(t-1)+|\cA|, \; \mbox{for all } n \geq 1.
\end{equation}
\end{theorem}
\begin{proof}
For $j \geq 1$, we have $t=d(\mu,1)\leq d(\mu,j)$. 
Also, $|\cK_{j-1}| \geq p_X(j-1)$, as each extension graph corresponding to a word of length $j-1$ contains at least one connected component.
Combining these two bounds with the inequality given by \Cref{lem:dimfrequencies}, we obtain
\begin{equation} \label{eq:thm:tijdeman:2}
    t \leq  p_X(j) - |\cK_{j-1}(X)| + 1 \leq 
    p_X(j) - p_X(j-1) + 1, \; \text{for any $j \geq 1$.}
\end{equation}
Hence, $(n-1)t \leq p_X(n) - p_X(1) + (n-1)$ for all $n \geq 1$, from which \eqref{eq:thm:tijdeman:1} follows. 
\end{proof}

The following classical result relating the cardinalities of extensions will be used in the proof of \Cref{thm:tijdeman2}. 
 
\begin{lemma}[\cite{cassaigne}] 
    \label{lemma:extgraphs&multiplicity}
    Let $X \subseteq \cA^\Z$ be a subshift.
    For $w \in \cL(X)$, define
    \[  m_X(w) = |\Gamma_X(w)| - |L_X(w)| - |R_X(w)| + 1. \]
    Then, 
    \[  \sum_{w \in \cL_n(X)} m(w) = 
        p_X(n+2) -2 p_X(n+1) + p_X(n)
        \enspace \text{for any $n \geq 0$-}
        \]
\end{lemma}

We now can state the following sufficient condition for dendricity.
\begin{theorem}[\cite{andrieu-notes}] \label{thm:tijdeman2}
Let $X \subseteq \cA^\Z$ be a transitive subshift and suppose that one of its invariant measures $\mu$ has rationally independent coordinates, \emph{i.e.}, the dimension of the $\Q$-vector space generated by the set $\{\mu([a]) : a \in \cA\}$ is $|\cA|$.
If 
\begin{equation} \label{eq:thm:tijdeman:3}
    p_X(n)= (|\cA|-1)n+1 \enspace \text{for all }n \geq 1,
\end{equation}
then $X$ is dendric.
\end{theorem}

\begin{proof}
Let us assume that \eqref{eq:thm:tijdeman:3} holds and the dimension of the $\Q$-vector space generated 
by the set $\mu([a]) : a \in \cA\}$ is $|\cA|$.
Fix $n\geq 1$. 
By \Cref{lem:dimfrequencies}, $|\cA|\leq p_X(n)-|\cK_{n-1}|+1.$ 
Hence 
$$|\cA|\leq p_X(n) -p_X(n-1)  + p_X(n-1)-|\cK_{n-1}|+1 = |\cA|-1+ p_X(n-1) -|\cK_{n-1}|+1.$$
This gives  $|\cK_{n-1}|\leq p_X(n-1) $, which means that all extension graphs of words of length $n-1$ are connected. Since this is true for all $n\geq 1$, we deduce that all extension graphs in $X$ are connected.

It remains to show that each extension graph is acyclic.
Observe that, by \Cref{lemma:extgraphs&multiplicity} and \eqref{eq:thm:tijdeman:3},
$$ \sum_{w\in\cL_n(X)} m_X(w) = p_X(n+2) - 2p_X(n+1) +p_X(n) = 0
    \; \text{ for all $n \geq 0$.} $$
Moreover, since the extension graph of every $w \in \cL(X)$ is connected, each term $m_X(w)$ in the previous summation is nonnegative; thus, $m_X(w) = 0$ for all $u \in \cL(X)$, as their sum is zero.
Using the classical result that a tree is a connected graph with exactly one  edge less than its number of vertices, we deduce that $\Gamma_X(u)$ is a tree.
Therefore, each extension graph in $X$ is a tree.
\end{proof}

\begin{remark}
    The converse of the previous theorem is false: codings of regular interval exchange transformations are dendric and may have rationally dependent letter frequencies.
\end{remark}

\subsection{Sufficient conditions for having trivial coboundaries}
\label{sec:studyofcoboundaries:conditions_for_trivial_only}
We now provide sufficient conditions for a minimal shift to have only trivial coboundaries. 
This is of particular interest in view of Theorems \ref{theo:EigCharac:FAR:Mult} and \ref{theo:EigCharac:decisive:positive}.

\begin{proposition}\label{prop:dim}
    Let $X \subseteq \cA^\Z$ be a minimal subshift such that there exists $\mu\in\cM(X,S)$ for which the vector $(\mu([a]))_{a\in\cA}$ has $\Q$-independent coordinates. 
    Then, any coboundary of $\cL(X)$ is trivial.
\end{proposition}
\begin{proof}
It follows from \Cref{lem:dimfrequencies} that $|\cK_0| \leq 1$, that is, $\Gamma(\varepsilon)$ is connected.
Therefore, by \Cref{cor:connectedtrivial}, every coboundary in $\cL(X)$ is trivial.
\end{proof}

We point out that in the setting of $S$-adic subshifts, sufficient conditions have been identified to get $\Q$-independent letter frequencies \cite{BST}. 
We mention one of them here.
The directive sequence $\boldsymbol{\sigma}$ is said to be \emph{algebraically irreducible} if, for each $k \in \mathbb{N}$, the characteristic polynomial of the incidence matrix $M_{k,\ell}$ of $\tau_{k,\ell}$ is irreducible for all sufficiently large~$\ell$.
In \cite[Lemma 4.2]{BST}, the authors prove that for an algebraically irreducible directive sequence $\boldsymbol{\sigma}$ such that $\cL(X_{\boldsymbol{\sigma}})$ is balanced on letters, the vector with letter frequencies $(\mu([a])_{a\in \cA}$ has $\Q$-independent coordinates, where $\mu$ is the unique invariant measure of $X_{\boldsymbol{\sigma}}$.  We also remark that there is an important family of $S$-adic subshifts whose letter frequencies are always $\Q$-independent, namely the family of subshifts obtained by {\em Arnoux-Rauzy} sequences (see \Cref{ex:modifiedAR}).

\begin{remark}\label{rem:irreducible_trivial}
    
Note that if $\btau$ is an algebraically irreducible directive sequence, then for any $k\in \N$ and all sufficiently large $\ell \gt k$, neither  $1$, nor $0$ is  an eigenvalue of the matrix $M_{k,l}$.
Now, in \Cref{finitary=>cobords_are_eigenvectors} below, we show that, for any letter-coboundary $c_n$ in level $X_{\btau}^{(n)}$, the vector $(c_n \circ \tau_{n,m}(a) : a \in \cA_m)$ is  either an eigenvector of $M_{n,m}$ with eigenvalue $1$, or  belongs to  the  kernel of  $M_{n,m}$  for  infinitely many  $m \ge n$.
Therefore, in the algebraically irreducible case, $c_n \circ \tau_{n,m}(a) \equiv 0$ for all large enough $m \gt n$.
This observation is used in \Cref{cor:EigCharac:FAR:Mult:cocoboundary} to simplify our criterion in \Cref{theo:EigCharac:FAR:Mult}; see $(P_2) \Rightarrow (P_3)$. Observe that  \Cref{finitary=>cobords_are_eigenvectors}
  is  an $S$-adic  counterpart of  \cite[Lemma 2.8]{mercat} which is stated for primitive substitutions.
\end{remark}

\begin{lemma}
\label{finitary=>cobords_are_eigenvectors}
Let $\btau = (\tau_n \colon \cA_{n+1}^* \to \cA_n^* : n \ge 0)$ be a primitive directive sequence for which $(|\cA_n| : n \ge 0)$ is bounded.
Fix $n \ge 0$ and a letter-coboundary $c_n \colon \cA_n^* \to \R$ in $X_{\btau}^{(n)}$.
Then, after possibly relabeling the alphabets $\cA_m$, there exists a strictly increasing sequence $(m_k)_{k \ge 0}$, with $m_0 > n$, such that $\cA_{m_k} = \cA_{m_\ell}$ and $c_n \circ \tau_{n,m_k} = c_n \circ \tau_{n,m_\ell}$ for all $\ell > k \ge 0$.

In particular, for all $k \ge 1$, the row vector $\vec{c}_{m_0} \coloneqq \big(c_n(\tau_{n,m_0}(a))\big)_{a \in \cA_n}$ is either  zero, or an eigenvector of $M_{m_0,m_k}$ associated to eigenvalue~1, or else  lies in the kernel of $M_{n,m_0}$.
\end{lemma}
\begin{proof}
We use \Cref{lem:cobord&rho} to find a map $\rho_n \colon \cA_n \to \R$ such that $c_n(a) = \rho_n(b) - \rho_n(a)$ for all $ab \in \cL(X_{\btau}^{(n)})$ of length 2.
Since $(|\cA_n| : n \ge 0)$ is bounded, we may assume (after possibly relabeling the alphabets) that there exist an increasing sequence $(m_k : k \ge 0)$ such that $\cA \coloneqq \cA_{m_k}$ for every $k \ge 0$.
Recall that $\First(u)$ denotes the first letter of a non-empty word $u$.
Since the sets $\{\First \circ \tau_{n,m_k} \colon \cA \to \cA_n : k \ge 0\}$ and $\{\cL(X_{\btau}^{(m_k)}) \cap \cA^2 : k \ge 0\}$ are finite, by taking a subsequence if necessary, we may assume that there exist $p \colon \cA \to \cA_n$ and $\cL_2 \subseteq \cA^2$ such that
\begin{enumerate}[label=(\roman*)]
    \item the set of length-2 words in $\cL(X_{\btau}^{(n)})$ equals $\cL_2$ for all $k \ge 0$, and
    \item $p(a) = \First(\tau_{n,m_k}(a))$ for all $a \in \cA$ and $k \ge 0$.
\end{enumerate}
Let us show that $c_n \circ \tau_{n,m_k} = c_n \circ \tau_{n,m_\ell}$ for all $k,\ell \ge 0$.

Consider an arbitrary $a \in \cA$.
We can take $b \in \cA$ with $ab \in \cL(X_{\btau}^{(m_k)})$.
Using Item (ii), we can compute:
\[  c_n(\tau_{n,m_k}(a)) = 
    \rho_n(\First(\tau_{n,m_k}(b))) - 
    \rho_n(\First(\tau_{n,m_k}(a))) = 
    \rho_n(p(b)) -
    \rho_n(p(a)).
\]
Now, by Item (i), $ab$ also lies in $\cL(X_{\btau}^{m_\ell})$.
Therefore, the previous argument applies with $m_\ell$, yielding
\[  c_n(\tau_{n,m_\ell}(a)) = 
    \rho_n(p(b)) - \rho_n(p(a)) = 
    c_n(\tau_{n,m_k}(a)).
\]
Taking $\ell = 0$, we get that the vector $\vec{c}_{m_0}$ satisfies $\vec{c}_{m_0} = \vec{c}_{m_0} \, M_{m_0,m_k}$ for all $k \ge 1$.
So, either $\vec{c}_{m_0}$  is zero, or it lies in the kernel of $M_{n,m_0}$,  or else  $\vec{c}_{m_0}$ is an eigenvector of $M_{m_0,m_k}$ with eigenvalue 1.
\end{proof}

We now turn to sufficient conditions ensuring triviality of letter-coboundaries based on return words.
Recall that the {\em Parikh vector} of a word $u \in \cA^*$ is the element $\Vec{w} \in \Z^\cA$ whose $a$-th coordinate equals the number of times that $a$ occurs in $w$.
We denote by $\cR_X(u)$ the set of return words to a word $u \in \cA^* \setminus \{\varepsilon\}$ in a minimal subshift $X$, and write $\Vec{\cR}_X(u) = \{\Vec{w} : w \in \cR(u)\}$.

\begin{proposition} 
    \label{prop:generatingabelian}
Let $X \subseteq \cA^\Z$ be a minimal subshift.
If there exists a word $u \in \cL(X)$ such that the set of integer linear combinations of elements of $\Vec{\cR}_X(u)$ is a finite index subgroup of $\Z^\cA$, then any letter-coboundary of $\cL(X)$ is trivial.
\end{proposition}
\begin{proof}
Let $c \colon \cA^* \to \R$ be an arbitrary letter-coboundary in $X$.
We denote by $V$ the $\Z$-linear span of the abelianized return words $\Vec{\cR}_X(u)$, which has finite index in $\Z^\cA$ by hypothesis.
Since $c$ is a monoid morphism into an abelian group, there exists a factorization of $c$ through the abelianization of $\cA^*$, that is, there is a $\Z$-linear map $\tilde{c} \colon \Z^\cA \to \R$ such that $c(w) = \tilde{c}(\Vec{w})$ for all $u \in \cA^*$.
In particular, $\tilde{c}(\Vec{w}) = c(w) = 0$ for every $w$ in $\Vec{\cR}_X(u)$ (and thus in $V$).

Let now $z \in \Z^\cA$ be any vector.
Since $V$ has finite index in $\Z^\cA$, the element $z + V$ has finite order in the additive abelian group $\Z^\cA / V$, so there exists a positive integer $n \geq 1$ such that $n z \in V$.
We obtain $n \, \tilde{c}(z) = \tilde{c}(nz) = 0$, and thus $\tilde{c}(z) = 0$.
Therefore, $c(w) = \tilde{c}(\Vec{w}) = 0$ for every $w \in \cA^*$, so $c$ is trivial.
\end{proof}

In many cases of interest, the hypothesis of \Cref{prop:generatingabelian} is guaranteed by a stronger condition, that we now introduce.
The word monoid $\cA^*$ has a canonical embedding into the free group $\F_{\cA}$ with free generators $\cA$, so we consider $\cA^*$ to be a subset of $\F_\cA$ 
Then, if $X \subseteq \cA^\Z$ is a subshift and $u \in \cL(X)$, the set $\cR_X(u)$ of return words to $u$ in $X$ generates a subgroup of $\F_\cA$.
For certain classes of subshifts of interest, such as dendric subshifts 
and certain subshifts of geometric origin (such as the natural coding of an interval exchange transformation and hypercubic billiard words), the set of return words $\cR_X(u)$ generates the whole free group $\F_\cA$, for every $u \in \cL(X)$.
By mapping onto $\Z^\cA$ through the abelianization map, one gets that $\Vec{\cR}_X(u)$ generates $\Z^\cA$, and thus that the hypothesis of \Cref{prop:generatingabelian} holds.
This argument proves something slightly stronger:

\begin{corollary}
\label{cor:generatingfreegroup}
Let $X \subseteq \cA^\Z$ be a minimal subshift. 
If there exists $u \in \cL(X)$ such that the set of return words to $u$ in $X$ generates a finite index subgroup of $\F_\cA$, then the   letter-coboundaries of  $X$ trivial.
\end{corollary}

\Cref{cor:generatingfreegroup} together with \Cref{thm:manifold} lead us to the conclusion that in any minimal subshift $X\subseteq \cA^{\Z}$ for which the extension graph of the empty word has more than one connected component, none of the subgroups of $\F_{\cA}$ generated by return words in $\cL(X)$ is a finite index subgroup. Dendric subshifts correspond exactly to those minimal subshifts in which the set of return words to any word is a basis of the free group over the alphabet \cite{dendricity&return_words}. For these systems, one can construct a proper unimodular $S$-adic representation using return words (see \cite[Proposition 3.8]{unimodular}), so that the upper bound that we obtain in Equation $\eqref{eq:dimEunimodular}$ reduces to $\dim_{\Q}E(X,S)\leq |\cA|$.

A converse form of \Cref{cor:generatingfreegroup} holds for subshifts with a \emph{suffix-connected} language, as it is proved in \cite[Corollary 1.2]{goulet}, where the author shows that for such a subshift, the subgroups generated by return words are the whole free group if and only if the extension graph of the empty word is connected.


The following example shows that not even a weak converse of \Cref{cor:generatingfreegroup} is possible: there are aperiodic minimal subshifts for which all the   letter-coboundaries are trivial, and  such that none of the sets $\Vec{\cR}_X(u)$, for $u \in \cL(X)$, generates a finite-index subspace of $\Z^\cA$.

\begin{example}\label{ex:fibo}
    Let $\cA = \{0,1\}$ and $X \subseteq \cA^\Z$ be any minimal subshift. 
    We define $\cB = \{0,1,2,3,4,5\}$ and the substitution $\sigma \colon \cA^* \to \cB^*$ by
    \[  \sigma \colon \begin{cases}
            0 &\mapsto 01 \, 02 \, 03 \, 04 \, 055 \\
            1 &\mapsto 01 \, 02 \, 03 \, 055 \, 04
        \end{cases}     \]
    Let $Y$ the shift-orbit of $\sigma(X)$, which is a minimal aperiodic subshift.
    From $\sigma(0)$ and $\sigma(1)$ we see that $\cL(Y)$ contains $0a$ and $a0$ for every $a \in \{1,2,3,4,5\}$, and also $55$.
    Hence the extension graph of the empty word in $Y$ is connected, so  all the letter-coboundaries of $Y$ are trivial by \Cref{thm:manifold}.
    
    We claim that for every $u \in \cL(Y) \setminus \{\varepsilon\}$, the $\Z$-linear span $V_Y(u)$ of $\Vec{\cR}_Y(u)$ has infinite index in $\Z^{\cB}$.
    First, we note that if $u$ starts with the letter $a$, then any return word to $u$ is a concatenation of return words to $a$; so, $V_Y(u)$ is contained in $V_Y(a)$.
    Therefore, it suffices to show that $V_Y(a)$ has infinite index in $\Z^\cB$ for each $a \in \cB$.

    Since both $\sigma(0)$ and $\sigma(1)$ begin with $0$, every return word to $0$ occurs inside $\sigma(0)$ or $\sigma(1)$.
    This allows us to compute
    \[  \cR_Y(0) = \{01, 02, 03, 04, 055\}.  \]
    This set has only five elements, which is less than the dimension of $\Z^\cB$, so $V_Y(0)$ has infinite index in $\Z^\cB$.
    Now fix $a\in\{1,2,3,4,5\}$. In each of $\sigma(0)$ and $\sigma(1)$, the letter $a$ appears either once (if $a\in\{1,2,3,4\}$) or twice (if $a=5$). 
    Consequently, any return word to $a$ is contained in $\sigma(bb')$ for some $bb' \in \cL(X)$ with $b,b' \in \cA$. 
    There are at most four possibilities for $b,b' \in \cA$, so these contribute at most $4$ distinct return words.
    Additionally, there are no return words to $a$ occurring entirely inside $\sigma(0)$ or $\sigma(1)$ when $a\in\{1,2,3,4\}$, and exactly one such word when $a=5$, namely $5$. 
    Therefore $|\cR_Y(a)| \le 5$ for all $a$, and hence $\rank V_Y(a) \le 5 < 6$. 
    It follows that $V_Y(a)$ has infinite index in $\Z^{\cB}$.
\end{example}

The following theorem states that the set of abelianized return words to any word in the language of the subshift generated by a proper, unimodular and primitive directive sequence generates the group $\Z^\cA$, where $\cA$ is the alphabet of $X$.

\begin{theorem} 
    \label{thm:proper+unimod=>RW_gen_Zd}
Let $X$ be a subshift generated by a proper, unimodular and primitive directive sequence $\btau = (\tau_n \colon \cA_{n+1}^* \to \cA_n^* : n \geq 0)$.
Then, for any $n\geq 0$ and any $u$ in the language of $X_{\btau}^{(n)}$, the set of integer linear combinations of abelianized return words to $u$ in $X$ equals $\Z^{\cA_n}$.
In particular, each $X_{\btau}^{(n)}$ admits no nontrivial letter-coboundary.
\end{theorem}

\begin{proof}
The directive sequence $\btau_n = (\tau_m : m > n)$ is proper, unimodular, primitive and generates $X_{\btau}^{(n)}$.
So, we may assume without loss of generality that $n = 0$.
Also, since $\btau$ is unimodular, we can assume that all the alphabets $\cA_n$ are equal, say to $\cA$.

Let $u \in \cL(X)$ be arbitrary.
The minimality of $X$ implies that there exists $L > 0$ such that $u$ occurs in any $w \in \cL(X)$ of length at least $L$.
Moreover, since $\btau$ is proper, we can find $n \geq 0$ and $w \in \cA_0^*$ such that 
\begin{enumerate}
    \item[(i)]  $w$ is a prefix of $\tau_{0,n}(a)$ for all $a \in \cA$,
    \item[(ii)] $w$ has length at least $L$, and thus contains an occurrence of $u$.
\end{enumerate}
Let $V$ be the $\Z$-linear span of $\Vec{\cR}_X(u)$.
As the matrix $M_{0,n}$ of $\tau_{0,n}$ is unimodular, $M_{[0,n}(\Vec{\cA}_n)$ generates $\Z^\cA$, so it is enough to prove that $V$ contains $M_{0,n}(\Vec{\cA}_n)$.

Let $a \in \cA$.
We take $b \in \cA$ such that $ab \in \cL(X_{\btau}^{(n)})$.
By Item (i), we can write $\tau_{0,n}(a) = w s_a$ and $\tau_{0,n}(b) = w s_b$ for some $s_a,s_b \in \cA^*$. 
Moreover, by Item (ii), $w = t u t'$ for some $t,t' \in \cA^*$.
Hence,
    $$ tut' s_a tut' s_b = w s_a w s_b = \tau_{0,n}(ab) \in \cL(X). $$
In particular, the word $p = u t' s_a t$ is a concatenation of return words to $u$.
Since $p$ has the same Parikh vector as $\tau_{0,n}(a)$, we conclude that $M_{0,n}(\Vec{a}) \in V$, completing the proof.

Finally, by result just proved, \Cref{prop:generatingabelian} implies that  the letter-coboundaries in $X_{\btau}$ are all trivial.
\end{proof}

We note that \Cref{thm:proper+unimod=>RW_gen_Zd} gives a necessary condition for a subshift to be generated by a proper, unimodular, and primitive directive sequence:  the letter-coboundaries must   all be trivial.
Now, if this condition is not met, then the subshift still might be {\em conjugate} to another subshift admitting such a representation. 
Since the extension graph of the empty word is not invariant under conjugacy, \Cref{thm:proper+unimod=>RW_gen_Zd} together with \Cref{cor:connectedtrivial} imply that the property of admitting a proper unimodular representation is not preserved neither.
See also Remark \ref{rem:notpreservedconnecetdgraph}.

\subsection{Back to eigenvalues}

We now gather previous results allowing to state simple sufficient and necessary conditions for $\alpha$ to be an additive eigenvalue.

\begin{corollary}
\label{cor:EigCharac:FAR:Mult:cocoboundary}
Let $X$ be a subshift generated by a primitive and recognizable directive sequence $\btau = (\tau_n \colon \cA_{n+1}^* \to \cA_n^*)_{n\geq0}$.
Suppose that $\btau$ is decisive or $(|\cA_n| : n \geq 0)$ is bounded.
Furthermore, suppose that, for every $n \geq 0$, one of the following conditions holds:
\begin{enumerate} [label=$(C_\arabic*)$]
    \item the extension graph of the empty word in $X_{\btau}^{(n)}$ is connected;
    \item  the vector $(\mu_n([a]) : a\in\cA_n)$ has $\Q$-independent coordinates for some  invariant measure $\mu_n$ of $X_{\btau}^{(n)}$;
    \item there exists a word $u\in\cL(X)$ such that the set of integer linear combinations of elements of $\vec{\cR}_X(u)$ is a finite index subgroup of $\Z^{\cA}$;
    \item  the directive sequence $\tau_n$ is left- or right-proper;
    \item  the sequence $(|\cA_n| : n \ge 0)$ is bounded and $1$ is not an eigenvalue of the matrix of $\tau_{n,m}$ for all large enough $m \gt n$.
\end{enumerate}
Then, for any $\alpha \in \R$,  the following implications hold: $(P_1) \Rightarrow (P_2) \Rightarrow (P_3)$ and $(P_2) \Rightarrow (P_4)$, where 
\begin{enumerate} [label=$(P_\arabic*)$]
    \item
    $\sum_{n\geq0} \max\bigl\{
    \|\alpha h_n(u) \| :
    a \in \cA_{n+1}, u \in \prefixes(\tau_n(a))
    \bigr\}  < \infty;$

    \item $\alpha$ is an additive eigenvalue of $X$;

    \item 
    $\max\big\{
    \|\alpha h_n(u) \| :
    a \in \cA_{n+1}, u \in \prefixes(\tau_n(a))
    \big\}$ converges to $0$ as $n \to \infty$;

    \item there exists $n \geq 0$ and $(w_a)_{a\in\cA_n} \in \Z^{\cA_n}$ such that $\alpha = w_a\, \mu(B_n(a))$ for all $\mu \in \cM(X,S)$.
\end{enumerate}
\end{corollary}

We remark that Condition $(P_3)$ implies (via application of a triangular inequality) the classical condition for $\alpha$ to be an  eigenvalue (see also \Cref{eq:EigCharac:necessary:proper}):
\begin{equation*}
    \lim_{n\to\infty} \max\big\{\,\| \alpha h_n(a) \| : a \in \cA_n \big\} = 0.
\end{equation*}

\begin{proof}[Proof of \Cref{cor:EigCharac:FAR:Mult:cocoboundary}]
We first prove that $(P_1) \Rightarrow (P_2)$.  
For each $n \ge 0$, define
\[
    \varepsilon_n = \max\bigl\{ \|\alpha h_n(u)\| : a \in \cA_{n+1},\, u \in \prefixes(\tau_n(a)) \bigr\}.
\]
Assumption $(P_1)$ is equivalent to the series $\sum_{n \ge 0} \varepsilon_n$ being finite.
We aim to verify the condition in \Cref{theo:EigCharac:Mult} with $\rho_n(a) \equiv 0$, 
that is,
\begin{equation}
    \label{eq:0:cor:EigCharac:FAR:Mult:cocoboundary}
    \lim_{n \to \infty} \sup\bigl\{ \|\alpha h_n(v)\| : v \in \cL(X_{\btau}^{(n)}) \bigr\} = 0.
\end{equation}
Let us first consider a word $v$ that is a prefix of $\tau_{n,m}(a)$ for some $m > n$ and $a \in \cA_m$.  
We may decompose such a prefix as
\[
    v = \tau_{n,m-1}(v_{m-1})\, \tau_{n,m-2}(v_{m-2}) \cdots \tau_n(v_{n+1})\, v_n,
\]
where for each $n \le i < m$, $v_i a_i \in \prefixes(\tau_i(a_{i+1}))$ and $a_i \in \cA_i$ (with $a_m = a$).  
Using this decomposition, we estimate
\[
    \|\alpha h_n(u)\|
    \le \sum_{n \le i < m} \|\alpha h_i(v_i)\|
    \le \sum_{i \ge n} \varepsilon_i.
\]
Now let $v$ be any word occurring in $\tau_{n,m}(a)$ (not necessarily as a prefix).  
Then there exists $u \in \prefixes(\tau_{n,m}(a))$ such that $uv \in \prefixes(\tau_{n,m}(a))$.
Applying the previous estimate to both $u$ and $uv$, we obtain
\[
    \|\alpha h_n(v)\|
    \le \|\alpha h_n(uv)\| + \|\alpha h_n(u)\|
    \le 2 \sum_{i \ge n} \varepsilon_i.
\]

Since the series $\sum_{i \ge 0} \varepsilon_i$ converges, its tails $\sum_{i \ge n} \varepsilon_i$ tend to zero as $n \to \infty$.  
This proves \eqref{eq:0:cor:EigCharac:FAR:Mult:cocoboundary}, and thus, by \Cref{theo:EigCharac:Mult}, that $\alpha$ is an eigenvalue of $(X,S)$.
\smallskip

We now suppose that $\alpha$ is an eigenvalue of $X$ and prove $(P_3)$.
If $(C_1)$, $(C_2)$, $(C_3)$ or $(C_4)$ holds, then applying \Cref{theo:EigCharac:decisive:Mult}, we get a sequence of letter-coboundaries $(c_n)_{n\geq 0}$ on $X_{\btau}^{(n)}$ such that
\begin{equation}
    \label{eq:cor:EigCharac:FAR:Mult:cocoboundary}
    \varepsilon_n \coloneqq \sup\bigl\{ \|c_n(u) - \alpha h_n(u)\| : u \in \cL(X_{\btau}^{(n)}) \bigr\}
\end{equation}
converges to 0 as $n \to \infty$.
If, instead, $(C_5)$ holds, we can apply \Cref{theo:EigCharac:FAR:Mult}, which gives $n_0 \ge 0$ and a letter-coboundary $c_{n_0}$ in $X_{\btau}^{(n_0)}$ such that \eqref{eq:hip_cob:theo:EigCharac:FAR:Mult} goes to 0 as $n \to \infty$; equivalently, the sequence $(c_n \coloneqq c_{n_0} \circ \tau_{n_0,n} : n \gt n_0)$ satisfies that $(\varepsilon_n : n \ge n_0)$ converges to $0$ as $n \to \infty$.

Next, we prove that each of the conditions $(C_i)$ implies that, for every $n \ge 0$ there exists $m(n) \gt n$ such that $c_n \circ \tau_{n,m(n)} \equiv 0$.
If $(C_1)$ holds, then every letter-coboundary is trivial by \Cref{thm:manifold}. 
If $(C_2)$ holds, then every letter-coboundary is trivial by \Cref{prop:dim}.
If $(C_3)$ holds, then every letter-coboundary is trivial by \Cref{prop:generatingabelian}.
Suppose that $(C_4)$ holds.
We only treat the case in which $\tau_n$ is left-proper; the other is similar.
Observe that, for all $n$, $c_n \circ \tau_n$ is a letter-coboundary on $X_{\btau}^{(n+1)}$ by \Cref{lem:compositioncoboundary}.
Also, if $b$ is the common first letter of each $\tau_n(a)$, then each $\tau_n(a)$ is a concatenation of return words to $b$.
Hence, $c_n(\tau_n(a))=0$ for all $a\in\cA_{n+1}$, and thus $c_n\circ\tau_n\equiv 0$.
Suppose now that $(C_5)$ is satisfied. 
By hypothesis, $1$ is not an eigenvalue of $M_{n,m}$ for all large $m\gt n$, so by \Cref{finitary=>cobords_are_eigenvectors}, $c_n\circ \tau_{n,m}$ must be trivial. 

We now use the previous claim and the convergence of \eqref{eq:cor:EigCharac:FAR:Mult:cocoboundary} to obtain $(P_3)$.  
Let $\varepsilon > 0$, and choose $n \in \N$ such that $\varepsilon_n < \varepsilon$.
For $k \geq m(n)$, we have
\[  c_n \circ \tau_{n,k} 
    = c_n \circ \tau_{n,m(n)} \circ \tau_{m(n),k} 
    \equiv 0.   \]
Thus, for all $u \in \cL(X_{\btau}^{(k)})$, $\|\alpha h_k(u)\| \leq \varepsilon_n < \varepsilon$ since $\tau_{n,k}(u) \in \cL(X_{\btau}^{(n)})$.
This proves $(P_3)$.

The implication $(P_2) \Rightarrow (P_4)$ is direct from \Cref{cor:EigCharac:FAR:eigs&freqs}.
\end{proof}



\section{On  coboundaries of specular subshifts and  linear involutions}
\label{sec:specular}

In this section, we determine exactly when the rational eigenvalues of  a minimal specular subshift are all trivial.
The main result is established in \Cref{theo:eigs_specular}, and crucially relies on developing an $S$-adic structure for this class of subshifts (\Cref{theo:Sadic_for_specular}), as well as on the use of letter-coboundaries and an analysis of the evolution of extension graphs between the levels of the $S$-adic structure.

The study of  specular subshifts is motivated by  the generalization of interval exchanges as  linear involutions.  A linear involution is an injective piecewise isometry defined on a pair of intervals. This generalization of the notion of interval exchange allows one to
work with nonorientable foliations on nonorientable surfaces. They are  first return maps  on transversals to  foliations defined by quadratic differentials. Linear involutions were introduced by Danthony and Nogueira  in\cite{DanthonyNogueira1988,DanthonyNogueira1990}, generalizing
interval exchanges with flip(s) (these are interval exchange transformations which reverse orientation in at least one interval).  The study
of linear involutions was later developed by Boissy and Lanneau in \cite{BoissyLanneau2009}.  See also \cite{Gabre:2012,Skrip:2023,LMS:21} for later references.
We rely here on $S$-adic representations based on return words
(more precisely mixed return words  such as as defined below), and  not on (Rauzy) induction.  This combinatorial viewpoint leverages the fact that the morphisms  associated to return words  are well understood (see \Cref{specular:ret_word_real_iso}).

\subsection{Mixed return words}

Let $\cA$ be a finite alphabet and $\theta \colon \cA \to \cA$ be an involution, \emph{i.e.}, $\theta(\theta(a)) = a$ for every $a \in \cA$.
Consider the group $G_\theta$ defined by the presentation $G_\theta = \langle a \in \cA \mid a\theta(a) = \varepsilon \rangle$, so that $\theta(a) = a^{-1}$ in $G_\theta$ for each $a \in \cA$.
Such a group $G_\theta$ is called a \emph{specular group}.
It is known (see e.g.\ \cite{specular_TCS}) that every specular group is isomorphic to the free product $\Z^{*i} * (\Z/2\Z)^{*j}$, where $i$ is the number of letters $a \in \cA$ satisfying $\theta(a) \neq a$ and $j$ is the number of letters satisfying $\theta(a) = a$.
In a specular group $G_\theta$, a \emph{reduced word} is an element of $\cA^*$ containing no factor of the form $\theta(a)a$ or $a\theta(a)$ for any $a \in \cA$.
Any element of $G_\theta$ can be uniquely represented as a reduced word.
A subset $B$ of a specular group $G$ is called a \emph{monoidal basis} of $G$ if it is \emph{symmetric} (closed under taking inverses), generates $G$ as a monoid, and every product $b_1 b_2 \cdots b_m$ of elements of $B$ such that $b_k b_{k+1} \neq \varepsilon$ for $1 \leq k < m$ is different from the identity element $\varepsilon$.

Given an alphabet $\cA$, an involution $\theta$ on $\cA$, and the specular group $G_\theta$, a minimal subshift on $\cA$ is {\em specular}   relative to $\theta$   if its language is  a symmetric set of   reduced words in $G_\theta$ in which the extension graph of every non-empty word is a tree and the extension graph of the empty word has exactly two connected components \cite{specular_TCS}.
Hence, in a specular subshift $X$ with involution $\theta$, the reduced word $u^{-1}$ obtained from a word $u \in \cL(X)$ by taking inverse in $G_\theta$ belongs to $\cL(X)$.

\begin{example}\label{ex:specular}
    Consider the substitution $\sigma \colon \{a,b,c,d\}^* \to \{a,b,c,d\}^*$ given by 
    \[  \sigma \colon \begin{cases}
        a & \mapsto ab \\
        b & \mapsto cda \\
        c & \mapsto cd \\
        d & \mapsto abc 
        \end{cases}     \]
    It is shown in \cite{specular_TCS} that $X_\sigma$ is a specular subshift  relative to the involution $\theta(a) = c$, $\theta(b) = d$.
\end{example}

So far, given a subshift $X \subseteq \cA^\Z$ and $u \in \cA$, we have called a {\em return word} to $u$ in $X$ a word $w \in \cA^*$ such that $wu \in \cL(X)$ and $wu$ contains exactly two occurrences of $u$, one as a prefix and one as a suffix. 
In this section, we need different variations of this notion, which we now recall from \cite{specular_TCS}.
Let $U\subseteq \cL(X)$ be a nonempty subset.
A {\em complete return word} to the set $U$ in $X$ is a word in $\cL(X)$ with exactly two occurrences of (possibly different) words in $U$, one as a prefix and one as a suffix.
Suppose now that $X$ is a  minimal specular subshift with involution $\theta$, and consider $u \in \cL(X)$ such that $u$ and $u^{-1}$ {\em do not overlap}, \emph{i.e.}, no nonempty prefix of one of $u,u^{-1}$ is a suffix of the other.
To each complete return word $w$ to $\{u,u^{-1}\}$, we associate a word $N(w)$ as follows: 
if $w$ contains $u^{-1}$ as a prefix, we erase it; if $w$ contains $u$ as a suffix, we also erase it, and call the resulting word $N(w)$.
These two operations can be made in any order without effect on the final result since $u$ and $u^{-1}$ do not overlap. 
The set of {\em mixed return words} to $u$ in $X$ is the set of all words $N(w)$, where $w$ is a complete return word to $\{u,u^{-1}\}$ in $X$.
We denote by $\cR_X(u)$ and $\cR^{\mathrm{mixed}}_X(u)$ the set of usual and mixed return words to $u$ in $X$.
For $L \subseteq \cL(X)$, we denote by $\cR^{\mathrm{comp}}_X(L)$ the set of complete return words to $L$ in $X$.

Let $X\subseteq \cA^\Z$ be a minimal specular subshift with involution $\theta$. 
Since $\cL(X)$ is biextendable, every letter $a\in\cA$ appears exactly twice in the extension graph $\Gamma_X(\varepsilon)$: 
once as a vertex in $L_X(\varepsilon)$, and once as a vertex in $R_X(\varepsilon)$. 
A letter is said to be {\em even} if these two occurrences are in the same connected component of $\Gamma_X(\varepsilon)$, 
and {\em odd} otherwise. 
A word $w\in\cL(X)$ is {\em even} if it has an even number of odd letters, and {\em odd} otherwise. 
The {\em even subgroup} is the subgroup of $G_\theta$ generated by even words. 
It is an index 2 subgroup of $G_\theta$  by \cite[Theorem 8.1, Proposition 4.8]{specular_TCS}.

\begin{lemma}\cite[Proposition 8.3.10]{DP}
    \label{specular:symmetry_in_Gamma(esp)}
    Let $X$ be a minimal specular subshift with involution $\theta$ and denote by $\Gamma^0,\Gamma^1$ the two connected components of $\Gamma_X(\varepsilon)$.
    If $(a_L,b_R) \in \Gamma^i_X$, with $i \in \{0,1\}$, then $(b^{-1}_L,a^{-1}_R) \in \Gamma^{1-i}_X$.
\end{lemma}

\begin{proposition}\cite[Theorem 6.12 and Theorem 6.17]{specular_TCS}
    \label{specular:mixed_gen_Gtheta}
    Let $X$ be a minimal specular subshift with involution $\theta$.
    For every $u \in \cL(X)$ such that $u,u^{-1}$ do not overlap, the set of mixed return words to $u$ in $X$ has $|\cA|$ elements, and is a symmetric monoidal basis of $G_\theta$.
\end{proposition}

Let $X$ be a minimal specular subshift and fix a nonempty word $u \in \cL(X)$.
By \Cref{specular:mixed_gen_Gtheta}, the set $\cR^\mathrm{mixed}_X(u)$ is symmetric and has exactly $|\cA|$ elements.
Thus there exists a bijection $\tau_u \colon \cA \to \cR^\mathrm{mixed}_X(u)$ satisfying $\tau_u(a^{-1}) = \tau_u(a)^{-1}$, which extends uniquely to a morphism $\tau_u \colon \cA^* \to \cA^*$ called the {\em mixed derived substitution} of $u$.
The condition $\tau_u(a^{-1}) = \tau_u(a)^{-1}$ ensures that $\tau_u$ descends to a group morphism $\tilde{\tau}_u \colon G_\theta \to G_\theta$, which is an isomorphism by \Cref{specular:mixed_gen_Gtheta}.
In the following lemma, we strengthen this property by showing that $\tau_u$ extends to an isomorphism $\hat{\tau}_u \colon \F_\cA \to \F_\cA$ of the underlying free groups.
This is crucial for developing our S-adic structure for minimal specular subshifts in \Cref{theo:Sadic_for_specular}.
We first introduce some terminology needed in the proof.

A set of words $Q$ is a {\em prefix code} (resp.\ {\em suffix code}) if no word in $Q$ is a prefix (resp.\ suffix) of another.
If $X$ is a subshift, a prefix code $Q$ is {\em $\cL(X)$-complete} if for every $u \in \cL(X)$ there exists $v \in Q$ such that one of $u, v$ is a prefix of the other; a suffix code $Q$ is {\em $\cL(X)$-complete} if the analogous condition holds with ``suffix'' in place of ``prefix''.
Given a subshift $X$, a word $w \in \cL(X)$, a suffix code $Q$, and a prefix code $Q'$, the generalized extension graph $\Gamma_{X,Q, Q'}(w)$ is the undirected bipartite graph whose vertex set is the disjoint union of a left copy $Q_L$ of $Q$ and a right copy $Q'_R$ of $Q'$, and whose edges are the pairs $(u_L, v_R) \in Q_L \times Q'_R$ such that $uwv \in \cL(X)$.
Since this graph is bipartite, there is no ambiguity in representing undirected edges as ordered pairs.

\begin{lemma}
\label{specular:ret_word_real_iso}
Let $X$ be a minimal specular subshift with involution $\theta$ and let $a \in \cL(X)$ be a letter with $a \neq \theta(a)$.
Then, the mixed derived substitution $\tau \colon \cA^* \to \cA^*$ of $a$ extends to an isomorphism $\hat{\tau} \colon \F_{\cA} \to \F_{\cA}$.
\end{lemma}
\begin{proof}
Let $Q^{\mathrm{p}}$ denote the set of nonempty strict prefixes of elements of $\cR^{\mathrm{comp}}_X(\{a,a^{-1}\})$, and let $Q^{\mathrm{s}}$ denote the set of nonempty strict suffixes of elements of $\cR^{\mathrm{comp}}_X(\{a,a^{-1}\})$.
By minimality of $X$, $Q^{\mathrm{p}}$ is a finite $\cL(X)$-complete suffix code and $Q^{\mathrm{s}}$ is a finite $\cL(X)$-complete prefix code.
Using the fact that $\Gamma_X(\varepsilon)$ consists of two trees and that $\Gamma_X(u)$ is a tree for every nonempty $u \in \cL(X)$, one can carry out the same strategy as in Section~8 of \cite{DolceP} to show that $\Gamma \coloneqq \Gamma_{X, Q^{\mathrm{s}}, Q^{\mathrm{p}}}(\varepsilon)$ consists of two trees $\Gamma^0$ and $\Gamma^1$.
Furthermore, \Cref{specular:symmetry_in_Gamma(esp)} implies that for every edge $(u_L,v_R)$,
\begin{equation}
    \label{eq:specular:ret_word_real_iso:1}
    (u_L,v_R) \in \Gamma^i 
    \ \Longleftrightarrow \
    (v^{-1}_L,u^{-1}_R) \in \Gamma^{1-i}.
\end{equation}

The main part of the proof is the claim below, which we state after introducing some notation.
For $u \in Q^{\mathrm{p}}$, which we know begins with either $a$ or $a^{-1}$, we define $\rho_L(u)$ by removing the initial letter $a^{-1}$ if present, and $\rho_L(u) = u$ otherwise. 
Symmetrically, for $v \in Q^{\mathrm{s}}$, we define $\rho_R(v)$ by removing the final letter $a$ if present, and $\rho_R(v) = v$ otherwise.
Observe that for any edge $(u_L,v_R) \in \Gamma$, we have $uv \in \cR_X^{\mathrm{comp}}(\{a,a^{-1}\})$ and $\rho_L(u)\, \rho_R(v) \in \cR_X^\mathrm{mixed}(a)$.

\textbf{Claim.}
Let $H$ be the subgroup of $\F_\cA$ generated by $\cR^\mathrm{mixed}_X(a)$.
Then $\rho_L(v), \rho_R(u) \in H$ for all $v_L \in Q^{\mathrm{p}}$ and $u_R \in Q^{\mathrm{s}}$.

{\em Proof of the claim.}
We say that a vertex of $\Gamma$ satisfies Property $(\Delta)$ if either it is a left vertex of the form $v_L$ with $\rho_L(v) \in H$, or it is a right vertex of the form $v_R$ with $\rho_R(v) \in H$.
Since the vertex set of $\Gamma$ is $(Q^{\mathrm{p}})_L \cup (Q^{\mathrm{s}})_R$, the claim is equivalent to the assertion that every vertex of $\Gamma$ satisfies $(\Delta)$.
It therefore suffices to establish the following two points:
\begin{enumerate}
    \item[(i)] each connected component of $\Gamma$ contains at least one vertex satisfying $(\Delta)$;
    \item[(ii)] if a vertex of $\Gamma$ is adjacent to a vertex satisfying $(\Delta)$, then it satisfies $(\Delta)$ as well.
\end{enumerate}
By minimality of $X$, we have $a^{-1} \in Q^{\mathrm{p}}$ and $a \in Q^{\mathrm{s}}$, and $\rho_L(a^{-1}) = \rho_R(a) = \varepsilon \in H$.
Thus $a^{-1}_L$ and $a_R$ both satisfy $(\Delta)$.
By \eqref{eq:specular:ret_word_real_iso:1}, $a_R$ and $a^{-1}_L$ belong to distinct connected components of $\Gamma$, hence~(i) holds.
For~(ii), take an edge $(u_L,v_R)$ in $\Gamma$ and assume that one of $\rho_L(u), \rho_R(v)$ lies in $H$. We show  that the other does as well.
Since $(u_L,v_R) \in \Gamma$, we have $uv \in \cR_X^\mathrm{comp}(\{a,a^{-1}\})$, so $\rho_L(u)\,\rho_R(v) \in \cR_X^\mathrm{mixed}(a) \subseteq H$.
Being $H$ a group and one of $\rho_L(u), \rho_R(v)$ lies in $H$, the other must as well.
This completes the proof of the claim.
\medskip

It remains to show that $H = \F_\cA$ (this implies that $\hat{\tau}_a$ is an isomorphism since $\F_\cA$ is Hopfian).
For this, it suffices to show that $H$ contains every letter of $\cA$.
Let $b \in \cA$ be arbitrary.
If $b = a$, then by minimality of $X$ there exists a return word $au \in \cR^\mathrm{comp}_X(\{a,a^{-1}\})$, so $(a_L, u_R) \in \Gamma$, and the claim gives $a = \rho_L(a) \in H$.
Similarly, if $b = a^{-1}$, there exists $ua^{-1} \in \cR^\mathrm{comp}_X(\{a,a^{-1}\})$, so $(u_L, a^{-1}_R) \in \Gamma$, and the claim gives $a^{-1} = \rho_R(a^{-1}) \in H$.
Finally, suppose $b \notin \{a, a^{-1}\}$.
By minimality, there exists $u \in \cR^\mathrm{comp}_X(\{a,a^{-1}\})$ containing an occurrence of $b$.
Since $u$ begins and ends with letters in $\{a, a^{-1}\}$, we can factorize $u = v\, b\, v'$ with $v \in Q^{\mathrm{p}}$ and $v' \in Q^{\mathrm{s}}$.
Then $\rho_L(v)\, b\, \rho_R(v') = \rho_L(\rho_R(u)) \in \cR_X^\mathrm{mixed}(a) \subseteq H$.
By the claim, $\rho_L(v), \rho_R(v') \in H$, so $b = \rho_L(v)^{-1}\, \rho_L(\rho_R(u))\, \rho_R(v')^{-1} \in H$.
\end{proof}

\subsection{Mixed derived subshifts and \texorpdfstring{$S$}{}-adic representation}

Let $X \subseteq \cA^\Z$ be a minimal specular subshift.
Given $u \in \cL(X)$, we let $\tau_u \colon \cA^* \to \cA^*$ be the mixed derived substitution of $u$ and define the {\em mixed derived subshift} of $u$ as the set $X_u$ of all sequences $y \in \cA^\Z$ such that $\tau_u(y) \in X$.
One easily verifies that $X_u$ is indeed closed and shift-invariant, and that $X$ is the shift-orbit of $\tau_u(X_u)$.

\begin{proposition}
    \label{specular:coding_is_specular}
    Let $X$ be a minimal specular subshift with involution $\theta$.
    For every letter $a \in \cL(X)$, with $a \neq \theta(a)$, the mixed derived subshift $X_a$ of $a$ is a minimal specular subshift.
\end{proposition}
\begin{proof}
Let $\tau_a \colon \cA_a^* \to \cA^*$ be the mixed derived substitution associated to the letter $a$.
For every $b \in \cA_a$, there are uniquely determined elements $\xi_L(b) \in \{\varepsilon,a^{-1}\}$ and $\xi_R(b) \in \{\varepsilon,a\}$ such that $\xi_L(b)\tau_a(b)\xi_R(b) \in \cR_X^{\mathrm{comp}}(\{a,a^{-1}\})$.
Explicitly, $\xi_L(b)=\varepsilon$ if $\tau_a(b)$ begins with $a$, and $\xi_L(b)=a^{-1}$ otherwise; 
similarly, $\xi_R(b)=\varepsilon$ if $\tau_a(b)$ ends with $a^{-1}$, and $\xi_R(b)=a$ otherwise.

It follows from recognizability of $\tau_a$ that
\begin{equation}
    \label{eq:specular:coding_is_specular:uniq_ext}
    \tau_a([w]) \subseteq S^{|\xi_L(b_0)|}[\xi_L(b_0)\tau_a(w)\xi_R(b_{k-1})]
    \enspace 
    \text{for all $w=b_0\cdots b_{k-1} \in \cL(X_a)$.}
\end{equation}
For a word $u$, let $\rho_L(u)$ be the word obtained from $u$ by deleting its first letter if this first letter is $a^{-1}$, and let $\rho_L(u)=u$ otherwise.
Similarly, let $\rho_R(u)$ be the word obtained from $u$ by deleting its last letter if this last letter is $a$, and let $\rho_R(u)=u$ otherwise.
Then $\rho_L(\rho_R(u)) = \rho_R(\rho_L(u)) \in \cR_X^{\mathrm{mixed}}(a)$ for every $u \in \cR_X^{\mathrm{comp}}(\{a,a^{-1}\})$.

We also define
\begin{equation}\label{eq:triangle}
\begin{aligned}
    \cR_X^{\mathrm{left}}(a) &\coloneqq 
    \{\xi_L(b)\rho_R(\tau_a(b)) : b \in \cA_a\} = 
    \{\rho_R(u) : u \in \cR^\mathrm{comp}_X(\{a,a^{-1}\})\} \\
    \cR_X^{\mathrm{right}}(a) &\coloneqq 
    \{\rho_L(\tau_a(b)) \xi_R(b) : b \in \cA_a\} = 
    \{\rho_L(u) : u \in \cR^\mathrm{comp}_X(\{a,a^{-1}\})\}.
\end{aligned}
\end{equation}
The set $\cR_X^{\mathrm{left}}(a)$ is a $G_L$-complete suffix code, and $\cR_X^{\mathrm{right}}(a)$ is an $G_R$-complete prefix code, where $G_L = (\cA^*\{a,a^{-1}\}) \cap \cL(X)$ and $G_R = (\{a,a^{-1}\}\cA^*) \cap \cL(X)$.

We now prove the proposition.
\Cref{specular:ret_word_real_iso} gives that $\tau_a$ has unimodular incidence matrix, hence $\tau_a$ is recognizable by \cite{BSTY} and $X_a$ inherits the minimality of $X$.
To prove that $X_a$ is specular, we need to check three conditions: 
laminarity, symmetry, and that $X_a$ is a tree condition of characteristic~$2$.
For laminarity, we notice that if $b\, b^{-1} \in \cL(X_a)$, then $\tau_a(b)\tau_a(b^{-1})$ (before reduction) belongs to $\cL(X)$, contradicting that $X$ is laminar. 
To show that $X_a$ is symmetric, consider an arbitrary $w = b_0 \dots b_{k-1} \in \cL(X_a)$ and observe that $\xi_L(b_0) \tau_a(w) \xi_R(b_{k-1}) \in \cL(X)$ by \eqref{eq:specular:coding_is_specular:uniq_ext}, which implies, since $X$ is symmetric, that 
\[ \xi_L(b_{k-1}^{-1}) \tau_a(w^{-1}) \xi_R(b_0^{-1}) = 
    \bigl(\xi_L(b_0) \tau_a(w) \xi_R(b_{k-1})\bigr)^{-1} \in \cL(X), 
\] 
where we used $\tau_a(t^{-1}) = \tau_a(t)^{-1}$ and analogous identities for $\xi_L$ and $\xi_R$. 
The unique $\tau_a$-desubstitution of this word is $w^{-1}$, so $w^{-1} \in \cL(X_a)$, proving symmetry.

It remains to prove that $X_a$ is a tree set of characteristic~2.

\emph{Step 1: nonempty words have tree extension graphs.}
By the definitions above, there are onto maps $\pi_L \colon \cA_L \to \cR_X^{\mathrm{left}}(a)$ and $\pi_R \colon \cA_R \to \cR_X^{\mathrm{right}}(a)$ given by $\pi_L(b_L) = \xi_L(b) \rho_R(\tau_a(b))$ and $\pi_R(b_R) = \rho_L(\tau_a(b))\xi_R(b)$.
Fix a nonempty word $w=b_0\cdots b_{k-1} \in \cL(X_a)$, and set $w'=\xi_L(b_0)\tau_a(w)\xi_R(b_{k-1})$ and $$\Gamma=\Gamma_{X,\cR_X^{\mathrm{left}}(a),\cR_X^{\mathrm{right}}(a)}(w').$$
For every edge $(c_L,d_R)$ of $\Gamma_{X_a}(w)$, the extended image $\pi_L(c_L)\, w'\, \pi_R(d_R)$ belongs to $\cL(X)$.
Thus, there is map $\pi_w \colon \Gamma_{X_a}(w) \to \Gamma$ defined by $(c_L,d_R) \mapsto (\pi_L(c_L),\pi_R(d_R))$ that preserves incidence, \emph{i.e.}, if two edges of $\Gamma_{X_a}(w)$ have a common vertex at a certain end, then their images by $\pi_w$ have common vertex at the same end.
Furthermore, $\pi_w$ is a bijection because edges of $\Gamma$ have the form $(\pi_L(b_L), \pi_R(b'_R))$ by 
\eqref{eq:triangle}, and the unique $\tau_a$-desubstitution of $\pi_L(b_L) \, w' \, \pi_R(b'_R)$ is $w$. 
This implies that $\Gamma_{X_a}(w)$ is a tree if and only if $\Gamma$ is a tree. 
We can show that $\Gamma$ is a tree by carrying out the same strategy as in Section~8 of \cite{DolceP}, using that extension graphs in $X$ of non-empty words are trees. 
Hence, $\Gamma_{X_a}(w)$ is a tree for non-empty words $w \in \cL(X_a)$.

{\em Step 2: $\Gamma_{X_a}(\varepsilon)$ has exactly two connected components, each  being a tree.} 
For any two edges $(c_L, d_R)$ and $(c_L, d'_R)$ of $\Gamma_{X_u}(\varepsilon)$ sharing a left vertex, it follows from the definitions that $\xi_L(d) = \xi_L(d')$, and that $\xi_R(c) = \varepsilon$ if and only if $\xi_L(d) = a^{-1}$. 
In particular, the values $\xi_R(c)$ and $\xi_L(d)$ are constant on the left and right sides of each connected component of $\Gamma_{X_a}(\varepsilon)$. 
This allows us to partition $\Gamma_{X_a}(\varepsilon)$ into 
\begin{align*} 
    K_\varepsilon &\coloneqq 
    \{(c_L, d_R) : \xi_L(d) = \varepsilon,\; \xi_R(c) = a\}, \\ 
    K_{a^{-1}} &\coloneqq 
    \{(c_L, d_R) : \xi_L(d) = a^{-1},\; \xi_R(c) = \varepsilon\}. 
\end{align*}

Note that 
 $\tau_a(b)$ ends with $a$ (resp. begins with $a^{-1}$) if $b_L$ (resp.\ $b_R$) is a vertex of $K_\varepsilon$ (resp.\ $K_{a^{-1}}$). 
So, there are no edges between $K_\varepsilon$ and $K_{a^{-1}}$ since for any such edge $(b_L,b'_R)$ the word $\tau_a(b) \tau_a(b')$ would, by \eqref{eq:specular:coding_is_specular:uniq_ext}, contain an occurrence of $a^{-1} a$ or $a a^{-1}$ at the concatenation position, which contradicts that $X$ is laminary. 
Thus, $\Gamma_{X_a}(\varepsilon)$ decomposes into two disconnected sets: 
$K_\varepsilon$ and $K_{a^{-1}}$. 
Set 
\[ \Gamma_\varepsilon = 
    \Gamma_{X,\,\cR^\mathrm{left}_X(a),\,\cR^\mathrm{right}_X(a)}(a) \quad \text{and} \quad
    \Gamma_{a^{-1}} = 
    \Gamma_{X,\,\cR^\mathrm{left}_X(a),\,\cR^\mathrm{right}_X(a)}(a^{-1}). 
\] 
By \eqref{eq:specular:coding_is_specular:uniq_ext}, there are maps 
\[  \pi'_\varepsilon \colon K_\varepsilon \to \Gamma_\varepsilon 
    \quad\text{and}\quad 
    \pi'_{a^{-1}} \colon K_{a^{-1}} \to \Gamma_{a^{-1}} 
\] 
given by $(b_L,b'_R) \mapsto (\pi_L(b_L), \pi_R(b'_R))$, where $\pi_L$ and $\pi_R$ are the onto maps given by $\pi_L(b_L) = \xi_L(b) \rho_R(\tau_a(b))$ and $\pi_R(b_R) = \rho_L(\tau_a(b))\xi_R(b)$ that were defined in Step 1.
The same recognizability argument used in Step 1 shows that the maps $\pi'_\varepsilon$ and $\pi'_{a^{-1}}$ are bijections that preserve incidence. 
Furthermore, following the strategy used in Section~8 of \cite{DolceP}, based on the fact that extension graphs in $X$ of nonempty words are trees, one can show that the graphs $\Gamma_{\varepsilon}$ and $\Gamma_{a^{-1}}$ are trees. 
Therefore, $K_\varepsilon$ and $K_{a^{-1}}$ are trees as well. 
We conclude that $\Gamma_{X_a}(\varepsilon)$ consists of two disconnected sets, $K_\varepsilon$ and $K_{a^{-1}}$, each of which is a tree. 
This completes the proof that $X_a$ is specular. 
\end{proof}

\begin{theorem}
\label{theo:Sadic_for_specular}
Let $X$ be a minimal aperiodic specular subshift with involution $\theta$ without fixed points.
Then there exists a recognizable, primitive directive sequence $\btau = (\tau_n \colon \cA_{n+1}^* \to \cA_n^* : n \ge 0)$ generating $X$ such that, for all $n \ge 0$:
\begin{enumerate}
    \item $X_{\btau}^{(n)}$ is a minimal specular subshift;
    \item $\tau_n$ is a composition of elementary substitutions and has a unimodular incidence matrix.
\end{enumerate}
\end{theorem}
\begin{proof}
The S-adic structure is built by iterating mixed return words, following the standard approach; we refer the reader to \cite[Proposition~3.8]{unimodular} and \cite[Proposition~5.22]{bifix_decoding} for close constructions.
Fix a letter $a \in \cA$ and define inductively $X_0 = X$, $\tau_n \colon \cA^* \to \cA^*$ as the mixed derived substitution of $a$ in $X_n$, and $X_{n+1}$ as the mixed derived subshift of $a$ from $X_n$.
By \Cref{specular:coding_is_specular}, each $X_n$ is a minimal specular subshift, so this induction is well-defined.
Moreover, \Cref{specular:ret_word_real_iso} ensures that each $\tau_n$ is a composition of elementary morphisms (thus it has a unimodular incidence matrix).

Set $\btau = (\tau_n : n \ge 0)$.
By construction, $X_{\btau}^{(n)} \subseteq X_n$ for every $n \ge 0$, and minimality of $X_n$ forces $X_{\btau}^{(n)} = X_n$.
Since $X$ is aperiodic, $\btau$ is everywhere growing, hence primitive.
Finally, recognizability follows from \cite[Theorem~3.1]{BSTY}, using that each $\tau_n$ has a unimodular (hence invertible) incidence matrix.
\end{proof}

The following lemma  revisits \Cref{lem:compositioncoboundary}.

\begin{lemma}
    \label{map_between_cob_spaces:iso}
    Let $Y \subseteq \cB^\Z$ be a minimal subshift, let $\tau \colon \cB^* \to \cA^*$ be a morphism, and let $X$ be the shift-orbit closure of $\tau(Y)$.
    Then the map $\Phi \colon c \mapsto c \circ \tau$ is an $\R$-linear map from $\cC(X)$ to $\cC(Y)$.
    If the incidence matrix of $\tau$ has full rank, then $\Phi$ is injective; if furthermore $\Gamma_X(\varepsilon)$ and $\Gamma_Y(\varepsilon)$ have the same number of connected components, then $\Phi$ is an isomorphism.
\end{lemma}
\begin{proof}
Let $c \in \cC(X)$. Then $\Phi(c)$ is a letter-coboundary of $Y$, \emph{i.e.}, $\Phi(c) \in \cC(Y)$, by  \Cref{lem:compositioncoboundary}.
That $\Phi$ is $\R$-linear is straightforward to verify.

If the incidence matrix of $\tau$ has full rank, then the Parikh vectors $\{\vec{\tau(b)} : b \in \cB\}$ span $\R^\cA$.
Since the value of a letter-coboundary on a word depends only on its Parikh vector, this spanning property implies that $\Phi$ has trivial kernel, hence is injective.
If additionally $\Gamma_X(\varepsilon)$ and $\Gamma_Y(\varepsilon)$ have the same number of connected components, then by \Cref{thm:manifold} the spaces $\cC(X)$ and $\cC(Y)$ have equal finite dimension over $\R$, and an injective linear map between finite-dimensional spaces of equal dimension is an isomorphism.
\end{proof}

\subsection{Spectral properties}

We now can describe rational eigenvalues 
of specular subshifts.  We will see below  that  the eigenvalue  $1/2$ which occurs when  every letter of $\cA$ is even  is associated  to a non-trivial coboundary.

\begin{theorem}
    \label{theo:eigs_specular}
    Let $X$ be a minimal specular subshift with involution without fixed points.
    Then,  the continuous additive rational eigenvalues of $X$ are trivial, \emph{i.e.},
    $E(X,S) \cap \mathbb{Q}= \mathbb{Z}$, unless every letter of $\cA$ is even, in which case $\frac{1}{2}\mathbb{Z} \subseteq E(X,S) \cap \mathbb{Q}$.
\end{theorem}
\begin{proof}
It suffices to show that if $1/q \in E(X,S)$ for some positive integer $q$, then $q = 1$, unless every letter is even.
By \Cref{theo:Sadic_for_specular}, $X$ is generated by a primitive, recognizable directive sequence $\btau = (\tau_n \colon \cA^* \to \cA^*)_{n \ge 0}$ in which every level subshift $X_{\btau}^{(n)}$ is a minimal specular subshift and every incidence matrix $M_n$ is unimodular.
Let $\vec{h}_n = (h_n(a))_a=\bigl(|\tau_{[0,n)}(a)|\bigr)_{a \in \cA} \in \mathbb{R}^\cA$ denote the length vector at level $n$, so that $\vec{h}_0 = (1,\dots,1)$ and $\vec{h}_0 M_0 \cdots M_{n-1} = \vec{h}_n$.
Applying the eigenvalue criterion \Cref{theo:EigCharac:FAR:Mult} to the eigenvalue $1/q$ yields sequences of integer row vectors $(\vec{w}_n)_{n \ge 0}$ in $\mathbb{Z}^\cA$, real row vectors $(\vec{v}_n)_{n \ge 0}$ in $\mathbb{R}^\cA$ with $\vec{v}_n \to \vec{0}$, and letter-coboundaries $(c_n)_{n \ge 0}$ with $c_n \in \cC(X_{\btau}^{(n)})$, such that
\begin{equation}\label{eq:eig_criterion}
    \frac{1}{q}\,\vec{h}_n = \vec{w}_n + \vec{v}_n + \vec{c}_n
    \qquad \text{for all } n \ge 0,
\end{equation}
where $\vec{c}_n \coloneqq (c_n(a))_{a \in \cA}$.
By \Cref{existence_limit_coboundary}, after passing to a contraction and relabeling levels, we may assume that $c_n = c_1$ for all $n \ge 1$.
Multiplying~\eqref{eq:eig_criterion} by $q$ then gives
\begin{equation}\label{eq:eig_mod}
    \vec{h}_n - q\vec{w}_n = q\vec{c}_1 + q\vec{v}_n
    \qquad \text{for all } n \ge 1.
\end{equation}
The left-hand side is integer-valued, while the right-hand side has $q\vec{c}_1$ constant and $q\vec{v}_n \to \vec {0}$, so $\vec{v}_n = \vec{0}$ for all sufficiently large $n$; fix such an $n$.

Since the matrices $M_k$ are unimodular, the vector $\vec{w} \coloneqq \vec{w}_n\,(M_0 \cdots M_{n-1})^{-1}$ lies in $\mathbb{Z}^\cA$.
Furthermore, both $\Gamma_X(v)$ and $\Gamma_{X_{\btau}^{(n)}}(v)$ have exactly two connected components, so \Cref{map_between_cob_spaces:iso} and the unimodularity of $M_0 \cdots M_{n-1}$ together imply that  the map $c \mapsto c \circ \tau_{[0,n)}$ is an isomorphism,  \emph{i.e.,} $\cC(X) \xrightarrow{\sim} \cC(X_{\btau}^{(n)})$.
Hence there exists $c \in \cC(X)$ with $c \circ \tau_{[0,n)} = c_n$; we set $\vec{c} \coloneqq (c(a))_{a \in \cA}$ so that $\vec{c}\,M_0 \cdots M_{n-1} = \vec{c}_n$.
Substituting into~\eqref{eq:eig_mod} and using $\vec{v}_n = 0$ yields
\begin{equation}\label{eq:eig_final}
    \vec{h}_0 = q\vec{w} + \vec{c}.
\end{equation}
Reducing modulo $q$ gives $c(a) \equiv 1 \pmod{q}$ for all $a \in \cA$.
If $X$ contains an odd letter $a$, then the vertices $a_R$ and $a_L$ lie in the same connected component of $\Gamma_X(v)$, which forces $c(a) = 0$.
Combined with $c(a) \equiv 1 \pmod{q}$, this gives $q = 1$.

It remains to show that $\tfrac{1}{2}\Z \subseteq E(X,S) $ when every letter is even.
Let $\Gamma^0$ and $\Gamma^1$ be the two connected components of $\Gamma_X(v)$, and for each $a \in \cA$, let $\zeta(a) \in \{0,1\}$ be the index such that $a_R \in \Gamma^{\zeta(a)}$.
Since every letter is even, \Cref{specular:symmetry_in_Gamma(esp)} gives $\theta(a)_L \in \Gamma^{1-\zeta(a)}$ for every $a$.
Therefore, whenever $ab \in \cL(X)$, we have $a_L \in \Gamma^{1-\zeta(a)}$ and $b_R \in \Gamma^{\zeta(b)}$, and since $a_L$ and $b_R$ must lie in the same component, we get $\zeta(b) \equiv \zeta(a) + 1 \pmod{2}$.
Thus consecutive letters in any word of $X$ alternate between $\Gamma^0$ and $\Gamma^1$, and $\zeta$ lifts to an eigenfunction $\bar{\zeta} \colon X \to \R/\Z$ for the eigenvalue $1/2$. Remark that  the eigenvalue  $1/2$ is associated  to the non-trivial coboundary $\zeta$.
\end{proof}

We now refer to \cite{BoissyLanneau2009,BDDPRR:17} for   definitions  and general properties of    linear involutions.
Since the natural coding of a minimal linear involution without connection is a minimal specular subshift  \cite{BDDPRR:17}, and since the notions of even and odd letters coincide for specular subshifts and linear involutions, the following is an immediate consequence of \Cref{theo:eigs_specular}.
\begin{corollary}
    \label{theo:eigs_linear_involutions}
    Let $X$ be the natural coding of a minimal linear involution without connection.
    Then  the rational eigenvalues of $X$ are trivial, except when every letter is even, in which case $\{k/2 : k \in \mathbb{Z}\} \subseteq E(X,S) \cap \mathbb{Q}$.
\end{corollary}
It is worth noting that this corollary does not extend to irrational eigenvalues.
Indeed, there are numerous examples of linear involutions without connections admitting irrational eigenvalues; see e.g.\ \cite{specular_TCS}, or  Example \ref{example:linearinvolution}.
However, it is conjectured that \emph{almost all} linear involutions in a certain natural subclass are weakly mixing (\emph{i.e.}, admit no nontrivial eigenvalues), in the same spirit as the Avila--Forni theorem (see \cite{AvilaForni}) for interval exchange transformations.
Pursuing this conjecture lies beyond the scope of the present paper, as it would require a detailed understanding of the dynamics of the substitutions $\tau_n$ arising in the S-adic representations of \Cref{theo:Sadic_for_specular}, but partial results in this direction have been announced in \cite{Arbulu}, in which letter-coboundaries and the criteria developed in this paper play a crucial role.

We close this section with an example of a natural coding of a linear involution without connection that has an irrational eigenvalue, showing that \ref{theo:eigs_specular} and Corollary \ref{theo:eigs_linear_involutions} cannot be extended to cover irrational eigenvalues. THis example is  based on a skew product  by ${\mathbb Z}/2{\mathbb Z}$ of the Fibonacci shift. 

\begin{example}\label{example:linearinvolution}
Consider the substitutions
\[
    \tau: \begin{cases}
        a \mapsto abaab \\ b \mapsto aba
    \end{cases}
    \qquad
    \sigma: \begin{cases}
        a_0 \mapsto a_0 b_1 a_1 a_0 b_1 \\
        b_0 \mapsto a_0 b_1 a_1 \\
        a_1 \mapsto a_1 b_0 a_0 a_1 b_0 \\
        b_1 \mapsto a_1 b_0 a_0
    \end{cases}
\]
and let $X_\tau \subseteq \{a,b\}^\Z$ and $X_\sigma \subseteq \{a_0,a_1,b_0,b_1\}^\Z$ be the substitutive subshifts they generate.
The morphism $\pi \colon \{a_0,a_1,b_0,b_1\}^* \to \{a,b\}^*$ that drops the subscript (\emph{i.e.},\ $\pi(a_i) = a$, $\pi(b_i) = b$) satisfies $\tau \circ \pi = \pi \circ \sigma$, so it extends to a factor map $\pi \colon X_\sigma \to X_\tau$ that is everywhere 2-to-1.
Concretely,  the subshift $X_\sigma$ can be realized as the skew product of $X_\tau$  and ${\mathbb Z}/2{\mathbb Z}$ with respect to the morphism  $h:X_{\sigma}  \rightarrow {\mathbb Z}/2{\mathbb Z}$  given by $h(x) = 1$ if $x_0 = a$ and $h(x) = 0$ if $x_0 = b$, \emph{i.e.},     the dynamical system  acting on ${\mathbb Z}/2{\mathbb Z}\times X_{\sigma}$
as 
 $(g,x) \mapsto (gh(x_0),Sx)$   (see  \emph{e.g.} \cite{BertheGouletOuelletNybergBroddaPerrinPetersen2024}
 for more details and \cite[Section 4.4]{specular_conf}).

The subshift $X_\tau$ is the Fibonacci subshift (as $\tau$ is the cube of the Fibonacci substitution), which is minimal and aperiodic.
One checks that the subshift $X_\sigma$ is likewise minimal and aperiodic (with the substitution $\sigma$ being primitive); one can verify that it is specular with involution $\theta(a_i) = a_{1-i}$, $\theta(b_i) = b_{1-i}$, and in fact $X_\sigma$ is the natural coding of a linear involution.

Using   Theorem \ref{theo:EigCharac:FAR:Mult} (see also the machinery developed in the proof of \Cref{theo:eigs_specular}),  let us  check that
$
    \alpha \coloneqq \tfrac{1}{4}(\sqrt{5}-1)
$
is an eigenvalue of $X_\sigma$, associated to the letter-coboundary $c $  defined by
\[ c \colon \{a_0,a_1,b_0,b_1\}^* \to \mathbb{R}, \quad 
    c(a_0) = -\tfrac{1}{2}, \quad
    c(b_0) = 0, \quad
    c(a_1) = \tfrac{1}{2}, \quad
    c(b_1) = 0.
\]
We first use  Remark \ref{rem:constantK}  for checking that $c$ is indeed a coboundary 
by noticing  that the  set of  factors of $X_{\sigma}$ of length $2$ is equal to  $\{a_1 b_0,  b_1 a_1,  b_0 a_0,  a_0 a_1,  a_0 b_1,  a_1 a_0\}$, and thus that the graph  $\Gamma_X(\varepsilon)$ has two connected components.



We  then restate the problem in  terms of linear algebra  by   using  \Cref{theo:EigCharac:FAR:Mult} (which in the substitutive case is equivalent to \cite{Host86}).
 We recall that  $\vec{\mathbf 1}$ stands for  the row vector $(1,1,1,1)$.
    Then, the height vector $\vec h_n = (h_n(a) : a \in \cA)$ equals $\vec{\mathbf 1} \, M_\sigma^n$.
  It is sufficient  to find a decomposition of the form  $$\alpha\, \vec h_0 = \alpha \vec{\mathbf 1}= \vec v + \vec w + \vec c ,$$ with $\vec w \in \Z^4$ and $\vec v  M_{\sigma} ^n $ converging exponentially fast to $0$. The  fact that  the condition  on $\vec {v}$ is sufficient for    guaranteeing  Item (2) of  \Cref{theo:EigCharac:FAR:Mult}      works   as in  the proof of
\Cref{balanced=>spaces_decomposition}, for example by using the Dumont-Thomas decomposition \cite{Dumont-Thomas} or the prefix-suffix automaton of $\tau$ \cite{CanSie}. By ordering  the alphabet as $a_0,a_1,b_0,b_1$,  a suitable decomposition is provided by 
$$ \alpha \vec{\mathbf 1} = \vec{w} + \vec{c} + \vec{v},$$
where 
$\vec{w}=(1,0,0,0)$ is an integer vector, $\vec{c}=(-1/2,1/2,0,0)$ is a coboundary vector  and $\vec{v}=( 1/4(\sqrt{5}-3),1/4(\sqrt{5}-3),1/4(\sqrt{5}-1),1/4(\sqrt{5}-1) )$ is a  left eigenvector associated to the  contracting eigenvalue $2 -\sqrt 5$   of the incidence matrix of $\sigma$.

In particular, since $E(X_\tau) = 2\alpha\Z + \Z$, the specular subshift $X_\sigma$ has strictly more eigenvalues than the Fibonacci subshift $X_\tau$.

Note that $X_{\sigma}$ is not balanced by  Theorem \ref{balanced=>spaces_decomposition}. Indeed,  $\Gamma_{X_\sigma}(\varepsilon)$ consists of two connected components, so the letter-coboundary space  $\vec{\cC}_\sigma$ (as defined in Section \ref{subsec:discrepancy_subst}) has dimension 1, but the matrix of $\sigma$ has two eigenvalues of modulus 1, hence $\vec{V}_\sigma + \vec{\cC}_\sigma$ has codimension at least 2.
\end{example}

\section{An automatic word in the rational base \texorpdfstring{$\tfrac{3}{2}$}{3/2}}
\label{sec:TM}

In this section, we  show how  our methods can be applied   to  handle  in details a natural example of   an infinite alphabet-rank example.
We consider the subshift $X$ generated by the so-called $\tfrac{3}{2}$--Thue--Morse word $\mathbf{t} \in \{0,1\}^\N$ and determine its group of additive eigenvalues; see \Cref{TM32:main} for the precise statement.
The analysis relies crucially on \Cref{theo:EigCharac:decisive:Mult} and on a study of letter-coboundaries, which we develop in \Cref{TM32:ext_graphs_are_connected}.

Rational bases were introduced by Akiyama, Frougny and Sakarovitch \cite{Akiyama}, which led to the notion of automatic words in such bases as a variation of the classical setting of automatic words in an integer base.
A general theory of such words has not yet been developed, so here we treat only the case of the $\tfrac{3}{2}$--Thue--Morse word, a representative example of an automatic word in a rational base, building on recent work \cite{TM32}.
The methods presented here are, however, quite general, provided one can establish basic results such as the minimality of the associated dynamical system.

The classical Thue--Morse word is generated by the Thue--Morse substitution. Our analysis of the $\tfrac{3}{2}$--Thue--Morse word  also rests on a description in terms of substitutions. 
For automatic words in an integer base, such descriptions are classical: these words can even be characterized as codings of fixed points of constant-length substitutions.  
In the rational-base setting, Rigo and Stipulanti~\cite{MichelManon} obtained an analogous description based on what we call here \emph{constant-length skew-substitutions}.
Our plan for determining the eigenvalues of $X$ is therefore to first translate the skew-substitution formalism into the S-adic terminology of the present paper (\Cref{TM32:Sadic_generates+minimality}), and then to apply \Cref{theo:EigCharac:decisive:Mult} to the resulting S-adic structure.
The last step requires passing to the $2$-block presentation to ensure decisiveness and some understanding of the language of $X$ (\Cref{TM32:ext_graphs_are_connected}).

\subsection{Skew-substitutions}
We briefly describe the notion of a skew-substitution, restricted to the context needed here.
A \emph{skew-substitution} over the alphabet $\cA$ is a map $\sigma \colon \Z/2\Z \times \cA^* \to \cA^*$ such that
\[	\sigma(i, a_0a_1 \dots a_{k-1}) = 
	\sigma(i, a_0) \, \sigma(i+1, a_1) \, 
	\dots \, \sigma(i+k-1, a_{k-1})
\]
for every $a_0a_1 \cdots a_{k-1} \in \cA^*$.
Equivalently, $\sigma(i, a_0a_1 \dots a_{k-1})$ is obtained by applying $\sigma(i+j \bmod 2, a_j)$ to each letter $a_j$, alternating between the two maps according to the parity of $j+i$.
In particular, $\sigma$ is uniquely determined by its values $\sigma(i,a)$ for $i \in \Z/2\Z$ and $a \in \cA$.
As with ordinary substitutions, $\sigma$ extends canonically to act on one-sided infinite words $x \in \cA^\N$.
We say that $\sigma$ is of \emph{constant length} (\emph{uniform} in some references) of \emph{type $(2,1)$} if $|\sigma(0,a)| = 2$ and $|\sigma(1,a)| = 1$ for every $a \in \cA$.

The $\tfrac{3}{2}$--Thue--Morse word $\mathbf{t}$ is defined by letting $\mathbf{t}_n$ be the digit-sum of the expansion of $n$ in base $\tfrac{3}{2}$, taken modulo $2$ (with the  expansion in base $\tfrac{3}{2}$ introduced in \cite{Akiyama}).
Rigo and Stipulanti~\cite{MichelManon} described $\mathbf{t}$ in terms of a uniform skew-substitution of type $(2,1)$.

\begin{theorem}
	\label{TM32:as_skew}
	Define $\cA = \{0,1\}$ and let $\sigma \colon \Z/2\Z \times \cA^* \to \cA^*$ be the uniform skew-substitution of type $(2,1)$ given by 
	\[  \sigma(0,\cdot) : \begin{cases}
			0 & \mapsto 00 \\
			1 & \mapsto 11
		\end{cases} \qquad 
		\sigma(1,\cdot) : \begin{cases}
			0 & \mapsto 1 \\
			1 & \mapsto 0
		\end{cases}     \]
	Then the $\tfrac{3}{2}$--Thue--Morse word is the unique infinite word $\mathbf{t} \in \{0,1\}^\N$ beginning with $0$ satisfying $\sigma(0, \mathbf{t}) = \mathbf{t}$.
\end{theorem}

From a skew-substitution one can define a natural S-adic structure that simulates it, as follows.
Consider a uniform skew-substitution $\sigma \colon \Z/2\Z \times \cA^* \to \cA^*$ of type $(2,1)$.
Define the map $P \colon \Z \to \Z$ by $P(j) = \lceil \tfrac{3j}{2} \rceil$, that is, $P(2k) = 3k$ and $P(2k+1) = 3k+2$.
Then, for any infinite word $x \in \cA^\N$ satisfying $\sigma(0,x) = x$, an occurrence of a word $u$ in $x$ at a position congruent to $j$ modulo $2^{n+1}$ gives rise to an occurrence of $\sigma(j \bmod 2, u)$ in $x$ at a position congruent to $P(j)$ modulo $3 \cdot 2^{n-1}$. This property will be used freely in what follows.

Let $\cA_n = \cA \times \Z/2^n\Z$, regarded as an alphabet.
We define the (ordinary) substitution $\tau_n \colon \cA_{n+1}^* \to \cA_n^*$ as follows.
Given $a \in \cA$ and $j \in \Z/2^{n+1}\Z$, write $\sigma(j \bmod 2, a) = b_0 b_1 \dots b_{k-1}$.
Then we set $\tau_n((a,j)) = (b_0, P(j))(b_1, P(j)+1) \cdots (b_{k-1}, P(j)+k-1)$.
This yields a directive sequence $\btau = (\tau_n)_{n \ge 0}$, which we call the \emph{directive sequence modelling $\sigma$}.

It follows from the definitions that
\begin{equation}
	\label{TM32:eq:Sadic_models_skew}
	\tau_{[0,n)}((a,j)) = 
	\sigma(P^{n-1}(j) \bmod 2, \sigma(P^{n-2}(j) \bmod 2, \dots \sigma(j \bmod 2, a) \dots ))
\end{equation}
for every $a \in \cA$, $j \in \Z/2^n\Z$, and $n \ge 1$.

The following two lemmas are proved straightforwardly by induction, so we omit their proofs.

\begin{lemma}
	\label{TM32:nLevel_is_just_a_coloring}
Let $y \in X_{\btau}^{(n)}$,  $x \in \cA^\Z$ and $z \in (\Z/2^n\Z)^\Z$ be the sequences defined by $y_k = (x_k, z_k)$ for all $k \in \Z$.
Then $x \in X$ and $z$ is periodic with period $2^n$.
\end{lemma}

\begin{lemma}
	\label{TM32:image_lengths}
Let $n \ge 1$ and let $u \in \cL(X_{\btau}^{(n)})$ satisfy $|u| \equiv 0 \pmod{2^n}$.
Then $|\tau_{[0,n)}(u)| = \bigl(\tfrac{3}{2}\bigr)^n |u|$.
\end{lemma}

From now on, $\sigma$ denotes the skew-substitution of \Cref{TM32:as_skew} and $\btau = (\tau_n \colon \cA_{n+1}^* \to \cA_n^* : n \ge 0)$ the directive sequence modelling it.
Let also $X \subseteq \cA^\Z$ be the subshift generated by the $\tfrac{3}{2}$--Thue--Morse word $\mathbf{t}$.

\begin{lemma}
	\label{TM32:Sadic_generates+minimality}
The subshift $X$ is generated by $\btau$.
Moreover, the $2^n$-th power $S^{2^n}$ of the shift acts minimally on $X$ for every $n \ge 0$.
\end{lemma}
\begin{proof}
We define the words $u_n \in \cA^*$ by $u_0 = 0$ and $u_{n+1} = \sigma(0, u_n)$ for $n \ge 0$.
Note that $|u_n| \to \infty$ and $n \to \infty$, and that $u_n$ is a prefix of $\mathbf{t}$ for every $n \ge 0$ since $\mathbf{t} = \sigma(0, \mathbf{t})$.
So, since \eqref{TM32:eq:Sadic_models_skew} ensures that $u_n = \tau_{[0,n)}((0,0))$, we deduce that $X_{\btau}$ contains $X$.

Let us now prove that $X_{\btau}$ is contained in $X$.
In \cite{TM32} it is proved that every word of the form $\sigma(i_1 \bmod 2, \sigma(i_2 \bmod 2, \dots \sigma(i_n \bmod 2, a) \dots))$, with $n \ge 1$, $a \in \cA$, and $i_1, \dots, i_n \in \{0,1\}$, occurs in $\mathbf{t}$.
In view of~\eqref{TM32:eq:Sadic_models_skew}, this is equivalent to saying that $\tau_{[0,n)}(a,j)$ occurs in $\mathbf{t}$ for every $n \ge 1$ and $(a,j) \in \cA_n$.
Therefore $X$ contains $X_{\btau}$, so $X = X_{\btau}$.
Finally, \cite[Proposition 14]{TM32} also ensures that every word occurring in $\mathbf{t}$ does so at all residue classes modulo $2^n$, for every $n \ge 0$; 
this gives the minimality of $(X, S^{2^n})$ for all $n \ge 0$.
\end{proof}

\subsection{A decisive \texorpdfstring{$S$}{S}-adic representation}

In order to apply \Cref{theo:EigCharac:decisive:Mult}, we need a directive sequence for $X$ that is recognizable, primitive, and decisive.
Recognizability of $\btau$ can be verified directly.
Unfortunately, $\btau$ is not decisive.
We address decisiveness by performing a block-code construction on $\sigma$ analogous to that used for ordinary substitutions in \Cref{subsec:block_presentations}.

We define the {\em centered $2$-block presentation} of $\sigma$ as follows.
Let $L$ denote the set of words of length $5$ in $\cL(X)$; 
one verifies that $L$ has 26 elements.
We identify $L$ with the English alphabet $\cB \coloneqq \{a,b,c,\dots,z\}$ via the lexicographic order.
We define a skew-substitution $\sigma' \colon \Z/2\Z \times \cB^* \to \cB^*$ as follows.
For $i \in \Z/2\Z$ and $b \in \cB$ encoding the word $a_{-2}a_{-1}a_0a_1a_2 \in \cA^5$, compute $w \coloneqq \sigma(i-2 \bmod 2, \, a_{-2}a_{-1}a_0a_1a_2)$ and read it through a sliding window of width $5$ (two symbols to the left, the current symbol, two to the right), starting at position $|\sigma(i-2 \bmod 2, \, a_{-2}a_{-1})| $ (where $\sigma(i \bmod 2, a_0)$ begins in $w$) and ending at position $|\sigma(i-2 \bmod 2, \, a_{-2}a_{-1}a_0)| - 1$ (where it ends).
This produces a word of length $|\sigma(i \bmod 2, a_0)|$ over $\cA^5$, which we then encode into $\cB$ letterwise.
The result is $\sigma'(i,b) \in \cB^{|\sigma(i \bmod 2,\, a_0)|}$, and the map $\sigma'$ so obtained is a uniform skew-substitution of type $(2,1)$ (since it has the same image lengths as $\sigma$).
Explicitly, one computes:

\begin{equation*}
    \sigma'(0, \cdot) = \left\{
    \begin{array}{lll}
    \begin{aligned}
    a &\mapsto ip \\
    b &\mapsto iq \\
    c &\mapsto hn \\
    d &\mapsto ho \\
    e &\mapsto my \\
    f &\mapsto lw \\
    g &\mapsto lx \\
    h &\mapsto bc \\
    i &\mapsto bd \\
    \end{aligned}
    &
    \begin{aligned}
    j &\mapsto gl \\
    k &\mapsto gm \\
    l &\mapsto fj \\
    m &\mapsto fk \\
    n &\mapsto up \\
    o &\mapsto uq \\
    p &\mapsto tn \\
    q &\mapsto to \\
    r &\mapsto yw \\
    \end{aligned}
    &
    \begin{aligned}
    s &\mapsto yx \\
    t &\mapsto oc \\
    u &\mapsto od \\
    v &\mapsto nb \\
    w &\mapsto sl \\
    x &\mapsto sm \\
    y &\mapsto rj \\
    z &\mapsto rk \\
    \end{aligned}
    \end{array}
    \right.
    \qquad
    \sigma'(1, \cdot) = \left\{
    \begin{array}{lll}
    \begin{aligned}
    a &\mapsto e \\
    b &\mapsto e \\
    c &\mapsto g \\
    d &\mapsto g \\
    e &\mapsto a \\
    f &\mapsto d \\
    g &\mapsto d \\
    h &\mapsto w \\
    i &\mapsto w \\
    \end{aligned}
    &
    \begin{aligned}
    j &\mapsto t \\
    k &\mapsto t \\
    l &\mapsto v \\
    m &\mapsto v \\
    n &\mapsto e \\
    o &\mapsto e \\
    p &\mapsto g \\
    q &\mapsto g \\
    r &\mapsto d \\
    \end{aligned}
    &
    \begin{aligned}
    s &\mapsto d \\
    t &\mapsto w \\
    u &\mapsto w \\
    v &\mapsto z \\
    w &\mapsto t \\
    x &\mapsto t \\
    y &\mapsto v \\
    z &\mapsto v \\
    \end{aligned}
    \end{array}
    \right.
\end{equation*}

We set $\cB_n = \cB \times \Z/2^n\Z$ and let $\btau' = (\tau'_n \colon \cB_{n+1}^* \to \cB_n^*)_{n \ge 0}$ be the directive sequence modeling $\sigma'$.
The centered $2$-block presentation of $\mathbf{t}$, denoted $\mathbf{t}'$, satisfies $\sigma'(0, \mathbf{t}') = \mathbf{t}'$.
We let $X'$ be the subshift generated by $\mathbf{t}'$; 
then, $X'$ is conjugate to $X$ and thus these two systems have the same eigenvalue group.
By \Cref{TM32:Sadic_generates+minimality}, $X'$ is generated by $\btau'$.
Note that $\btau'$ is recognizable but not everywhere growing; it does, however, satisfy the weaker property
\begin{equation*}
	\min\bigl\{
	|\tau'_{[0,n)}(ab)| : 
	ab \in \cL(X_{\btau'}^{(n)}) \cap \cB_n^2
	\bigr\} \to \infty 
	\enspace \text{as } n \to \infty.
\end{equation*}
An argument similar to that of \Cref{prop:decisive:suff_conds} then shows that the Kakutani--Rokhlin partitions associated to $\btau'$ generate the topology of $X'$ (it is here that the $2$-letter sliding block code is needed in place of the $1$-letter one used in \Cref{prop:decisive:suff_conds}).
This together with recognizability is sufficient for applying the necessary condition in \Cref{theo:EigCharac:decisive:Mult}, as a careful reading of its proof reveals.

The final ingredient needed to prove \Cref{TM32:main} is the following combinatorial computation on the language $X'$. 
It is based on a variant of extension graphs, constructed for each fixed step $\ell \ge 1$ from the non-adjacent pairs of letters $(\mathbf{t}'_k, \mathbf{t}'_{k+\ell})$ in $\mathbf{t}'$.
To understand letter-coboundaries in the S-adic structure given by $\boldsymbol{\tau}'$, it suffices to show that the graph $E_\ell$ obtained from these pairs is connected when $\ell$ is a power of $2$. 
However, proving this requires an induction over all steps $\ell \ge 1$.

\begin{lemma}
	\label{TM32:ext_graphs_are_connected}
For $\ell \ge 1$, define the undirected bipartite graph $E_\ell$ with vertex sets $\cB_L = \{a_L : a \in \cB\}$ and $\cB_R = \{b_R : b \in \cB\}$, and with an edge $(a_L, b_R)$ whenever $\mathbf{t}'_k = a$ and $\mathbf{t}'_{k+\ell} = b$ for some $k \in \N$.
Then $E_\ell$ is connected for every $\ell \ge 9$.
\end{lemma}
\begin{proof}
Suppose $(a_L, b_R) \in E_\ell$, so there exists a word $u$ of length $\ell+1$ beginning with $a$ and ending with $b$ that occurs in $\mathbf{t}'$ at some position $k$.
By \cite[Proposition~14]{TM32}, every subword of $\mathbf{t}$ occurs at both even and odd positions; since $\mathbf{t}'$ is a sliding block code of $\mathbf{t}$, the same holds for $\mathbf{t}'$.
In particular, $u$ occurs at both even and odd positions of $\mathbf{t}'$.
Since $\sigma'(0, \mathbf{t}') = \mathbf{t}'$, the words $\sigma'(i \bmod 2, u)$ also occur in $x$, and one deduces that
$\bigl(\sigma'(i \bmod 2, a)[r]_L,\; \sigma'(i+\ell \bmod 2, b)[s]_R\bigr) \in E_{\lfloor (3\ell+i)/2 \rfloor + s - r}$
for every $(a_L, b_R) \in E_\ell$, $i \in \Z/2\Z$, and valid indices $r, s$.
Inverting this argument via instead a desubstitution yields recurrence relations expressing $E_\ell$ in terms of $E_{\ell'}$ for smaller values $\ell'$.
To make this precise, let us denote
\[	\Phi(E_h;\, p, p', i, i') = 
	\bigl\{
	\bigl(\sigma'(i \bmod 2, a)[p]_L,\; \sigma'(i' \bmod 2, b)[p']_R\bigr)  :
	(a_L, b_R) \in E_h \bigr\};
\]
then, the recurrences take the explicit form: for every $\ell \ge 1$,
\begin{enumerate}
	\item $E_{3\ell} = \Phi(E_{2\ell};0,0,0,0) \cup \Phi(E_{2\ell};1,1,0,0) \cup \Phi(E_{2\ell};0,0,1,1)$;
	\item $E_{3\ell+1} = \Phi(E_{2\ell}; 0,1,0,0) \cup \Phi(E_{2\ell+1};1,0,0,1) \cup \Phi(E_{2\ell+1};0,0,1,0)$;
	\item $E_{3\ell+2} = \Phi(E_{2\ell+1};0,0,0,1) \cup \Phi(E_{2\ell+2};1,0,0,0) \cup \Phi(E_{2\ell+1};0,1,1,0)$.
\end{enumerate}
Since the language of $\mathbf{t}'$ is decidable, a computer program can compute $E_\ell$ for any given $\ell$.
This and the formulas above permit to verify that the sequence $(E_\ell)_{\ell \ge 1}$ is eventually periodic with period $27$, starting from $\ell = 9$.
A direct check then confirms that $E_\ell$ is connected for every $9 \le \ell < 9 + 27$, and the lemma follows thanks to the eventual periodicity of $(E_\ell)_{\ell \ge 1}$.
\end{proof}

\begin{theorem}
	\label{TM32:main}
The subshift $X$ generated by the $\tfrac{3}{2}$--Thue--Morse word has additive eigenvalue group
\[	E(X,S) = \bigl\{ \tfrac{k}{3^n} :  k\in\Z,\, n \ge 0 \bigr\}.	\]
\end{theorem}
Note that we expect $X$ to be an almost 2-to-1 extension of the $3$-adic odometer~\cite[Lemma 9 and Proposition 15]{TM32}.
\begin{proof}
The inclusion $\frac{1}{3^n} \in E(X,S)$ can be verified by showing that $U_n \coloneqq \bigcup_{a \in \cB} \tau'_{[0,n)}([(a,0)])$ is a clopen cyclic partition of $X'$ of size $3^n$.
Alternatively, in \cite{TM32} the authors prove that the map $\pi \colon X \to \cA^\Z$ defined by $\pi(x)_k = (x_{k+1} + x_k) \bmod 2$ is a factor map onto a subshift $Y$ that can be identified as the subshift generated by a $(9,4)$-Toeplitz word in the sense of \cite{CassaigneKarhumaki}.
In particular, $Y$ is a Toeplitz subshift whose maximal equicontinuous factor is $\Z_3$ \cite{CassaigneKarhumaki}.
It follows that the maximal equicontinuous factor of $X$ itself factors onto $\Z_3$, which is equivalent to $E(X,S)$ containing $\frac{k}{3^n}$ for all $k \in \Z$ and $n \ge 0$.

We now prove the reverse inclusion. 
Let $\alpha \in E(X,S)$; we show that $\alpha = \frac{k}{3^n}$ for some $k \in \Z$ and $n \ge 0$.
As explained in the paragraph before \Cref{TM32:ext_graphs_are_connected}, the necessary condition of \Cref{theo:EigCharac:decisive:Mult} holds for $\btau'$ and $\alpha$, hence, there exist a level $n \ge 10$ and a letter-coboundary $c \colon \cB_n^* \to \R$ in $X_{\btau'}^{(n)}$ such that
\begin{equation}
	\sup\bigl\{ \|c(u) - \alpha\, h_n(u)\| : 
	u \in \cL(X_{\btau'}^{(n)}) \bigr\}  
	< \frac{1}{6|\cB|}.
\end{equation}
\Cref{thm:manifold} further yields a map $\rho \colon \cB_n \to \R$ such that $c(ab) = \rho(b) - \rho(a)$ for every $ab \in \cL(X_{\btau'}^{(n)})$.

Fix $y \in X_{\btau'}^{(n)}$ with $y_0 = (x_0, 0)$.
By \Cref{TM32:nLevel_is_just_a_coloring}, there exists $x \in X'$ such that $y_k = (x_k, k \bmod 2^n)$ for all $k \in \Z$.
Then $c(y_{[0,\, j2^n)}) = \rho((x_{j2^n}, 0)) - \rho((x_0, 0))$ for all $j \in \Z$, and substituting into the bound above gives
\[	\|\rho((x_{j2^n},0)) - \alpha\, h_n(y_{[0,\,j2^n)}) - \rho((x_0,0))\|
	= \|c(y_{[0,\,j2^n)}) - \alpha\, h_n(y_{[0,\,j2^n)})\|
	< \frac{1}{6|\cB|}.
\]
By \Cref{TM32:image_lengths}, $h_n(y_{[0,\, j2^n)}) = |\tau'_{[0,n)}(y_{[0,\, j2^n)})| = 3^n j$, so the above simplifies to
\begin{equation}
	\label{TM32:eq:rho_approx}
	\|\rho((x_{j2^n},0)) - 3^n j\alpha - \rho((x_0,0))\| < \frac{1}{6|\cB|}
\end{equation}
for all $j \in \Z$.
Taking differences between consecutive values of $j$ yields
\[	\|\rho((x_{(j+1)2^n},0)) - \rho((x_{j2^n},0)) - 3^n \alpha\| <  \frac{1}{3|\cB|}
\]
for all $j \in \Z$.
Equivalently, $\|\rho((a,0)) - \rho((b,0)) - 3^n \alpha\| < \tfrac{1}{3|\cB|}$ for every edge $(a_L, b_R) \in E_{2^n}$.
In particular, whenever two letters $b, b' \in \cB$ share a common left neighbor in $E_{2^n}$, one has $\|\rho((b',0)) - \rho((b,0))\| < \tfrac{2}{3|\cB|}$.
Since $E_{2^n}$ is connected by \Cref{TM32:ext_graphs_are_connected}, iterating this estimate along any path gives $\|\rho((b',0)) - \rho((b,0))\| < \tfrac{1}{3}$ for all $b, b' \in \cB$.
Combining this uniform bound on $\rho$ with~\eqref{TM32:eq:rho_approx} gives $\|3^n j\alpha\| < \tfrac{1}{3}$ for all $j \in \Z$, which forces $\alpha = p/3^n$ for some $p \in \Z$.
\end{proof}

\section{Further examples}\label{sec:examples}

In this section we provide several examples that   prove  among others the sharpness of the hypothesis in our results, as developed in \Cref{subsec:ex:nocoboundary}   which illustrates the fact   the existence of an eigenvalue does not imply the existence of a letter-coboundary satisfying the convergence
from \Cref{theo:EigCharac:FAR:Mult}. \Cref{subsec:nontrivialdynamical} shows that for any minimal substitutive subshift with an irrational eigenvalue, there exists a substitution generating a subshift conjugate to the original one in which all letter-coboundaries associated to that irrational eigenvalue are non-trivial.  \Cref{subsec:Itau} provides an example  of application of \Cref{cor:EigCharac:FAR:eigs&freqs} which relates  eigenvalues and measures of the  base  atoms of  the towers. \Cref{ex:sturmian_avec_cobord}  and \Cref{ex:CM} aim to illustrate the difference between coboundaries which are strictly combinatorial and those which are associated to eigenvalues in the sense of \Cref{theo:EigCharac:decisive:Mult,theo:EigCharac:FAR:Mult,theo:EigCharac:decisive:positive}, clearly showing that either type may or may not be trivial in general.  In particular, \Cref{subsec:nontrivial}
illustrates the fact that non-trivial coboundaries can appear in minimal subshifts without having
any impact on eigenvalues.
\Cref{subsec:nonproper}  shows that non-proper directive sequences can also have trivial coboundaries only.
We end this section 
with a characterization of  the group of additive eigenvalues of  systems  generated by directive sequences of constant length and finite alphabet rank,  and show in particular that only rational eigenvalues are possible, a result which was proved in \cite{BMY} by very different methods.

\subsection{Eigenvalues without coboundaries}
\label{subsec:ex:nocoboundary}

In this section we exhibit primitive and recognizable directive sequences  (with unbounded alphabet rank) generating  shifts  having an additive eigenvalue that admits no coboundary characterizing  this  eigenvalue.
This proves the fact that Theorems \ref{theo:EigCharac:FAR:Mult} and \ref{theo:EigCharac:decisive:positive} cannot be extended to all primitive and recognizable directive sequences, since for some directive sequences, the existence of an eigenvalue does not imply the existence of a letter-coboundary satisfying the convergence from \Cref{theo:EigCharac:FAR:Mult}.

The example is based on the following construction. Let $\cA = \{0,1\}$, $Y \subseteq \cA^\Z$ be any Sturmian subshift, and $\sigma \colon \cA^* \to \cA^*$ be the period-doubling substitution $\sigma(0) = 01$, $\sigma(1) = 00$.
We can view the  (direct) product system $(X_{\sigma} \times Y, S\times S)$ as a subshift over the alphabet $\cA\times\cA$.
We start with a technical lemma which  guarantees minimality of this subshift.

\begin{lemma}
    \label{lem:sturmXperdoubl_is_minimal}
    Let $\cA = \{0,1\}$, $Y \subseteq \cA^\Z$ be any Sturmian subshift, and $\sigma \colon \cA^* \to \cA^*$ be the period-doubling substitution $\sigma(0) = 01$, $\sigma(1) = 00$.
    The product system $(X_{\sigma} \times Y, S\times S)$ is minimal.
\end{lemma}
\begin{proof}
    We give two proofs: a short one using machinery from topological dynamics, and a combinatorial one inspired by  the recent analysis of finite skew-products of minimal subshifts from \cite{densityLanguages}; see also \cite{Balkova2016}.
    
    It is well known that $Y$ and $X_\sigma$ are an almost 1-to-1 extensions of their maximal equicontinuous factors, which are an irrational one-dimensional rotation and an odometer.
    In particular, they are point-distal systems \cite{ellis} and do not share any non-trivial eigenvalues.
    Therefore, by a classic result in topological dynamics \cite[Chapter 11]{ellis}, $Y$ and $X_\sigma$ are disjoint, {\em i.e.}, the product system $(X_{\sigma} \times Y, S\times S)$ is minimal.
    \medskip
    
    We now give a second proof, of a combinatorial nature.
    Let $z = (x,y) \in X_\sigma \times  Y$, $v \in \cL(Y)$ and $u \in \cL(X_\sigma)$ be arbitrary.
    It is enough to show that there exists $i \in \Z$ such that $y_{[i,i+|v|)} = v$ and $x_{[i,i+|u|)} = u$, as this would show that orbit in $(Y \times X_\sigma, S \times S)$ is dense.
    
    First, we use the primitivity and left-properness of $\sigma$ to find $n \geq 1$ and $p \in \cA^*$ such that $p u$ is a prefix of both $\sigma^n(0)$ and $\sigma^n(1)$.
    Next, we take a $\sigma^n$-factorization $(x',k)$ of $x$ in $X_\sigma$. 
    Since $|\sigma^n(0)| = |\sigma^n(1)| = 2^n$, we have that $S^{j2^n-k}x = \sigma^n(S^j x)$ for any $j \in \Z$.
    Hence, $x_{[j2^n-k,j2^n-k+2^n)} = \sigma^n(x'_j)$, and thus
    \begin{equation}
        \label{eq:lem:sturmXperdoubl_is_minimal:0}
        x_{[j2^n-k+|p|,j2^n-k+|p|+|u|)} = u
        \enspace \text{for all $j \in \Z$.}
    \end{equation}
    We now relate the previous recurrence times of $u$ in $x$ with the recurrence times of $v$ in $y$ using a result of Balkova {\em et al.}
    Consider the morphism $\varphi \colon \cA^* \to \Z/2^n\Z$ defined by $\varphi(0) = \varphi(1) = 1$.
    Since $Y$ is Sturmian, \cite[Theorem 3.3]{Balkova2016} ensures that the set $\cR_Y(v)$ of return words to $v$ in $Y$ satisfies 
    \begin{multline}
        \label{eq:lem:sturmXperdoubl_is_minimal:1}
        \{ j \in \Z/2^n\Z : \text{$y_{[i,i+|v|)} = v$ for some $i \in 2^n\Z + j$} \} \\ = 
        \{ \varphi(y_{[0,i)}) : i \geq 1,\, y_{[i,i+|v|)} = v\} = \Z/2^n\Z.
    \end{multline}
    Hence, there exists $i \geq 1$ such that $y_{[i,i+|v|)} = v$ and $i \in 2^n\Z + |p|+k$.
    Using \eqref{eq:lem:sturmXperdoubl_is_minimal:0} with $j \coloneqq (i-|p|-k)/2^n \in \Z$ gives $x_{[i, i+|u|)} = u$.
    As $y_{[i,i+|v|)} = v$, the proof is complete.
\end{proof}

\begin{theorem}
    \label{thm:NoCobsInftyRank}
    There exist a recognizable and primitive directive sequence $\btau = (\tau_n \colon \cA_{n+1}^* \to \cA_n^*)_{n\geq0}$ and an eigenvalue $\alpha$ of $X_{\btau}$ such that, for any choice of letter-coboundaries $c_n \colon \cA_n^* \to \R$ in $X_{\btau}^{(n)}$, for $n \geq 0$,
    \begin{equation} \label{eq:statement:thm:NoCobsInftyRank}
        \liminf_{n\to\infty} \max_{a\in\cA_n} \|c_n(a) - \alpha h_n(a)\| > 0.
    \end{equation}
    In particular, the criteria of Theorems \ref{theo:EigCharac:FAR:Mult} and \ref{theo:EigCharac:decisive:positive} are not satisfied for any sequence of coboundaries.
\end{theorem}

\begin{proof}
Let $\cA = \{0,1\}$, $Y \subseteq \cA^\Z$ be any Sturmian subshift, and $\sigma \colon \cA^* \to \cA^*$ be the period-doubling substitution $\sigma(0) = 01$, $\sigma(1) = 00$.
By \Cref{lem:sturmXperdoubl_is_minimal}, the product system $(X_\sigma \times Y, S \times S)$ is minimal; in particular,
\begin{equation}
    \label{eq:1:thm:NoCobsInftyRank}
    \text{$\cL_k(X_{\sigma} \times Y) = \cL_k(X_\sigma) \times \cL_k(Y)$ for $k \geq 0$.}
\end{equation}
We are going to use this property to obtain an S-adic structure for $X_\sigma \times Y$.

Fix, for every $n \geq 0$, a bijection $\xi_n \colon \cA \times \cL_{2^n}(Y) \to \cC_n$ between $\cA \times \cL_{2^n}(X_\sigma)$ and an alphabet $\cC_n$.
We define a morphism $\tau_n \colon \cC_{n+1}^* \to \cC_n^*$ as follows.
For $(a,w) \in \cA \times \cL_{2^{n+1}}(Y)$, we write $\sigma(a) = a_0 a_1$ and $w = u v$, where $|u| = |v| = 2^n$ and $u,v \in \cL_{2^n}(Y)$, and then set $\tau_n(\xi_{n+1}(a,w)) = \xi_n(a_0, u) \xi_n(a_1, v)$.
Since $\sigma$ is recognizable over $X_\sigma$, it is not difficult to check that $\btau \coloneqq (\tau_n : n \geq 0)$ is recognizable.
Moreover, we have that
\begin{equation}
    \label{eq:2:thm:NoCobsInftyRank}
    \tau_{0,n}(\xi_n(a,w)) = (\sigma^n(a), w)
    \enspace \text{for all $(a,w) \in \cA \times \cL_{2^n}(Y)$,}
\end{equation}
so $\btau$ is everywhere growing and generates  $X_\sigma \times Y$ (for every  word $(v,v')$ occurring in   $X_\sigma \times Y$, there exist $n $, $a $ and $w$ such that $(v,v')$ occurs in $(\sigma^n(a), w)$). 
Since  $X_\sigma \times Y$ minimal, we also have that $\btau$ is primitive.

We now show that each $X_{\btau}^{(n)}$ admits the trivial coboundary only.
In view of \Cref{thm:manifold}, it is enough to show that $\Gamma_{X_{\btau}^{(n)}}(\varepsilon)$ is connected.
Observe that, by \eqref{eq:1:thm:NoCobsInftyRank}, two vertices $(a_L,u)_L$ and $(b_R,v)_R$ of $\Gamma_{X_{\btau}^{(n)}}(\varepsilon)$ are connected if and only if $a_L$ and $b_R$ are connected in $\Gamma_{X_{\sigma}}(\varepsilon)$ and $u_L$ and $v_R$ are connected in $\Gamma_{Y_n}(\varepsilon)$, where 
\[ Y_n = \{(y_{[k\cdot 2^n,(k+1)\cdot 2^n)})_{k \in \Z} : y \in Y\} \]
is the $2^n$-power shift of $Y$.
Now, a direct computation shows that $\Gamma_{X_\sigma}(\varepsilon)$ is connected.
Also, since $Y$ is a dendric subshift \cite[Example 3.2]{dendric}, $Y_n$ is dendric as well \cite[Theorem 3.13]{dendric}, so $\Gamma_{Y_n}(\varepsilon)$ is connected.
We deduce that $\Gamma_{X_{\btau}^{(n)}}(\varepsilon)$ is connected, from which it follows that $X_{\btau}^{(n)}$ admits no non-trivial  coboundary.

We now prove that $X_\sigma \times Y$ satisfies the conclusion of the theorem.
Observe that $(Y,S)$ is a dynamical factor of $(X_\sigma \times Y, S\times S)$.
Thus, if $\alpha$ is the parameter of  the Sturmian subshift $Y$ (see Definition \ref{def:dendric}), then $\alpha$ is an additive eigenvalue of $X_\sigma \times Y$.
We may choose $\alpha$ such that $(\|2^n \alpha\| : n \geq 0)$ stays bounded away from $0$, for instance, by ensuring that both digits $0$ and $1$ occur in the binary expansion of $\alpha$ with bounded gaps.

With the aim of obtaining a contradiction, assume that there exist coboundaries $c_n \colon \cA_n^* \to \R$ in $X_{\btau}^{(n)}$ for which \eqref{eq:statement:thm:NoCobsInftyRank} fails to hold.
Since each $X_{\btau}^{(n)}$ admits no non-trivial coboundary, we have that 
\[ \liminf_{n \to \infty} \| \alpha h_n(a) \| = 0. \]
Now, from \eqref{eq:2:thm:NoCobsInftyRank} we see that $h_n(a) = 2^n$, so $\| 2^n \alpha \| \to 0$ as $n \to \infty$.
This contradicts the  above assumption on  $\alpha$, proving \eqref{eq:statement:thm:NoCobsInftyRank}.
By combining this with \Cref{rem:FAR_crit=>Host}, we conclude that the criteria of Theorems \ref{theo:EigCharac:FAR:Mult} and \ref{theo:EigCharac:decisive:positive} cannot be satisfied for any sequence of coboundaries.
\end{proof}
Note that the construction developed above    provides    everywhere growing and primitive  $S$-adic expansions to  minimal shifts  distinct from  the expansions discussed in Remark \ref{rem:anyminimal}.

\subsection{Non-trivial letter-coboundaries}
\label{subsec:nontrivialdynamical}

The following result shows that for any minimal substitutive subshift with an irrational eigenvalue, there exists a substitution generating a subshift conjugate to the original one in which all letter-coboundaries associated to that irrational eigenvalue are non-trivial.

\begin{proposition} 
    \label{prop:exs_nontrival_cobs}
    Let $\tau \colon \cA^* \to \cA^*$ be a primitive substitution.
    Suppose that $X_\tau$ has an irrational eigenvalue $\alpha$.
    Then, there exists a primitive substitution $\sigma \colon \cB^* \to \cB^*$ generating a subshift conjugate to $X$ and $a \in \cA$ such that $\alpha |\sigma^n(a)|$ stays bounded away from $\Z$ as $n \to \infty$.
    
    In particular, if $n \geq 0$ and $c_n$ is any letter-coboundary in $X_\sigma$ satisfying Condition \eqref{eq:hip_cob:theo:EigCharac:FAR:Mult} of \Cref{theo:EigCharac:FAR:Mult}  with $\btau = (\tau : n \geq 0)$ and $\alpha$,
    then $c_n$ is nontrivial.
\end{proposition}
\begin{proof}
We start with some simplifications.
Since any minimal substitutive subshift is generated, up to a conjugacy, by a proper primitive substitution \cite{DURAND_HOST_SKAU_1999}, we may assume that $\tau$ is proper.
Since $X_\tau$ aperiodic (as $\alpha$ is irrational), there exists $k \geq 1$ such that $a^k \notin \cL(\tau)$ for all $a \in \cA$.
This implies that the centered $k$-block presentation $\tau^{[k]} \colon (\cA^{[k]})^* \to (\cA^{[k]})^*$ of $\tau$ (see Section \ref{subsec:block_presentations}) satisfies that $a \notin E^R_{X_{\tau^{[k]}}}(a)$ for all $a \in \cA^{[k]}$.
Now, it is known that $\tau^{[k]}$ inherits from $\tau$ the fact  that $\tau^{[k]}$ is primitive, that some power of $\tau^{[k]}$  also inherits  properness, and that $\tau^{[k]}$ generates a subshift conjugate to $X_\tau$ by a sliding block code of radius 0 (see \Cref{subsec:block_presentations} or \cite[Section 2.4.5]{DP}).
So, we may assume without loss of generality that $\tau$ itself satisfies that $a \notin E^R_{X_{\tau}}(a)$ for all $a \in \cA$, that is,
 there exists no letter $b$  such that $baa \in  \cL(\tau)$.

As $\tau$ is proper, there is a letter $a \in \cA$ such that $\tau(b)$ starts with $a$ for each $b \in \cA$.
We need one additional simplification.
Since $\alpha$ is an eigenvalue of $X_{\tau}$ and $\tau$ is primitive and proper, we can use \Cref{theo:EigCharac:FAR:Mult}  to obtain a decomposition  of the form $\alpha \vec{\mathbf 1} = \vec w + \vec v$, where $\vec w, \vec v \in \R^\Z$ satisfy that $\vec w M_\tau^n \in \Z^\cA$ for some $n \geq 0$ and $\vec v M_\tau^m$ converges exponentially fast to $0$.
Note that $\vec w + \vec v \notin \Z^\cA$ as $\alpha$ is irrational, so $\|\vec w + \vec v\| > 0$.
Hence, up to considering a power of $\tau$, we have that
\begin{equation}  \label{eq:exs_nontrival_cobs:acc_conv}
    \text{$\vec w M_\tau \in \Z^\cA$ and $\sum_{j \geq 1} \| \vec v M_\tau^j\| < \frac{1}{2}\|\vec w + \vec v\|$.}
\end{equation}

Let us now construct the new substitution $\sigma$.
For $b \in E^R_{X_\tau}(a)$, we can write $\tau(b) = a u_b$.
Then, we define $\sigma \colon \cA^* \to \cA^*$ by
\begin{equation} \label{eq:exs_nontrival_cobs:defi_sigma}
    \sigma(b) =\begin{cases}
        \tau(a) a   & \text{if $b = a$} \\
        u_b   & \text{if $b \in E^R_{X_\tau}(a)$} \\
        \tau(b) & \text{otherwise}.
    \end{cases}
\end{equation}
Note that this definition is consistent and that $\sigma$ is non-erasing since $a \notin E^R_{X_\tau}(a)$.  Otherwise, let $b \in E^R_{X_\tau}(a)$ be  such that  $u_b$
is equal to  the empty  word. Since $ab  \in   \cL(\tau)$,  there exists a letter $c$  (possibly equal to $a$ or $b$) such that
$abc \in   \cL(\tau)$,  which implies  $aaa  \in   \cL(\tau)$, since the image of $c$ starts with $a$ by properness. 
We also define $\psi \colon \cA \to \cA$ by $\psi(b) = a$ for $b \in E^R_{X_\tau}(a)$, and $\psi(c)$ as the empty word for $c \in \cA \setminus E^R_{X_\tau}(a)$.
An inductive argument shows that
\begin{equation*}
    \tau(c_0 c_1 \dots c_{\ell-1}) \, \psi(c_\ell) = 
        \psi(c_0)\, \sigma(c_0 c_1 \dots c_{\ell-1})
        \enspace \text{for all $c_0 c_1 \dots c_\ell \in \cL(\tau)$ with $\ell \geq 1$.}
\end{equation*}
It follows from this that $\sigma$ is primitive and generates the same subshift as $\tau$.

We now prove that $\alpha |\sigma^n(a)|$ stays bounded away from $\Z$.
By \eqref{eq:exs_nontrival_cobs:defi_sigma}, $\sigma^n(a) = \tau^n(a) \tau^{n-1}(a) \dots \tau(a) a$, and thus  $|\sigma^n(a)| = \sum_{0 \leq j \leq n} |\tau^n(a)|$.
With this, the decomposition $\alpha \vec{\mathbf 1} = \vec w + \vec v$ and the left-hand side of \eqref{eq:exs_nontrival_cobs:acc_conv}, we compute the following:
\[  \alpha |\sigma^n(a)| = 
    \sum_{0 \leq j \leq n} \alpha \vec{\mathbf 1} M_\tau^j \  \tr{\vec{\mathbf{1}}_a} = 
    \Big( \vec w + \vec v + \sum_{1\leq j\leq n}  \vec{v} M_\tau^j  \Big) \, \tr{\vec{\mathbf{1}}_a}
    \pmod{\Z}.
    \]
Therefore, by the right-hand side of \eqref{eq:exs_nontrival_cobs:acc_conv}, $\| \alpha |\sigma^n(a)| \| \geq \frac{1}{2}\|\vec w + \vec v\| > 0$ for every $n \geq 1$.

Finally, if $n \geq 0$ and $c_n$ is a coboundary in $X_\sigma$ satisfying Condition \eqref{eq:hip_cob:theo:EigCharac:FAR:Mult} of \Cref{theo:EigCharac:FAR:Mult}, then $\bigl\|\alpha |\sigma^m(a)| - c_n \tau_{n,m}(a)\bigr\|$ converges to $0$ as $m \to \infty$ by \Cref{rem:FAR_crit=>Host}, so $c_n \not\equiv 0$.
\end{proof}


\begin{example}
\label{ex:sturmian_avec_cobord}
The construction used in the proof of \Cref{prop:exs_nontrival_cobs} gives the following concrete example.
It is obtained from the Sturmian subshift of parameter $1 - \sqrt{2}$ (see Definition \ref{def:dendric}).

Let $\tau \colon \cA^* \to \cA^*$ be the primitive substitution defined on $\cA = \{0,1,2,3\}$  by
\[ \tau \colon \begin{cases}
        0 &\mapsto 0 2 \\
        1 &\mapsto 0 1 \\
        2 &\mapsto 0 1 3 \\
        3 &\mapsto 3 0 2.
    \end{cases}     \]
Let us  show  that $\alpha = 1 - \sqrt{2}$ is an eigenvalue of  the minimal subshift $X_\tau$ associated to a nontrivial coboundary, 
by verifying Item (3) of \Cref{theo:EigCharac:FAR:Mult}.
Let us  again put this condition in matricial form: it is sufficient to finding row vectors $\vec{c}, \vec{v}, \vec{w}$  and  a non-negative integer $n$ such that 
 $$\alpha \vec{\mathbf 1}  M_\tau^n=\alpha  \vec {h}_n= \vec{v} + \vec{w} + \vec{c} ,$$
  where $\vec{w}$ has integer entries, 
  $\vec{c}$  in the letter-coboundary vector space $\vec{\cC}_{X_\tau}$ (as  defined in  \Cref{subsec:discrepancy_subst}), and 
  $\sup\bigl\{\|\vec{v} \cdot M_\tau^m {\vec{u} }: u \in \cL(X_\tau)\| \bigr\}$ tends to 0 as $m \to \infty$ (where $\vec{u}$ stands for the vector whose entries count the number of occurrences of each  given letter in $u$).
Again  by using   for example the Dumont-Thomas decomposition \cite{Dumont-Thomas}, the later condition  is implied by    the condition  $\vec{v} \, M_\tau^m$ converges exponentially fast to the vector  $\vec{0}$ as $m \to \infty$, which 
  is equivalent to $\vec{v}$ lying in the span of all the eigenspaces of $M_\tau$ associated to eigenvalues of modulus strictly smaller than $1$.

Let us now  detail how to  verify these conditions.
The matrix of $\tau$ is
\[  M_\tau = \begin{bmatrix}
        1 & 1 & 1 & 1 \\
        0 & 1 & 1 & 0 \\
        1 & 0 & 0 & 1 \\
        0 & 0 & 1 & 1 
    \end{bmatrix}.   \]
Its left  action on row vectors is diagonalizable; a basis of  left row eigenvectors is given by
\begin{equation}
    \label{eq:ex:sturmian_avec_cobord:eigenvectors}
    \begin{cases}
    \vec{u}  &= (1 + \sqrt{2}, 1, \sqrt{2}, 1)   \\   
    \vec{u'}  &= (\sqrt{2}, -1, -2, \sqrt{2}-1)  \\
    \vec{z}  &= (0, 1, 0, -1)   \\  
    \vec{z'} &=(-1, 1, 1, 0)      
    \end{cases}
\end{equation}
with eigenvalues $\lambda = 1 + \sqrt{2}$, $\alpha = 1 - \sqrt{2}$, $1$ and $0$, respectively.
The only contracting direction is generated by the vector  $\vec{u'}$.

Next, we compute the letter-coboundaries.
One can check that $\cL(X_\tau) \cap \cA^2 = \{01,02,13,20,30\}$, so the extension graph $\Gamma_{X_\tau}(\varepsilon)$ has three connected components, whose vertices on the right-hand side of the graph are $\{0_R\}$, $\{1_R,2_R\}$ and $\{3_R\}$.
By \Cref{rem:constantK} and \Cref{thm:manifold}, we deduce that a morphism $c \colon \cA^* \to \R$ is a letter-coboundary in $X_\tau$ if and only if there exists $\xi_0,\xi_1 \in \R$ such that 
\[  c(0) = -\xi_0, \, 
    c(1) = \xi_1, \, 
    c(2) = \xi_0
    \enspace \text{and} \enspace
    c(3) = \xi_0-\xi_1.
    \]
Thus, the coboundary vector space $\vec{\cC}_{X_\tau}$ is spanned by the row vectors $(-1,0,1,1)$ and $(0,1,0,-1)$.
It is easy to check that $\vec{\cC}_{X_\tau}$  is  then exactly the span of the row eigenvectors $\vec{z}$ and $\vec{z'}$ defined in \Cref{eq:ex:sturmian_avec_cobord:eigenvectors}.
Moreover, the coboundaries appearing in \Cref{theo:EigCharac:FAR:Mult} can always be assumed to be of the form $c \circ \tau_{n,m}$ (by increasing $n$ if necessary), so, since $\vec{z'} \, M_\tau = 0$, we can narrow down the useful vectors in  $\vec{\cC}_{X_\tau}$ to those in the span of $\vec{z}$.
Consequently, the relevant decompositions for $\alpha \vec{h}_n$ have the form 
\[  \text{$\alpha \vec{h}_n = \xi \, \vec{z} + \xi' \, \vec{u'} + \vec{w}$, with $\xi, \xi' \in \R$ and $\vec{w} \in \Z^\cA$} .\] 
Furthermore, if such a decomposition exists, then there a decomposition  with $\xi$ and $\xi'$ belonging to the field $\Q(\alpha)$,   which can be  found computationally.
In our case we find  the following  solution with $n=0$:
\begin{align*}
    \alpha \, \vec{h}_0 =  \alpha\, \vec{\mathbf 1}=
    -\tfrac{1}{\sqrt{2}} \, \vec{z} +
    (-1 + \tfrac{1}{\sqrt 2}) \, \vec{u'} + 
    (0,0,-1,-1)
\end{align*}
yields that  the coboundary $c \colon \cA^* \to \R$ given by ${\vec c}=-\tfrac{1}{\sqrt{2}}$
satisfies that $\| \alpha h_n(a) - c \circ \tau^n(a) \|$
converges to $0$ exponentially fast; therefore, $\alpha$ is an eigenvalue of $X_\tau$.
Remark that $c \circ \tau = c$, which is the condition found in Host's criterion \cite{Host86} (c.f. \Cref{existence_limit_coboundary}).
This, together with $c(3) = 1/\sqrt{2} \neq 0$, gives
\[  \lim_{n\to\infty} \| \alpha h_n(3) \| = \|1/\sqrt{2}\| > 0, \]
which shows that {\em any}  coboundary satisfying Condition \eqref{eq:hip_cob:theo:EigCharac:FAR:Mult} of \Cref{theo:EigCharac:FAR:Mult} for the eigenvalue  $\alpha$ must be non-trivial.

We remark that $X_\tau$ is conjugate to the Sturmian subshift of parameter $\sqrt{2}-1$. This is in accordance   with  \Cref{balanced=>spaces_decomposition} from which we deduce that
$X_{\tau}$ is letter-balanced.
\end{example}

\begin{remark}
    \label{rem:notpreservedconnecetdgraph}
    \Cref{prop:exs_nontrival_cobs} and the previous example shows that the number of connected components of the extension graph of the empty word is not preserved under conjugacy. 
    Indeed, $\Gamma_{X_\tau}(\varepsilon)$ is not connected since $X_\tau$ admits nontrivial letter-coboundaries, while the Sturmian subshift of parameter $1-\sqrt{2}$, to which $X_\tau$ is conjugate, has only connected extension graph \cite[Example 3.2]{dendric}.
    Thus, it is possible to have two conjugate systems, one   with a connected extension  graph associated to the empty word, and the other one with a disconnected one. This is also the case of the substitution from Example \ref{ex:extension_graph_CM}
    which  generates  a shift which is conjugate to  a Sturmian shift (see \cite{BCB19} for more details). 
\end{remark}

\subsection{On eigenvalues and  measures}  \label{subsec:Itau}


In this section, we give  a natural example of  a subshift where $I(\btau)$ (as defined in \Cref{eq:Itau}) is strictly contained in $I(X_{\btau},S)$  (see also Remark \ref{rem:measures}).
 

\begin{example}\label{ex:TM:freqs&eigs}

    Let $\cA = \{0,1\}$ and $\tau \colon \cA^* \to \cA^*$ be the Thue--Morse substitution, that is,
    \[   \tau \colon \begin{cases} 
            0 & \mapsto 01 \\
            1 & \mapsto 10
        \end{cases}     \]
    It is well-known \cite{Dekking1978} (see also  below)  that
    \[  E(X_\tau,S) = \{ k/2^n : k \in \Z,\, n \geq 0\}.  \]
    It is also known \cite{Dekking1992} that if $\mu$ is the unique invariant measure of $(X_\tau,S)$, then the frequency $\mu([v])$ of a factor $v$ with length $n\geq 2$ is either $\frac{1}{6}2^{-m}$ or $\frac{1}{3}2^{-m}$, where $m$ is such that $2^m<n\leq 2^{m+1}$, while $\mu([0])=\mu([1])=1/2$. This implies that
    \[  I(X_\tau, S) = 
        \Big\{ \frac{k}{3\cdot 2^n} : k \in \Z, \, n \geq 0 \Big\}.   \]
    Therefore, $E(X_\tau,S) \subsetneq I(X_\tau,S)$ since frequencies of words of length at least $2$ do not belong to $E(X_\tau,S)$. Note that thanks to  Gottschalk-Hedlund's Theorem (see \Cref{theo:GH}), this implies that $(X_{\tau},S)$ is not balanced on factors of length at least  than  or equal to  $2$ (see also \cite[Corollary 4.10]{BCB19}).
       
    On the other hand,  the measure of $B_n(a) $ is known to be $\mu(B_n(a)) = 1/2^{n+1}$; thus, with the notation from (\ref{eq:Itau}), 
    \[  I(\btau) = \{ k/2^n : k \in \Z,\, n \geq 0 \}, \]
    with  $\btau$    being  the constant directive sequence $\btau = (\tau : n \geq 0)$. 
    In particular, we recover that 
    \[  E(X_\tau,S) = I(\btau) \subsetneq I(X_\tau,S). \]
    This comes from   \Cref{cor:EigCharac:FAR:eigs&freqs}, by noticing that the substitution $\tau$ being primitive implies that $\btau$ is (strongly) primitive and recognizable (see  also \Cref{rem:recog_substitution}).
    We remark that the factor $1/3$ in $I(X_\tau,S)$ appears in measures of cylinders of words of length 2, such as $[00]$, which has measure $1/6$.
    Similarly, $\mu(\{\tau^n(x) \in X_\tau : x_0 = x_1 = 0\}) = 1/3\cdot 2^{n+1}$.  
    These sets are not required to compute $I(\btau)$.
\end{example}

\subsection{Non-trivial  coboundary  without eigenvalues}\label{subsec:nontrivial}

As  proved in \Cref{thm:proper+unimod=>RW_gen_Zd}, the language of a minimal subshift generated by a proper unimodular directive sequence admits only trivial coboundaries. 
Non-trivial coboundaries can appear in minimal subshifts without having any impact on eigenvalues, even in the case of non-proper directive sequences admitting  non-trivial eigenvalues, as the following example illustrates. 
This example has been constructed by J. Cassaigne and M. Minervino in order to provide an example of a subshift that is balanced on factors and that admits  $1$ as an eigenvalue of its incidence matrix, as explained in \cite[Example 2.7]{BCB19} (see also \Cref{subsec:discrepancy_subst}).

\begin{example}
    \label{ex:CM}

    We consider the substitution
    $\sigma$ from Example \ref{ex:extension_graph_CM},
    together with the   constant directive sequence $\bsigma= (\sigma : n \geq 0)$.
    
    We recall that its extension graph $\Gamma_{X_{\bsigma}}(\varepsilon)$ has two connected components.
    By \Cref{thm:manifold} (see \Cref{rem:constantK}),  a morphism $c \colon \cA^* \to \R$ is a letter-coboundary in $X_\sigma$ if and only if there exists $\xi \in \R$ such that
    \[  c(1) = -c(0) = \xi
        \enspace \text{and} \enspace
        c(2) = 0.   \]
    In particular, $X_{\bsigma}$ admits non-trivial  letter-coboundaries.

    Let us check that the group of eigenvalues $E(X_{\bsigma},S)$ of $X_{\bsigma}$ is $\Z +  \frac{1}{2}(3 - \sqrt{5}) \Z$,  and moreover that any letter-coboundary associated to $ \frac{1}{2}(3 - \sqrt{5})$ is trivial.
    In other words, the non-trivial  letter-coboundaries of  $X_{\bsigma}$   have no effect on the spectrum of the system.

    The  incidence matrix of  the substitution $\sigma$ is 
    \[  M_\sigma = \begin{bmatrix}
            2 & 0 & 1 \\
            1 & 1 & 1 \\
            0 & 1 & 1
        \end{bmatrix}  . \]
The substitution $\sigma$  is primitive and its determinant is $1$. In particular $X_{\bsigma}$ is minimal and  recognizable  (see \Cref{rem:recog_substitution}), which allows us to use \Cref{cor:EigCharac:FAR:eigs&freqs}.
    Moreover primitive substitutive subshifts are uniquely ergodic, so $X_{\bsigma}$ supports a unique invariant measure $\mu$.
    We deduce that $$E(X_{\bsigma}, S) \subseteq 
   I(\bsigma)=  \bigcup_{n \geq 0} \big\{ \sum_{a\in\cA} w_a \, \mu(B_n(a)) : (w_a)_{a\in\cA} \in \Z^{\cA} \big\}. 
  $$
    Furthermore, one can compute $\mu(B_n(a))$ as follows.
    Let $\vec u = (u_a)_{a\in\cA} \in \R^{\cA}$ be the Perron-Frobenius eigenvector of $M_\sigma$, normalized so that $\sum_{a\in\cA} u_a = 1$, and $\lambda$ be its Perron-Frobenius eigenvector; then, 
    \[  \mu(B_n(a)) = \frac{1}{\lambda^n} u_a
        \enspace \text{for all $n \geq 0$ and $a \in \cA$.}   \]
    This formula is classical; it is also a consequence of \Cref{prop:measure_transfer}.
    One has 
    \[  \lambda = \frac{1}{2}\big(3 + \sqrt{5}\big)
        \enspace \text{and} \enspace
        \vec u = \bigl( \tfrac{1 + \sqrt{5}}{2}, \tfrac{1 + \sqrt{5}}{2}, 1 \bigr).
        \]
    Note that we  recover the  result of   Lemma \ref{lem:cob_zero_integral}.
    Routinary computations then show that 
    \[  I(\bsigma) = \{k + \ell \bar{\lambda} : k,\ell \in \Z\} \eqqcolon \Z + \bar{\lambda} \Z,  \] 
    where $\bar{\lambda} \coloneqq 1/\lambda = \frac{1}{2}(3 - \sqrt{5})$ is the Galois conjugate of $\lambda$.
    Thus, $E(X_{\bsigma},S) \subseteq \Z + \bar{\lambda} \Z$.
    Let us show that, in fact, $E(X_{\bsigma},S) = I(\bsigma)$.
   Since $E(X_{\bsigma},S)$  is an additive group containing $1$, it is enough to show that $\bar{\lambda} \in E(X_{\bsigma},S)$.

    For that purpose, we will apply \Cref{theo:EigCharac:decisive:positive}. Decisiveness comes from Lemma \ref{lemma:decisive:same_first_letter_in_each_CC}.
    One also checks that every word of  length $2$  in $\cL(X_{\bsigma})$ occurs in $\sigma^3(a)$ for every  letter $a$ (see Figure \ref{fig:extensiongraph}). 
    Let us show  that $$\sum _n \max\{ \| \alpha h_n(v) \| : a \in \cA,\, v \in \prefixes(\sigma^6(a))\}$$  is finite.
    To this end, it is enough to prove that $\| \alpha h_n(a) \|$ converges exponentially fast to $0$ as $n$ tends to infinity, for every $a \in \cA$.
    
    As before, we restate the problem in  terms of linear algebra.
    We recall that  $\vec{\mathbf 1}$ stands for  the row vector $(1,1,1) \in \Z^{\cA}$.
    Then, the height vector $\vec h_n = (h_n(a) : a \in \cA)$ equals $\vec{\mathbf 1} \, M_\sigma^n$.
    So, we look for a decomposition of the form  $$\bar{\lambda} \, \vec h_n = \vec v_n + \vec w_n,$$ with $\vec w_n \in \Z^\cA$ and $\vec v_n$ converging exponentially fast to $0$.
    We  use the decomposition
    \[  \bar{\lambda} \, \vec{\mathbf 1} = 
        (\bar{\lambda}, \bar{\lambda}, \bar{\lambda}-1) +
        (0,0,1).
        \]
    Let $\vec v = (\bar{\lambda}, \bar{\lambda}, \bar{\lambda}-1)$ and $\vec w = (0,0,1)$.
    Then, $\vec v$ is a left-eigenvector of $M_\sigma$ of eigenvalue $\bar{\lambda}$. This implies
    \[  \bar{\lambda} \, \vec{\mathbf 1} \, M^n_\sigma = 
        \bar{\lambda}^n \, (\bar{\lambda}, \bar{\lambda}, \bar{\lambda}-1) +
        (0,0,1) M_\sigma^n.
        \]
    Since $|\bar{\lambda}| < 1$, the first term on the right-hand side converges exponentially fast to $0$ as $n \to \infty$, and the second term has integer entries.
    Thus, we can apply \Cref{theo:EigCharac:decisive:positive} to the trivial coboundary $c \equiv 0$, showing that $\bar{\lambda}$ is an eigenvalue of $X_\sigma$.

\end{example}

\subsection{Non-proper directive sequences with trivial coboundaries}
\label{subsec:nonproper}
As stressed in  Theorem \ref{theo:EigCharac:proper}, if a directive sequence is left-proper, then the sequence of coboundaries associated to any eigenvalue can be chosen to be trivial. 
Although any minimal subshift has a proper $S$-adic representation \cite{GPS92,DDMP}), properness is not a property  easily obtained by simple manipulations of a given directive sequence (see  the proof of \cite[Corollary 1.4]{espinoza22}), and in many cases the most natural directive sequence generating a minimal $S$-adic subshift is not proper.
However, as the following examples show, there are cases were all letter-coboundaries associated to a directive sequence are trivial, even if the sequence is not proper.

\begin{example}\label{ex:fabien}
    Consider the substitutions $\sigma$ and $\tau$ defined on $\cA=\{a,b,c\}$ by
    \[  \sigma \colon \begin{cases}
            a & \mapsto acb \\
            b & \mapsto bab \\
            c & \mapsto cbc
        \end{cases}
        \enspace \text{and} \enspace
        \tau \colon \begin{cases}
            a & \mapsto abc \\
            b & \mapsto acb \\
            c & \mapsto aac
        \end{cases}
        \]
    Let $\bsigma$ be the directive sequence consisting of runs of $\sigma$ that grow by one each time, with a single $\tau$ after each run; that is,
    \[  \bsigma = \sigma, \tau, \ \sigma,\sigma,\tau,\ \sigma,\sigma,\sigma,\tau, \dots     \]
    This directive sequence, introduced by F. Durand in  \cite{durand}, is primitive and recognizable, but neither left- nor right-proper. 
    As it is shown in \cite[Example 3.3]{BCBY}, for all $n \geq 0$, the extension graph $\Gamma_{X_{\bsigma}^{(n)}}(\varepsilon)$ of the empty word in the language of the level-$n$ subshift $X_{\bsigma}^{(n)}$ is connected. 
    Thus, any sequence of coboundaries associated to $\bsigma$ is trivial. 
    We remark that the eigenvalues of $X_{\btau}$ are described in \cite[Section 6]{BCBY}: they are $p/3^n$ for $n \geq 0$ and $p \in \Z$.
\end{example}

\begin{example}\label{ex:modifiedAR}
Consider the substitutions $\tau_0, \tau_1,\tau_2 \colon \cA^* \to \cA^*$ over on $\cA=\{a,b,c\}$ defined by
\[ \tau_0 \colon \begin{cases}
    0 & \mapsto 0 \\
    1 & \mapsto 10  \\
    2 & \mapsto 02
    \end{cases}
    \quad
    \tau_1 \colon \begin{cases}
    0 & \mapsto 10 \\
    1 & \mapsto 1  \\
    2 & \mapsto 21
    \end{cases}
    \quad
    \tau_2 \colon \begin{cases}
    0 & \mapsto 02 \\
    1 & \mapsto 21  \\
    2 & \mapsto 2
    \end{cases}
    \]

These substitutions correspond to modifications of the classical Arnoux-Rauzy substitutions $\sigma_0$, $\sigma_1$ and $\sigma_2$ (see \cite{ar91} for details), where the order of letters in some of the images has been changed. 
For any letter $a \in \{0,1,2\}$, the incidence matrix $M_{\tau_a}$ is the same as that of $\sigma_a$: $M_{\sigma_a}$.
Thus, a unimodular matrix, but none of the $\tau_a$ is left- nor right-proper.

Let $\bs = (\sigma_n)_{n\geq 0}$ be any directive sequence such that $\sigma_n \in \{\sigma_0,\sigma_1,\sigma_2\}$ for all $n$ and each of the $\sigma_a$ appears infinitely often. 
Such a directive sequence $\bs$ is called an Arnoux-Rauzy directive sequence, and the subshift generated by $\bs$ is an Arnoux-Rauzy subshift on a three-letter alphabet.
This subshift is  uniquely ergodic, see for instance \cite[Theorem 4.10]{unimodular}. 
According to \cite[Theorem 5.7]{BD14}, this implies that the cone
$$ C_{\bs}^{(0)} = \bigcap_{n\geq 1}M_{0,n}^{\bs} \R_+^3$$
is one-dimensional, where $M^{\bs}_{0,n}$ stands for the incidence matrix of the composition $\sigma_{0,n}$.

Take now a directive sequence $\bt=(\tau_n)_{n\geq 0} \in \{\tau_0,\tau_1,\tau_2\}^\N$, with each $\tau_a$ appearing infinitely often. 
Then, $\bt$ is primitive and recognizable, since the matrices $M_{\tau_i}$ are unimodular (see for instance \cite[Theorem 4.6]{BSTY}). 
It is clear that there exists an Arnoux-Rauzy directive sequence $\bs$ whose incidence matrices coincide with those of $\bt$, and then the cone
$$ C_{\bt}^{(0)} = \bigcap_{n\geq 1} M^{\bt}_{0,n}\R_+^3, $$
where $M^{\bt}_{0,n}$ is the incidence matrix of $\tau_{0,n}$, is exactly the same as $C_{\bs}^{(0)}$, and thus one-dimensional as well. 
Moreover, by the unimodularity, the subshift generated by $\bt$ is also uniquely ergodic, and $\mu_{\btau}([a]) = \mu_{\bsigma}([a])$ for every $a \in \cA$, with $\mu_{\btau}$ and $\mu_{\bsigma}$ the unique invariant measures of $X_{\btau}$ and $X_{\bsigma}$, respectively.
Arnoux-Rauzy subshifts are known to have rationally independent letter frequencies (see \cite[Theorem 2]{andrieu}), so one deduces that the collection $\{\mu_{\btau}([a]) : a \in \cA\}$ is rationally independent. 
We conclude, using \Cref{prop:dim}. that $X_{\bt}$ admits only trivial letter-coboundaries.

Let $n\geq 1$. 
Note that the subshift $X_{\bt}^{(n)}$ is generated by the shifted directive sequence $(\tau_k : k\geq n)$, whose incidence matrices are the same as the those of $(\sigma_k : k \geq n)$. 
Since the latter is also an Arnoux-Rauzy directive sequence, the argument carried out above shows that $X_{\bt}^{(n)}$ is uniquely ergodic and has rationally independent letter frequencies. 
By \Cref{prop:dim}, we conclude that $X_{\bt}^{(n)}$ has only trivial coboundaries. 
Thus, any sequence of coboundaries associated to $\bt$ is trivial.
\end{example}

\subsection{Eigenvalues of constant-length directive sequences}
\label{subsec:constantlength}
It follows from \Cref{cor:EigCharac:FAR:eigs&freqs} that, for a minimal subshift $X$ generated by a recognizable and primitive directive sequence of \emph{alphabet rank $r$}, \emph{i.e.}, $ \liminf |\cA_n |= r$, the rational dimension of its eigenvalues group is at most $r-1$.
In \Cref{prop:cnst_len&rat_spectrum} below, we consider the situation of constant-length directive sequences of finite alphabet rank, and show using \Cref{theo:EigCharac:Mult} that only rational eigenvalues are possible.
This was proved in \cite{BMY} by very different methods.
It can also be seen as an analogue of \cite{durand_frank_maass_toeplitz}, in which it was proved that Toeplitz subshifts of finite Toeplitz rank have only rational {\em measurable} eigenvalues.
In fact, our proof closely follows the one in \cite{durand_frank_maass_toeplitz}.

A directive sequence  $\btau$ is said to be of \emph{constant length}  if  each $\tau_n$  has  constant length, that is,
there exists $\ell_n$ such that
$|\tau_n(a)|=\ell_n$, for all letter $a \in \cA_{n+1}.$ We then set 
$|\tau_n|:=\ell_n$, and thus $|\tau_{0,n}|=\ell_0 \cdots \ell_{n-1}$.

\begin{proposition}
\label{prop:cnst_len&rat_spectrum}
Let $X$ be a subshift generated by a primitive and recognizable directive sequence $\btau = (\tau_n \colon \cA_{n+1}^* \to \cA_n^* : n \geq 0)$.
Assume that $\btau$ is of constant length  and has alphabet rank equal to $r$.
Then, there exists an integer $1 \leq q \leq r$, coprime with $|\tau_{0,n}|$ for every $n \geq 1$, such that 
\[  E(X,S) = \Big\{ \frac{p}{q|\tau_{0,n}|} : p \in \Z, \, n \geq 1 \Big\}.  \]
\end{proposition}

\begin{proof}
Set $q_n = |\tau_{0,n}|$ for $n \geq 1$.
By contracting $\btau$, we may assume without loss of generality that $r = \sup\bigl\{|\cA_n| : n \geq 1\bigr\}$.

First, it is folklore that $p/q_n$ is an eigenvalue of $(X,S)$ for every $n \ge 1$ and $p \in \Z$.
We include a brief proof for the readers' convenience.

Since $\btau$ is recognizable, the collection $\{S^k B_n(a) : a \in \cA_n, \, 0 \leq k \lt |\tau_n(a)| \}$ is a clopen partition of $X$.
So, since $\btau$ is of constant length, the set $B_n = \bigcup_{a \in \cA_n} B_n(a)$ yields the clopen partition $\cP_n = \{ S^k B_n : 0 \leq k \lt q_n \}$ of $X$, which is cyclic in the sense that $S^{q_n} B_n = B_n$.
This allows us to define a continuous map $g_n \colon X \to \R/\Z$ by $g(x) = kp/q_n$, where $0 \leq k \lt q_n$ is the  unique integer such that $x \in S^k B_n$.
Then, $g(S x) = g(x) + p {\pmod{\Z}}$ for every $x \in X$, and $g_n$ is an eigenfunction of $X$ with eigenvalue $p/q_n$.

Next, we show that every eigenvalue of $X$ has the form $\frac{p}{q\,q_{0,n}}$ for some $n \geq 1$, $p \in \Z$ and $1 \leq q \leq r$.
Assume that $\alpha \in \R$ is an eigenvalue of $X$.
Using \Cref{theo:EigCharac:Mult}, we obtain a family of real numbers $\brho = (\rho_n(a) : n \geq 0,\, a \in \cA_n)$ such that $\varepsilon_n(\brho)$ (see \Cref{defi_eps_n})
converges to $0$ as $n$ tends to infinity.
Now, since $\btau$ is of constant length, $h_n(v) = q_n |v|$ for all $v \in \cA_n^*$, hence
\[  \big \| \rho_n(v_0) - \rho_n(v_k) + kq_n\alpha \big\| \leq \varepsilon_n (\brho) \]
for all $v_0 \cdots v_k \in \cL(X_{\btau}^{(n)})$.
This implies that, for every $n \ge 0$ and $a_n \in \cA_n$ fixed, each element of the orbit $\cO_n \coloneqq \{ kq_n \alpha  \pmod{\Z}: k \in \Z\}$ lies within distance $\varepsilon_n(\brho)$ from the set $\Gamma_n \coloneqq \{\rho_n(b) - \rho_n(a_n) : b \in \cA_n\}$. Choose $n \geq 1$ large enough that $1/r - \varepsilon_n > 1/(r+1)$.
As $\Gamma_n$ has at most $r$ elements, the complement $(\R/\Z) \setminus \cO_n$ contains a non-empty interval $I$ of length at least $1/r - \varepsilon_n$, which implies  that $q_n \alpha$ is rational,  since otherwise   this would contradict   the  equidistribution of  the  sequence $(k q_n\alpha)_k$. 
Write $q_n \alpha = p/q$ with $p \in \Z$, $q \ge 1$, and $\gcd(p,q) = 1$.
We then have $\cO_n = \{ k/q : k \in \Z \} \pmod{\Z}$, so the largest interval contained in $(\R/\Z) \setminus \cO_n$ has length $1/q$.
It follows that $1/q \geq |I| > 1/(r+1)$, and therefore $q \leq r$.
We have proved that every eigenvalue $\alpha$ of $X$ has the form $\frac{p}{q\,q_n}$, with $q \leq r$.
Since $E(X,S)$ is an additive group, the conclusion of the proposition follows.
\end{proof}

We remark that in \Cref{prop:cnst_len&rat_spectrum} the finite alphabet rank hypothesis cannot be avoided. 
Indeed, consider a minimal subshift $X$ with infinitely many rational eigenvalues and at least one irrational eigenvalue (for instance, a symbolic extension of a product of an odometer and an irrational rotation). 
By \Cref{prop:cnst_len:sadic_repr} below, $X$ is generated by a primitive and recognizable directive sequence of constant length. 
Thus, the hypotheses of \Cref{prop:cnst_len&rat_spectrum} are satisfied, except for the finite alphabet rank condition, yet the conclusion cannot hold, since $X$ has irrational eigenvalues.

\begin{proposition}
    \label{prop:cnst_len:sadic_repr}
    Let $X$ be a minimal subshift with infinitely many rational eigenvalues.
    Then, $X$ is generated by a primitive, recognizable directive sequence of constant length.
\end{proposition}
\begin{proof}
The hypothesis allows us to choose a sequence of positive integers $(q_n)_{n\ge 0}$ such that $1 < q_n < q_{n+1}$, $q_n \mid q_{n+1}$, and $1/q_n$ is an eigenvalue of $X \subseteq \cA^\Z$ for every $n \ge 0$.
Then,  similarly as  in the proof of \Cref{prop:cnst_len&rat_spectrum},  there exist clopen sets $B_n \subseteq X$ such that $\{S^k B_n : 0 \le k < q_n\}$ is a partition of $X$ and $S^{q_n}(B_n)=B_n$.
We may assume that $B_{n+1} \subseteq B_n$.

Let $\cR_n$ denote the set of return words to $B_n$ in $X$, that is, $\cR_n$ consists of all words $x_{[i,j)}$ with $x \in X$ and $i<j$ such that $S^i x \in B_n$, $S^j x \in B_n$, and $S^k x \notin B_n$ for every $i<k<j$.
Note that $|u|=q_n$ for every $u \in \cR_n$.
Fix any bijection $\sigma_n \colon \cA_n \to \cR_n$ from a finite alphabet $\cA_n$ onto $\cR_n$.
This map extends uniquely to a substitution $\sigma_n \colon \cA_n^* \to \cA^*$.

Since $B_{n+1} \subseteq B_n$, for every $a \in \cA_{n+1}$ the word $\sigma_{n+1}(a)$ is uniquely expressed as a concatenation of words of the form $\sigma_n(b)$, with $b \in \cA_n$.
This defines a map $\tau_n \colon \cA_{n+1} \to \cA_n^*$ such that $\sigma_{n+1}(a)=\sigma_n(\tau_n(a))$.
Since $|\sigma_k(b)|=q_k$ for every $k \ge 0$ and $b \in \cA_k$, we have $|\tau_n(a)|=q_{n+1}/q_n$ for every $n \ge 0$ and $a \in \cA_{n+1}$.

Set $\btau=(\tau_n)_{n\ge 0}$.
Then $\btau$ is a constant-length directive sequence.
It is not difficult to see, from the fact that the $\sigma_n$ are defined via return times to a clopen set, that $\btau$ is recognizable.
Now, $\btau$ is everywhere growing (since $q_n \to \infty$ as $n \to \infty$), so $X_{\btau}$ is non-empty.
Furthermore, $\tau_{0,n}(a)=\sigma_{n-1}(a) \in \cL(X)$ for every $a \in \cA_n$ and $n \ge 1$, hence $X_{\btau} \subseteq X$.
By the minimality of $X$, it follows that $X_{\btau}=X$.
Since $X$ is minimal and $\btau$ is recognizable, $\btau$ must be primitive.
\end{proof}

\printbibliography

\end{document}